\documentclass[reqno]{amsart}
\usepackage{latexsym,amsfonts,amssymb,epsfig}
\usepackage{hyperref} 
\usepackage{color}

\DeclareMathOperator{\ch}{char}
\DeclareMathOperator{\gr}{gr}

\DeclareMathOperator{\ad}{ad}
\DeclareMathOperator{\Der}{Der}

\DeclareMathOperator{\Lie}{Lie}
\DeclareMathOperator{\Alg}{Alg}
\DeclareMathOperator{\Jord}{Jord}
\DeclareMathOperator{\Poisson}{Poisson}

\DeclareMathOperator{\End}{End}

\DeclareMathOperator{\wt}{wt}    
\DeclareMathOperator{\swt}{swt}  
\DeclareMathOperator{\jwt}{jwt}  
\DeclareMathOperator{\sswt}{\overline {swt}}  
\DeclareMathOperator{\Wt}{Wt}    
\DeclareMathOperator{\WtR}{WtR}    
\DeclareMathOperator{\Gr}{Gr}     
\DeclareMathOperator{\mult}{mult}     

\DeclareMathOperator{\wWtR}{WtR^\sharp}    
\DeclareMathOperator{\wGr}{Gr^\sharp}     
\DeclareMathOperator{\wdeg}{deg^\sharp}     

\renewcommand {\limsup}{\operatorname* {\overline{lim}}}
\renewcommand {\liminf}{\operatorname* {\underline{lim}}}
\renewcommand{\Re}{\mathop{\mathrm {Re}}}
\renewcommand{\Im}{\mathop{\mathrm {Im}}}
\DeclareMathOperator{\GKdim}{GKdim}
\DeclareMathOperator{\LGKdim}{\underline{GKdim}}
\newcommand\dd{\partial}

\renewcommand{\a}{\alpha}
\renewcommand{\b}{\beta}

\newcommand{\Z}{\mathbb Z}            
\newcommand{\R}{\mathbb R}            
\newcommand{\Q}{\mathbb Q}            
\newcommand{\N}{\mathbb N}            
\newcommand{\F}{\mathbb F}            
\newcommand{\C}{\mathbb C}            
\newcommand\NO{\mathbb N_0}           
\renewcommand{\H}{\mathcal H}         

\newcommand{\LL}{\mathbf L}        
\newcommand{\QQ}{\mathbf Q}        
\newcommand{\RR}{\mathbf R}        
\newcommand{\PP}{\mathbf P}        
\newcommand{\JJ}{\mathbf J}        
\newcommand{\KK}{\mathbf K}        

\renewcommand{\AA}{\mathbf A}      
\newcommand{\HH}{\mathbf H}        
\newcommand{\uu}{\mathbf u}         
\newcommand{\WW}{\mathbf W}         
\newcommand{\Kan}{\mathcal Kan}      
\newcommand{\Jor}{\mathcal Jor}      
\newtheorem{Theorem}{Theorem}[section]
\newtheorem{Corollary}[Theorem]{Corollary}
\newtheorem{Lemma}[Theorem]{Lemma}
\theoremstyle{remark}
\newtheorem{Remark}{Remark}
\theoremstyle{Example}

\theoremstyle{Conjecture}

\renewcommand{\theenumi}{\roman{enumi}}   
\begin{document}
\title{Fractal nil graded Lie, associative, Poisson, and Jordan superalgebras}
\author{Victor Petrogradsky}
\address{Department of Mathematics, University of Brasilia, 70910-900 Brasilia DF, Brazil}
\email{petrogradsky@rambler.ru}
\thanks{The first author was partially supported by grants CNPq~309542/2016-2, FAPESP~2016/18068-9}
\author{I.P.~Shestakov}
\address{Instituto de Mathem\'atica e Estat\'istica,
Universidade de Sa\~o Paulo, Caixa postal 66281, 05315-970, Sa\~o Paulo, Brazil}
\email{shestak@ime.usp.br}
\thanks{The second author was partially supported by grants FAPESP 2014/09310-5, CNPq 2014/09310-5
}
\subjclass[2000]{
16P90, 
16N40, 
16S32, 
17B50, 
17B65, 
17B66, 
17B70, 
17A70, 
17B63, 
17C10, 
17C50, 
17C70  
}
\keywords{restricted Lie algebras, p-groups, growth, self-similar algebras, nil-algebras, graded algebras,
Lie superalgebra, Lie algebras of differential operators, Poisson superalgebras, Jordan superalgebras, algebraic quantization}

\begin{abstract}
The Grigorchuk and Gupta-Sidki groups play fundamental role in modern group theory.
They are natural examples of self-similar finitely generated periodic groups.
The first author constructed their analogue in case of restricted Lie algebras of characteristic 2~\cite{Pe06},
Shestakov and Zelmanov extended this construction to an arbitrary positive characteristic~\cite{ShZe08}.
Thus, we have examples of finitely generated restricted Lie algebras with a nil $p$-mapping.
In characteristic  zero, similar examples of Lie and Jordan algebras do not exist
by results of Martinez and Zelmanov~\cite{MaZe99} and~\cite{Zelmanov}.
The first author constructed analogues of the Grigorchuk and Gupta-Sidki groups
in the world of Lie superalgebras of arbitrary characteristic,
the virtue of that construction is that Lie superalgebras have clear monomial bases~\cite{Pe16},
they have slow polynomial growth.
As an analogue of periodicity, $\mathbb{Z}_2$-homogeneous elements  are $\ad$-nilpotent.
A recent example of a Lie superalgebra is of linear growth, of finite width 4,
just infinite but not hereditary just infinite~\cite{PeOtto}.
By that examples, an extension of the result of Martinez and Zelmanov~\cite{MaZe99}
for Lie superalgebras of characteristic zero is not valid.
\par

Now, we construct a  just infinite fractal 3-generated Lie superalgebra ${\mathbf Q}$ over arbitrary field,
which gives rise to an associative hull ${\mathbf A}$, a Poisson superalgebra ${\mathbf P}$, and
two Jordan superalgebras ${\mathbf J}$ and ${\mathbf K}$, the latter being a factor algebra of ${\mathbf J}$.
In case $\mathrm{char}\, K\ne 2$, ${\mathbf A}$ has a natural filtration, which associated graded algebra has
a structure of a Poisson superalgebra such that $\mathrm{gr}\, {\mathbf A}\cong{\mathbf P}$,
also ${\mathbf P}$ admits an algebraic quantization using a deformed superalgebra ${\mathbf A}^{(t)}$.
The Lie superalgebra ${\mathbf Q}$ is finely $\mathbb{Z}^3$-graded by multidegree in the generators,
${\mathbf A}$, ${\mathbf P}$ are also $\mathbb{Z}^3$-graded, while ${\mathbf J}$ and ${\mathbf K}$ are
$\mathbb{Z}^4$-graded by multidegree in four generators.
By virtue of our construction, these five superalgebras have clear monomial bases and slow polynomial growth.
We describe multihomogeneous coordinates of bases of ${\mathbf Q}$, ${\mathbf A}$, ${\mathbf P}$ in space as
bounded by "almost cubic paraboloids".
We determine a similar hypersurface in $\mathbb R^4$ that bounds monomials of ${\mathbf J}$ and ${\mathbf K}$.
Constructions of the paper can be applied to Lie (super)algebras studied before to
obtain Poisson and Jordan superalgebras as well.
\par

The algebras ${\mathbf Q}$, ${\mathbf A}$, and the algebras without unit ${\mathbf P}^o$, ${\mathbf J}^o$, ${\mathbf K}^o$
are direct sums of two locally nilpotent subalgebras and
there are continuum such decompositions.
Also, ${\mathbf Q}={\mathbf Q}_{\bar 0}\oplus {\mathbf Q}_{\bar 1}$ is a nil graded Lie superalgebra,
so, ${\mathbf Q}$ again shows that an extension of the result of Martinez and Zelmanov
for Lie superalgebras of characteristic zero is not valid.
In case $\mathrm{char}\, K=2$, ${\mathbf Q}$ has a structure of a restricted Lie algebra with a nil $p$-mapping.
The Jordan superalgebra ${\mathbf K}$ is nil finely $\mathbb{Z}^4$-graded,
in contrast with non-existence of such examples (roughly speaking, analogues of the Grigorchuk group)
of Jordan algebras in characteristic zero~\cite{Zelmanov}.
Also, ${\mathbf K}$ is of slow polynomial growth, just infinite, but not hereditary just infinite.

We call the superalgebras  ${\mathbf Q}$, ${\mathbf A}$, ${\mathbf P}$, ${\mathbf J}$, ${\mathbf K}$
fractal because they contain infinitely many copies of themselves.
\end{abstract}
\maketitle

In Section~1 we survey known results,
Section~2 supplies basic definitions.
In Section~3 we briefly describe constructions
and formulate main properties of our five  main objects:
a Lie superalgebra $\QQ$, its associative hull $\AA$,
a related Poisson superalgebra $\PP$, and two Jordan superalgebras $\JJ$, $\KK$.
The present research is a continuation of a series of papers on fractal (self-similar) (restricted) Lie (super)algebras,
the main feature is that we extend the results to the classes of Poisson and Jordan superalgebras.

\section{Introduction: Self-similar groups and algebras}

\subsection{Golod-Shafarevich algebras and groups}
The General Burnside Problem puts the question whether a finitely generated periodic group is finite.
The first negative answer was given by Golod and Shafarevich,
who proved that, for each prime $p$, there exists a finitely generated infinite $p$-group~\cite{Golod64}.
The construction is based on a famous construction of a family of finitely generated
infinite dimensional associative nil-algebras~\cite{Golod64}.
This construction also yields examples of infinite dimensional finitely generated Lie algebras $L$
such that $(\ad x)^{n(x,y)}(y)=0$, for all $x,y\in L$, the field being arbitrary~\cite{Golod69}.
The field being of positive characteristic $p$,
one obtains an infinite dimensional finitely generated restricted Lie algebra $L$ such that the $p$-mapping is nil, namely,
$x^{[p^{n(x)}]}=0$, for all $x\in L$.
This gives a negative answer to a question of Jacobson whether
a finitely generated restricted Lie algebra $L$ is finite dimensional provided that
each element $x\in L$ is algebraic, i.e. satisfies some $p$-polynomial $f_{p,x}(x)=0$
(\cite[Ch.~5, ex.~17]{JacLie}). 
It is known that the construction of Golod yields associative nil-algebras of exponential growth.
Using specially chosen relations, Lenagan and Smoktunowicz
constructed associative nil-algebras of polynomial growth~\cite{LenSmo07}.
On further developments concerning Golod-Shafarevich algebras and groups see~\cite{Voden09}, \cite{Ershov12}.

A close by spirit but different construction was motivated by respective group-theoretic results.
A restricted Lie algebra $G$ is called {\it large} if there is a subalgebra  $H\subset G$ of finite codimension
such that $H$ admits a surjective homomorphism on a nonabelian free restricted Lie algebra.
Let $K$ be a perfect at most countable field of positive characteristic.
Then there exist infinite-dimensional finitely generated nil restricted Lie algebras over $K$ that
are residually finite dimensional and direct limits of large restricted Lie algebras~\cite{BaOl07}.

\subsection{Grigorchuk and Gupta-Sidki groups}
The construction of Golod is rather undirect, Grigorchuk gave a direct and elegant construction of
an infinite 2-group generated by three elements of order 2~\cite{Grigorchuk80}.
This group was defined as a group of transformations of the interval $[0,1]$ from which
rational points of the form $\{k/2^n\mid  0\le k\le 2^n,\ n\ge 0\}$ are removed.
For each prime $p\ge 3$, Gupta and Sidki gave a direct construction of an infinite $p$-group
on two generators, each of order $p$~\cite{GuptaSidki83}.
This group was constructed as a subgroup of an automorphism group of an infinite regular tree of degree $p$.

The Grigorchuk and Gupta-Sidki groups are counterexamples to the General Burnside Problem.
Moreover, they gave answers to important problems in group theory.
So, the Grigorchuk group and its further generalizations
are first examples of groups of intermediate growth~\cite{Grigorchuk84}, thus answering
in negative to a conjecture of Milnor that groups of intermediate growth do not exist.
The construction of Gupta-Sidki also yields groups of subexponential growth~\cite{FabGup85}.
The Grigorchuk and Gupta-Sidki groups are {\it self-similar}.
Now self-similar, and so called {\it branch groups}, form a well-established area in
group theory~\cite{Grigorchuk00horizons,Nekr05}.
Below we discuss existence of analogues of the Grigorchuk and Gupta-Sidki groups for other algebraic structures.

\subsection{Self-similar nil graded associative algebras}
The study of these groups lead to investigation of group rings and other related associative algebras~\cite{Sidki97}.
In particular, there appeared self-similar associative algebras defined by matrices
in a recurrent way~\cite{Bartholdi06}.
Sidki suggested two examples of self-similar associative matrix algebras~\cite{Sidki09}.
A more general family of associative algebras was introduced in~\cite{PeSh13ass},
this family generalizes the second example of Sidki~\cite{Sidki09},
also it yields a realization of a Fibonacci restricted Lie algebras (see below)
in terms of self-similar matrices~\cite{PeSh13ass}.
Another important feature of some associative  algebras $A$ constructed in~\cite{PeSh13ass} is that
they are sums of two locally nilpotent subalgebras $A=A_+\oplus A_-$
(see similar decompositions~\eqref{decompLL} below).
Recall that an algebra is said {\em locally nilpotent} if every finitely generated subalgebra is nilpotent.
But the desired analogues of the Grigorchuk and Gupta-Sidki groups should be (self-similar) associative nil-algebras,
in a standard way yielding new examples of finitely generated periodic groups.
But such examples are not known yet.
On similar open problems in theory of infinite dimensional algebras see review~\cite{Zelmanov07}.

\subsection{Self-similar nil restricted Lie algebras, Fibonacci Lie algebra}
Unlike associative algebras, for restricted Lie algebras,
natural analogues of the Grigorchuk and Gupta-Sidki groups are known.
Namely, over a field of characteristic 2,
the first author constructed an example of an infinite dimensional restricted Lie algebra $\LL$
generated by two elements, called a {\it Fibonacci restricted Lie algebra}~\cite{Pe06}.
Let $\ch K=p=2$ and $R=K[t_i| i\ge 0 ]/(t_i^p| i\ge 0)$ a truncated polynomial ring.
Put $\dd_i=\frac {\dd}{\partial t_i}$, $i\ge 0$.
Define the following two derivations of $R$:
\begin{align*}
v_1 & =\dd_1+t_0(\dd_2+t_1(\dd_3+t_2(\dd_4+t_3(\dd_5+t_4(\dd_6+\cdots )))));\\
v_2 & =\qquad\quad\;\,
\dd_2+t_1(\dd_3+t_2(\dd_4+t_3(\dd_5+t_4(\dd_6+\cdots )))).
\end{align*}
These two derivations generate
a restricted Lie algebra $\LL=\Lie_p(v_1,v_2)\subset\Der R$ and an associative algebra $\AA=\Alg(v_1,v_2)\subset \End R$.
The Fibonacci restricted Lie algebra has a slow polynomial growth
with Gelfand-Kirillov dimension $\GKdim \LL=\log_{(\sqrt 5+1)/2} 2\approx 1.44$~\cite{Pe06}.
Further properties of the Fibonacci restricted Lie algebra and its generalizations are studied in~\cite{PeSh09,PeSh13fib}.

Probably, the most interesting property of $\LL$ is that it has a nil $p$-mapping~\cite{Pe06},
which is an analog of the periodicity of the Grigorchuk and Gupta-Sidki groups.
We do not know whether the associative hull $\AA$ is a nil-algebra.
We have a weaker statement.
The algebras $\LL$, $\AA$, and the augmentation ideal
of the restricted enveloping algebra $\uu=\omega u(\LL)$ are direct sums of two locally nilpotent subalgebras~\cite{PeSh09}:
\begin{equation}\label{decompLL}
\LL=\LL_+\oplus \LL_-,\quad \AA=\AA_+\oplus \AA_-,\quad \uu=\uu_+\oplus \uu_-.
\end{equation}
There are examples of infinite dimensional associative algebras which
are direct sums of two locally nilpotent subalgebras~\cite{Kel93,DreHam04}.
Infinite dimensional restricted Lie algebras can have
different decompositions into a direct sum of two locally nilpotent subalgebras~\cite{PeShZe10}.
\medskip

In case of arbitrary prime characteristic,
Shestakov and Zelmanov suggested an example of a finitely generated restricted Lie algebra
with a nil $p$-mapping~\cite{ShZe08}.
That example yields the same decompositions~\eqref{decompLL} for some primes~\cite{Kry11,PeSh13ass}.
An example of a $p$-generated  nil restricted Lie algebra $L$,
characteristic $p$ being arbitrary, was studied in~\cite{PeShZe10}.
The virtue of that example is that for all primes $p$
we have decompositions~\eqref{decompLL} into direct sums of two locally nilpotent subalgebras.
But computations for that example are rather complicated.

Observe that only the original example has a clear monomial basis~\cite{Pe06,PeSh09}.
In other examples, elements of a Lie algebra are  linear combinations of monomials,
to work with such linear combinations is sometimes an essential technical difficulty, see e.g.~\cite{ShZe08,PeShZe10}.
A family of nil restricted Lie algebras of slow growth having good monomial bases
is constructed in~\cite{Pe17},
these algebras are close relatives of a two-generated Lie superalgebra of~\cite{Pe16}.

\subsection{Narrow groups and Lie algebras}
Let $G$ be a group and $G=G_1\supseteq G_2\supseteq \cdots$ its lower central series.
One constructs a related $\N$-graded Lie algebra
$L_K(G)=\oplus_{i\ge 1} L_i$, where $L_i= G_i/G_{i+1}\otimes_{\Z} K$, $i\ge 1$.
A product is given by $[a G_{i+1},b G_{j+1}]=(a,b)G_{i+j+1}$,
where $a\in G_i$, $b\in G_j$, and
$(a,b)=a^{-1}b^{-1}ab$ the group commutator.

A residually $p$-group $G$ is said of {\it finite width} if
all factors $G_i/G_{i+1}$ are finite groups with uniformly bounded orders.
The Grigorchuk group $G$ is of finite width, namely,
$\dim_{\F_2} G_i/G_{i+1}\in\{1,2\}$ for $i\ge 2$~\cite{Rozh96,BaGr00}.
In particular, the respective Lie algebra $L=L_K(G)=\oplus_{i\ge 1} L_i$ has a linear growth.
Bartholdi presented $L_{K}(G)$ as a self-similar restricted Lie algebra
and proved that the restricted Lie algebra $L_{\F_2}(G)$ is nil while $L_{\F_4}(G)$ is not nil~\cite{Bartholdi15}.
Also, $L_K(G)$ is {\it nil graded}, namely,
for any homogeneous element $x\in L_i$, $i\ge 1$, the mapping $\ad x$ is nilpotent,
because the group $G$ is periodic.

A Lie algebra $L$ is called of {\it maximal class} (or {\it filiform}),
if the associated graded algebra with respect to the lower central series
$\gr L=\mathop{\oplus}\limits_{n=1}^\infty \gr L_n$, where $\gr L_n=L^n/L^{n+1}$, $n\ge 1$,
satisfies
\begin{equation}\label{filiform}
    \dim \gr L_1=2,\quad \dim \gr L_n\le 1,\ n\ge 2,\quad \gr L_{n+1}=[\gr L_1, \gr L_n],\ n \ge 1,
\end{equation}
in particular, $\gr L$ is generated by $\gr L_1$.
An infinite dimensional filiform Lie algebra $L$ has the
smallest nontrivial growth function: $\gamma_L(n)=n+1$, $n\ge 1$.
In case of positive characteristic, there are uncountably many such algebras~\cite{CarMatNew97}.
Nevertheless, in case $p>2$, they were classified in~\cite{CarNew00}.
There are generalizations  of filiform Lie algebras.
Naturally $\N$-graded Lie algebras over $\R$ and $\C$
satisfying the condition $\dim L_n+\dim L_{n+1}\le 3$, $n\ge 1$, are classified recently by Millionschikov~\cite{Mil}.
More generally, an $\N$-graded Lie algebra
$L=\mathop{\oplus}\limits_{n=1}^\infty L_n$ is said of finite {\it width} $d$ in the case that $\dim L_n\le d$, $n\ge 1$,
the integer $d$ being minimal.

Pro-$p$-groups and $\N$-graded Lie algebras cannot be simple.
Instead, appears an important notion of being {\it just infinite}, namely,
not having non-trivial normal subgroups (ideals) of infinite index (codimension).
A group (algebra) is said {\it hereditary just infinite}
if and only if any normal subgroup (ideal) of finite index (codimension) is just infinite.
The Gupta-Sidki groups were the first in the class of periodic groups to be shown to be just infinite~\cite{GuptaSidki83A}.
The Grigorchuk group is also just infinite but not hereditary just infinite~\cite{Grigorchuk00horizons}.

\subsection{Lie algebras in characteristic zero}
Since the Grigorchuk group is of finite width,
a right analogue of it should be a Lie algebra of finite width having $\ad$-nil elements, in the next result
the components are of bounded dimension and consist of $\ad$-nil elements.
Informally speaking, there are no "natural analogues" of the Grigorchuk and Gupta-Sidki groups
in the world of Lie algebras of characteristic zero,
strictly in terms of the following result.
\begin{Theorem}[{Martinez and Zelmanov~\cite{MaZe99}}]
\label{TMarZel}
Let $L=\oplus_{\a\in\Gamma}L_\alpha$ be a Lie algebra over a field $K$ of
characteristic zero graded by an abelian group $\Gamma$. Suppose that
\begin{enumerate}
\item
there exists $d>0$ such that $\dim_K L_\alpha \le d $ for all $\alpha\in\Gamma$,
\item
every homogeneous element $a\in L_\a$, $\a\in\Gamma$, is ad-nilpotent.
\end{enumerate}
Then the Lie algebra $L$ is locally nilpotent.
\end{Theorem}

\subsection{Fractal nil graded Lie superalgebras}\label{SSsuper}
In the world of {\it Lie  superalgebras} of an {\it arbitrary characteristic},
the first author constructed analogues of the Grigorchuk and Gupta-Sidki groups~\cite{Pe16}.
Namely, two Lie superalgebras $\RR$, $\QQ$ were constructed,
which are also analogues of the Fibonacci restricted Lie algebra and other (restricted) Lie algebras mentioned above.
Constructions of both Lie superalgebras $\RR$, $\QQ$ are similar,
computations for $\RR$ are simpler, but $\QQ$ enjoys some more specific interesting properties.
The virtue of both examples is that they have clear monomial bases.
They have slow polynomial growth, namely,
$\GKdim \RR=\log_34\approx 1.26$ and $\GKdim \QQ=\log_38\approx 1.89$.
Thus, both Lie superalgebras are of infinite width.
In both examples, $\ad a$ is nilpotent,
$a$ being an even or odd element with respect to the $\Z_2$-gradings as Lie superalgebras.
This property is an analogue of the periodicity of the Grigorchuk and Gupta-Sidki groups.
The Lie superalgebra $\RR$ is $\Z^2$-graded, while $\QQ$ has a natural fine $\Z^3$-gradation
with at most one-dimensional components
(See on importance of fine gradins for Lie and associative algebras~\cite{BaSeZa01,Eld10}).
In particular, $\QQ$ is a nil finely graded Lie superalgebra, which shows
that an extension of Theorem~\ref{TMarZel} (Martinez and Zelmanov~\cite{MaZe99})
for the Lie {\it super}algebras of characteristic zero is not valid.
Also, $\QQ$ has a $\Z^2$-gradation which yields a continuum of different decompositions
into sums of two locally nilpotent subalgebras $\QQ=\QQ_+\oplus\QQ_-$.
Both Lie superalgebras are {\it self-similar},
they also contain infinitely many copies of itself, we call them {\it fractal} due to the last property.
(Except this paragraph, $\QQ$ denotes another Lie superalgebra, one of the main object of this paper).
\medskip

In~\cite{PeOtto}, we construct a similar but simpler and "smaller" example.
Namely, we construct a 2-generated fractal Lie superalgebra $\mathbf{R}$
(the same notation as above but this is a different algebra) over arbitrary field.
This Lie superalgebra $\RR$ is $\Z^2$-graded by multidegree in the generators and the
$\Z^2$-components are at most one-dimensional.
As an analogue of periodicity,
we establish that homogeneous elements of the $\Z_2$-grading $\mathbf{R}=\mathbf{R}_{\bar 0}\oplus\mathbf{R}_{\bar 1}$ are $\ad$-nilpotent.
In case of $\mathbb{N}$-graded algebras, a close analogue to being simple is being just infinite.
Unlike previous examples of Lie superalgebras~\cite{Pe16}, we are able to prove that $\mathbf{R}$ is just infinite,
but not hereditary just infinite.
This example is close to the smallest possible one, because $\mathbf{R}$ has a linear growth
with a growth function $\gamma_\mathbf{R}(m)\approx 3m$, as $m\to\infty$.
Moreover, its degree $\mathbb{N}$-gradation is of finite width 4 ($\ch K\ne 2$).
In case $\ch K=2$, we obtain a Lie algebra of width 2 that is not thin.

\subsection{Poisson and Jordan (super)algebras}
Poisson algebras naturally appear in different areas of algebra, topology and physics.
Probably, Poisson algebras were first introduced in 1976 by Berezin~\cite{Ber67}, see also Vergne~\cite{Ver69} (1969).
The free Poisson (super)algebras were introduced by Shestakov~\cite{Shestakov93}.
Applying Poisson algebras, Shestakov and Umirbaev managed to solve a long-standing problem:
they proved that the Nagata automorphism of the polynomial ring in three variables $\mathbb{C}[x,y,z]$ is wild~\cite{SheUmi04}.
Related algebraic properties of free Poisson algebras were studied by
Makar-Limanov, Shestakov and Umirbaev~\cite{MakShe12,MakUmi11}.
A basic theory of identical relations for Poisson algebras  was developed by Farkas~\cite{Farkas98,Farkas99}.
See further developments on the theory of identical relations of Poisson algebras,
in particular, the theory of so called  codimension growth in characteristic zero
by Mishchenko, Petrogradsky, Regev~\cite{MiPeRe}, and Ratseev~\cite{Ratseev13}.

Simple finite dimensional nontrivial Jordan superalgebras over an algebraically closed field of characteristic zero
were classified ~\cite{Kac77CA,Kantor92}.   
Infinite-dimensional $\Z$-graded simple Jordan superalgebras
with a unit element over an algebraically closed field of characteristic zero which components are uniformly bounded
are classified in~\cite{KacMarZel01}.
Recently, just infinite Jordan superalgebras were studied in~\cite{ZhePan17}.

\begin{Theorem}[{Zelmanov, private communication~\cite{Zelmanov}}]\label{TZelmanov}
Jordan algebras in characteristic zero satisfy a verbatim analogue of Theorem~\ref{TMarZel}.
\end{Theorem}
Strictly in terms of this result, we say again that
there are no natural analogues of the Grigorchuk and Gupta-Sidki groups in the class of Jordan algebras too.
On the other hand, the Jordan superalgebra $\KK$ constructed in the present paper shows
that an extension of this result to the Jordan superalgebras is not valid.
These facts resemble those for Lie algebras and superalgebras mentioned above.

We continue this research and construct
a similar but "smaller" example, namely, a fractal nil Jordan superalgebra of finite width in~\cite{PeSh18Jslow}.

\section{Basic definitions: superalgebras, growth}\label{Sdef}

\subsection{Associative and Lie Superalgebras}
Denote $\NO=\{0,1,2,\dots\}$.
By $K$ denote the ground field, $\langle S\rangle_K$ a linear span of a subset $S$ in a $K$-vector space.

Superalgebras appear naturally in physics and mathematics~\cite{Kac77,Scheunert,BMPZ}.
Put $\Z_2=\{\bar 0,\bar 1\}$, the group of order 2.
A {\em superalgebra} $A$ is a $\Z_2$-graded algebra $A=A_{\bar 0}\oplus A_{\bar 1}$.
The elements $a\in A_\alpha$ are called {\em homogeneous of degree} $|a|=\alpha\in\Z_2$.
The elements of $A_{\bar 0}$ are {\em even}, those of $A_{\bar 1}$ {\em odd}.
In what follows, if $|a|$ enters an expression,
then it is assumed that $a$ is homogeneous of degree $|a|\in\Z_2$,
and the expression extends to the other elements by linearity.
Let $A,B$ be superalgebras, a {\em tensor product} $A\otimes B$ is the superalgebra
whose space is the tensor product of the spaces $A$ and $B$ with the induced $\Z_2$-grading and the product:
$$
(a_1\otimes b_1) (a_2\otimes b_2)=(-1)^{|b_1|\cdot |a_2|}a_1a_2\otimes b_1b_2,\quad a_i\in A,\ b_i\in B.
$$
An {\em associative superalgebra} $A$ is a
$\Z_2$-graded associative algebra $A=A_{\bar 0}\oplus A_{\bar 1}$.
A {\em Lie superalgebra} is a $\Z_2$-graded algebra $L=L_{\bar 0}\oplus L_{\bar 1}$ with an
operation $[\ ,\ ]$ satisfying the axioms ($\ch K\ne 2,3$):
\begin{itemize}
\item
$[x,y]=-(-1)^{|x|\cdot |y| }[y,x]$,\qquad\qquad (super-anticommutativity);
\item
$[x,[y,z]]=[[x,y],z]+(-1)^{|x|\cdot| y|}[y,[x,z]]$,\qquad (Jacobi identity).
\end{itemize}
All commutators in the present paper are supercommutators.
Long commutators are {\it right-normed}: $[x,y,z]=[x,[y,z]]$.
We use a standard notation $\ad x(y)=[x,y]$, where $x,y\in L$.

Assume that $A=A_{\bar 0}\oplus A_{\bar 1}$ is an associative superalgebra.
One obtains a Lie superalgebra $A^{(-)}$
by supplying the same vector space $A$  with a {\em supercommutator}:
$$
[x,y]=xy-(-1)^{|x|\cdot |y|}yx,\quad x,y\in A.
$$
If 
$A^{(-)}$ is abelian,
then 
$A$ is called {\it supercommutative}.
Let $L$ be a Lie superalgebra,
one defines a {\it universal enveloping algebra}
$U(L)=T(L)/(x\otimes y- (-1)^{|x|\cdot|y|}y\otimes x-[x,y]\mid x,y\in L)$, where
$T(L)$ is the tensor algebra of the vector space $L$.
Now, the product in $L$ coincides with the supercommutator in $U(L)^{(-)}$.
A basis of  $U(L)$ is given by PBW-theorem~\cite{BMPZ,Scheunert}.

Let $V=V_{\bar 0}\oplus V_{\bar 1}$ be a vector space, we say that it is $\Z_2$-graded.
The associative algebra of all vector space endomorphisms
$\End V$ is an associative superalgebra:
$\End V=\End_{\bar 0} V\oplus \End_{\bar 1} V$,
where $\End_\alpha V=\{\phi\in\End V\mid \phi(V_\beta)\subset V_{\alpha+\beta}, \beta\in\Z_2\}$, $\alpha\in\Z_2$.
Thus, $\End^{(-)}V$ is a Lie superalgebra, called the {\em general linear superalgebra} ${gl}(V)$.

Let $A=A_{\bar 0}\oplus A_{\bar 1}$ be a $\Z_2$-graded algebra of arbitrary signature.
A linear mapping $\phi\in\End_{\beta} A$, $\beta\in\Z_2$,
is a {\em superderivative} of degree $\beta$ if it satisfies
\begin{equation*}
\phi(a\cdot b)=\phi(a)\cdot b+(-1)^{\beta |a|}a\cdot\phi(b),\quad a,b\in A.
\end{equation*}
Denote by $\Der_{\alpha}A\subset \End_{\alpha} A$ the space of all superderivatives of degree $\alpha\in\Z_2$.
One checks that $\Der A=\Der_{\bar 0} A\oplus \Der_{\bar 1} A$
is a subalgebra of the Lie superalgebra $\End^{(-)}A$.
All superderivations of the Grassmann algebra $\Lambda(n)=\Lambda(x_1,\dots,x_n)$
is a simple Lie superalgebra $\mathbf W(n)$ for $n\ge 2$.
In this paper by a derivation we always mean a superderivation.

\subsection{Lie superalgebras in small characteristics}
In case $\ch K=2,3$ the axioms of the Lie superalgebra have to be augmented
(\cite[section 1.10]{BMPZ}, \cite{Bou-Leites-09}, \cite{Pe16}).
\begin{itemize}
\item
$[z,[z,z]]= 0$,  $z\in L_{\bar 1}$\quad  (in case $\ch K= 3$).
\end{itemize}
Substituting $x=y\in L_{\bar 1}$ in the Jacobi identity,  we get $2 (\ad x)^2 z=[[x,x], z]$.
In case $\ch K\ne 2$ we get an identity
\begin{equation*}
(\ad x)^2 z= \frac 12 [[x,x],z],\qquad  x\in L_{\bar 1}, \ z\in L.
\end{equation*}
In the present paper we study Lie superalgebras of the form $A^{(-)}$,
they have squares for odd elements: $[x,x]=2 x^2$, $x\in A_{\bar 1}^{(-)}$.
One obtains an identity which is also valid for algebras $A^{(-)}$ in case $\ch K=2$:
\begin{equation}\label{squares}
(\ad x)^2 z= [x^2,z] 
,\qquad x\in A^{(-)}_{\bar 1},\ z\in A^{(-)}.
\end{equation}
So, in case $\ch K=2$, we add more axioms for the Lie superalgebras:
\begin{itemize}
\item there exists a {\em quadratic mapping (a formal square)}:
$(\ )^{[2]}:L_{\bar 1}\to L_{\bar 0}$, $x\mapsto x^{[2]}$, $x\in L_{\bar 1}$, satisfying:
\begin{align}
 (\lambda x)^{[2]}&=\lambda^2 x^{[2]},\qquad x\in L_{\bar 1},\ \lambda\in K;\nonumber\\
 (x+y)^{[2]}&=x^{[2]}+[x,y]+y^{[2]},\qquad x,y\in L_{\bar 1}; \label{square2}\\
 (\ad x)^2 z &= [x^{[2]},z],\qquad  x\in L_{\bar 1},\  z\in L,\
 \text{(a formal substitute of~\eqref{squares})};\nonumber
\end{align}
\item $[x,x]=0$, $x\in L_{\bar 0}$.
By putting $y=x$ in the second relation above, we get $[y,y]=0$, $y\in L_{\bar 1}$.
\end{itemize}
Thus, a Lie superalgebra in case $\ch K=2$ is just a $\Z_2$-graded Lie algebra
supplied with a quadratic mapping
$L_{\bar 1}\to L_{\bar 0}$, which is similar to the  $p$-mapping (see below).
In case $p=2$, to get the universal enveloping algebra,
we additionally factor out $\{y\otimes y-y^{[2]}\mid y\in L_{\bar 1}\}$.


\subsection{Restricted Lie (super)algebras}
Let $\ch K=p>0$.
A Lie algebra~$L$ is a
\textit{restricted Lie algebra} (or \textit{Lie $p$-algebra}),
if it is supplied with a unary operation
 $x\mapsto x^{[p]}$, $x\in L$, that satisfies the following axioms~\cite{JacLie,Strade1,StrFar}:
\begin{itemize}
\item $(\lambda x)^{[p]}=\lambda^px^{[p]}$, for $\lambda\in K$, $x\in L$;
\item $\ad(x^{[p]})=(\ad x)^p$, $x\in L$;
\item $(x+y)^{[p]}=x^{[p]}+y^{[p]}+\sum_{i=1}^{p-1}s_i(x,y)$, $x,y\in L$,
where $is_i(x,y)$~is the coefficient of $t^{i-1}$ in the polynomial
$\operatorname{ad}(tx+y)^{p-1}(x)\in L[t]$.
\end{itemize}
This notion is motivated by the following observation.
Let $A$ be an associative algebra over a field~$K$, $\ch K=p>0$.
Then the mapping  $x\mapsto x^p$, $x\in A^{(-)}$,
satisfies these conditions considered in the Lie algebra $A^{(-)}$.

A {\em restricted Lie superalgebra} $L=L_{\bar 0}\oplus L_{\bar 1}$ is a Lie superalgebra
such that the even component $L_{\bar 0}$ is a restricted Lie algebra and $L_{\bar 0}$-module
$L_{\bar 1}$ is restricted, i.e. $\ad(x^{[p]}) y=(\ad x)^p y$, for all $x\in L_{\bar 0}$, $y\in L_{\bar 1}$
(see. e.g.~\cite{Mikh88,BMPZ}).
Remark that in case $\ch K=2$, the restricted Lie superalgebras and
$\Z_2$-graded restricted Lie algebras are the same objects.
(Let $L=L_{\bar 0}\oplus L_{\bar 1}$ be a restricted Lie superalgebra,
it has the $p$-mapping on the even part: $L_{\bar 0}\to L_{\bar 0}$ and the formal square
on the odd part: $L_{\bar 1}\to L_{\bar 0}$.
We obtain the $p$-mapping
on the whole of algebra by setting $(x+y)^{[2]}=x^{[2]}+y^{[2]}+[x,y]$, $x\in L_{\bar 0}$,  $y\in L_{\bar 1}$).

Let $L$ be a restricted Lie (super)algebra, and
$J$ the ideal of the universal enveloping algebra~$U(L)$
generated by $\{x^{[p]}-x^p\mid x\in L_{\bar 0}\}$.
Then $u(L)=U(L)/J$ is the \textit{restricted enveloping algebra}.
In this algebra, the formal operation $x^{[p]}$ coincides with the ordinary power $x^p$ for all $x\in L_{\bar 0}$.
One has an analogue of PBW-theorem
describing a basis of $u(L)$ \cite[p.~213]{JacLie}, \cite{BMPZ}.

Let $L$ be a Lie (super)algebra. One defines the {\it lower central series} as
$L^1=L$ and $L^{n}=[L,L^{n-1}]$, $n\ge 2$. In case $\ch K=2$ the terms above
are augmented by $\langle x^2| x\in (L^{[n/2]})_{\bar 1}\rangle_K$.
In case of a restricted Lie (super)algebra,
we also add $\langle x^p| x\in (L^{[n/p]})_{\bar 0}\rangle_K$.

\subsection{Poisson superalgebras}\label{SSPoisson}
A $\Z_2$-graded vector space $A=A_{\bar 0}\oplus A_{\bar 1}$ is called a {\it Poisson superalgebra}
provided that, beside the addition, $A$ has two $K$-bilinear operations as follows:
\begin{itemize}
\item
$A=A_{\bar 0}\oplus A_{\bar 1}$ is an associative superalgebra with unit whose multiplication is denoted by
$a\cdot b$ (or $ab$), where $a, b\in A$.
We assume that $A$ is {\it supercommutative}, i.e. $a\cdot b=(-1)^{|a|\cdot |b|}b\cdot a$,
for all $a,b\in A$.
\item
$A=A_{\bar 0}\oplus A_{\bar 1}$
is a Lie superalgebra whose product is traditionally denoted by the {\it Poisson bracket}
$\{a, b\}$, where $a, b\in A$.
\item these two operations are related by the {\it super Leibnitz rule}:
\begin{equation*}
\{a\cdot b, c\}=a\cdot\{b, c\}+(-1)^{|b|\cdot |c|}\{a, c\}\cdot b,\qquad  a, b, c \in A.
\end{equation*}
\end{itemize}
Let $L$ be a Lie superalgebra,
$\{U_n| n\ge 0\}$ the natural filtration of its universal enveloping algebra $U(L)$.
Consider the {\it symmetric algebra}
$S(L)=\gr U(L)=\mathop{\oplus}\limits_{n=0}^\infty U_{n}/U_{n+1}$ (see~\cite{Dixmier}).
Recall that $S(L)$ is identified with a supercommutative algebra $K[v_i\,|\, i\in I]\otimes \Lambda (w_j,|\,j\in J)$, where
$\{v_i\,|\, i\in I\}$, $\{w_j\,|\, j\in J\}$, are bases of $L_{\bar 0}$, $L_{\bar 1}$, respectively.
Define a Poisson bracket by setting $\{v,w\}=[v,w]$, $v,w\in L$,
and extending to the whole of $S(L)$ by linearity and using the Leibnitz rule.
Thus, $S(L)$ is turned into a Poisson superalgebra, called the {\it symmetric algebra} of $L$.
Let $L(X)$ be the free Lie superalgebra generated by a graded set $X$,
then $S(L(X))$ is a free Poisson superalgebra~\cite{Shestakov93}.

Let $\ch K=2$, the axioms of a Lie superalgebra require existence of a formal square $y\mapsto y^{[2]}$ for all odd $y$.
Consider a free Poisson superalgebra $A=A_{\bar 0}\oplus A_{\bar 1}$ over $\Q$,
let $a\in A_{\bar 0}$, $b\in A_{\bar 1}$, then
$(ab)^{[2]}=\frac 12 \{ab,ab\}=\frac 12 ( \{a,a\}bb+aa\{b,b\}+ 2ab\{a,b\})=a^2b^{[2]}+ab\{a,b\}$.
Thus,  we add additional axioms for a {\it Poisson superalgebra in case} $\ch K=2$:
\begin{itemize}
\item
$(ab)^{[2]}=a^2b^{[2]}+ab\{a,b\}$ for all $a\in A_{\bar 0}$, $b\in A_{\bar 1}$;
\item $b^2=0$, for all $b\in A_{\bar 1}$.
\end{itemize}

One checks that validity of these axioms on any basis imply them for all elements (the second axiom is needed here).
Also, the computation above yields an additional axiom for a {\it restricted Poisson algebra $A$ in case} $\ch K=2$:
\begin{itemize}
\item
$(ab)^{[2]}=a^2b^{[2]}+a^{[2]}b^{2}+ab\{a,b\}$ for all $a\in A$.
\end{itemize}
Again, one checks that  it is sufficient to verify validity of this axiom on any basis.
Observe that the case $p=2$ was not considered in
a definition of a {\it restricted Poisson algebra} given for all  $p>2$ in~\cite{BezKal07,BaoYeZhang17}.

Let $A$, $P$ be Poisson superalgebras, their tensor product $A\otimes P$ is a Poisson superalgebra with operations:
$(a\otimes v)\cdot (b\otimes w)= (-1)^{|v||b|}ab\otimes vw$ and
$\{a\otimes v, b\otimes w\}= (-1)^{|v||b|} (\{a, b\}\otimes vw+ ab\otimes \{v,w\})$, where $a,b\in A$, $v,w\in P$.

Let $\Lambda(n)=\Lambda(x_1,\dots,x_n)$ be the Grassmann algebra in $n$ variables.
It is an associative superalgebra, where the $\Z_2$-grading
$\Lambda(n)=\Lambda_{\bar 0}(n)\oplus \Lambda_{\bar 1}(n)$
is given by parity of monomials in the generators.
One supplies $\Lambda(n)$ with a bracket:
$$
\{f,g\}=(-1)^{|f|-1}\sum_{i=1}^n
\frac{\partial f}{\partial x_i}\frac{\partial g}{\partial x_i},\qquad f,g\in\Lambda(n).
$$
This bracket is induced by relations $\{x_i,x_j\}=\delta_{i,j}$, $1\le i,j\le n$.
Then $\Lambda(n)$ is a simple Poisson superalgebra.

Consider a modification of this construction.
Let $H_{n}=\Lambda(x_1,\dots,x_n,y_1,\dots,y_n)$
be the Grassmann superalgebra supplied with a bracket determined by:
$\{x_i,y_j\}=\delta_{i,j}$, $\{x_i,x_j\}=\{y_i,y_j\}=0$ for $1\le i,j\le n$.
We obtain a simple {\it Hamiltonian} Poisson superalgebra with a bracket:
$$
\{f,g\}=(-1)^{|f|-1}\sum_{i=1}^n
\bigg(\frac{\partial f}{\partial x_i}\frac{\partial g}{\partial y_i}
+\frac{\partial f}{\partial y_i}\frac{\partial g}{\partial x_i}\bigg),\qquad f,g\in H_{n}.
$$

Let $P= P_{\bar  0} \oplus P_{\bar 1}$ be a Poisson superalgebra with products $\cdot$ and $\{\ ,\ \}$.
Recall that an {\it algebraic quantization} of $P$ is a polynomial extension
$P[t]$ supplied with an associative product $*$ that  agrees with the grading
$P[t]= P_{\bar  0}[t] \oplus P_{\bar 1}[t]$  and such that (see e.g.~\cite{Shestakov93}):
\begin{itemize}
\item $a * b = a\cdot b\ (\mathrm{ mod}\ t)$,\qquad  $a,b\in P$;
\item $a * b - (-1)^{|a||b|} b * a = t \{a,b\}\ (\mathrm{mod}\ t^ 2)$,\qquad  $a,b\in P$;
\item $f * t = t * f = ft$,\qquad   $f\in P[t]$.
\end{itemize}

\subsection{Jordan superalgebras}
While studying Jordan (super)algebras we always assume that $\ch K\ne 2$.
A {\it Jordan algebra} is an algebra $J$ satisfying the identities
\begin{itemize}
\item  $ab=ba$;
\item  $a^2(ca)=(a^2c)a$.
\end{itemize}
A {\em Jordan superalgebra} is a $\Z_2$-graded algebra  $J=J_{\bar 0} \oplus J_{\bar 1}$ satisfying the graded identities:
\begin{itemize}
\item  $ab=(-1)^{|a||b|}ba$;
\item  $(ab)(cd)+(-1)^{|b||c|}(ac)(bd)+(-1)^{(\!|b|+|c|)|d|}(ad)(bc) \\
=((ab)c)d+(-1)^{|b|(\!|c|+|d|)+|c||d|}((ad)c)b+(-1)^{|a|(\!|b|+|c|+|d|)+|c||d|}((bd)c)a.$
\end{itemize}

Let $A=A_{\bar 0}\oplus A_{\bar 1}$ be an associative superalgebra.
The same space supplied with
the product $a\circ b=\frac 12(ab+(-1)^{|a||b|}ba)$ is a Jordan superalgebra $A^{(+)}$.
A Jordan superalgebra $J$ is called {\it special} if it can be embedded into a Jordan
superalgebra of the type $A^{(+)}$.
Also, $J$ is called {\it i-special} (or {\it weakly special})
if it is a homomorphic image of a special one.

I.L.~Kantor suggested the following doubling process, which is applied to a Poisson (super)algebra $A$
and the result is  a Jordan superalgebra $\Kan(A)$~\cite{Kantor92}.
The $K$-module $\Kan(A)$ is the direct sum $A\oplus \bar A$,
where $\bar A$ is a copy of $A$,
let $a\in A$ then $\bar a$ denotes the respective element in $\bar A$.
Also, $\bar A$ is supplied with the opposite $\Z_ 2$-grading, i.e., $|\bar a| = 1 - |a|$ for a $\Z_2$-homogeneous $a\in A$.
The multiplication $\bullet$ on $\Kan(A)$ is defined by:
\begin{align*}
a \bullet b      &= ab,\\
\bar a \bullet b &= (-1)^{|b|} \overline{ab},\\
a \bullet \bar b &= \overline{ab}, \\
\bar a \bullet \bar b &= (-1)^{|b|} \{a,b\},\qquad a,b\in A.
\end{align*}
This construction is important because it yielded a new series of finite dimensional simple Jordan superalgebras
$\Kan(\Lambda(n))$, $n\ge 2$~\cite{Kantor92,KingMcCrimon92}.

\subsection{Growth}
We recall the notion of {\em growth}. Let $A$  be an associative (or Lie) algebra  generated by a finite set $X$.
Denote  by $A^{(X,n)}$ the subspace of $A$ spanned by all  monomials  in $X$
of length not  exceeding  $n$, $n\ge 0$.
In case of a Lie superalgebra of $\ch K=2$ we also consider formal squares of odd monomials of length at most $n/2$.
If $A$ is a restricted Lie algebra, put
$A^{(X,n)}=\langle\, [x_{1},\dots,x_{s}]^{p^k}\mid x_{i}\in X,\, sp^k\le n\rangle_K$~\cite{Pape01}.
Similarly, one defines the growth for restricted Lie superalgebras.
In either situation, one defines an {\em  (ordinary) growth function}:
$$
\gamma_A(n)=\gamma_A(X,n)=\dim_KA^{(X,n)},\quad n\ge 0.
$$
Let $f,g:\N\to\R^+$ be eventually increasing and positive valued functions.
Write $f(n)\preccurlyeq g(n)$ if and only if there exist positive constants $N,C$ such that $f(n)\le g(Cn)$
for all $n\ge N$.
Introduce equivalence $f(n)\sim g(n)$ if and only if  $f(n)\preccurlyeq g(n)$ and $g(n)\preccurlyeq f(n)$.
Different generating sets of an algebra yield equivalent growth functions~\cite{KraLen}.

It is well known that the
exponential growth is the highest possible growth for finitely generated Lie and
associative algebras. A growth function $\gamma_A(n)$ is
compared with polynomial functions $n^\alpha$, $\alpha\in\R^+$, by
computing the {\em upper and lower Gelfand-Kirillov
dimensions}~\cite{KraLen}:
\begin{align*}
\GKdim A&=\limsup_{n\to\infty} \frac{\ln\gamma_A(n)}{\ln n}=\inf\{\a>0\mid \gamma_A(n)\preccurlyeq n^\a\} ;\\
\LGKdim A&=\liminf_{n\to\infty}\,  \frac{\ln\gamma_A(n)}{\ln n}=\sup\{\a>0\mid \gamma_A(n)\succcurlyeq n^\a\}.
\end{align*}
By Bergman's theorem, the Gelfand-Kirillov dimension of an associative algebra cannot belong to the interval $(1,2)$~\cite{KraLen}.
Similarly, there are no finitely generated Jordan algebras with Gelfand-Kirillov dimension strictly between 1 and 2~\cite{MaZe96}.
Such a gap for Lie algebras does not exist, the Gelfand-Kirillov
dimension of a finitely generated Lie algebra can be arbitrary number $\{0\}\cup [1,+\infty)$~\cite{Pe97}.

Assume that generators $X=\{x_1,\dots,x_k\}$ are assigned positive weights $\wt(x_i)=\lambda_i$, $i=1,\dots,k$.
Define a {\it weight growth function}:
$$
\tilde \gamma_A(n)=\dim_K\langle x_{i_1}\cdots x_{i_m}\mid \wt(x_{i_1})+\cdots+\wt(x_{i_m})\le n,\
          x_{i_j}\in X\rangle_K,\quad n\ge 0.
$$
Set $C_1=\min\{\lambda_i\mid i=1,\dots,k \}$, $C_2=\max\{\lambda_i\mid i=1,\dots,k \}$,
then $\tilde\gamma_A(C_1 n) \le \gamma_A(n)\le \tilde\gamma_A(C_2 n)$ for $n\ge 1$.
Thus, we obtain an equivalent growth function $\tilde \gamma_A(n)\sim\gamma_A(n)$.
Therefore, we can use the weight growth function $\tilde\gamma_A(n)$ in order to
compute the Gelfand-Kirillov dimensions.
By $f(n)\approx g(n)$, $n\to\infty$, denote that $\mathop{\lim}\limits_{n\to\infty} f(n)/g(n)=1$.
Similarly, one studies the growth for Poisson and Jordan superalgebras.

Suppose that $L$ is a Lie (super)algebra and $X\subset L$.
By $\Lie(X)$ denote the subalgebra of $L$ generated by $X$,
(including application of the quadratic mapping in case $\ch K=2$).
Let $L$ be a restricted Lie (super)algebra,
by $\Lie_p(X)$ denote the restricted subalgebra of $L$ generated by $X$.
Assume that $X$ is a subset of an associative algebra $A$.
Write $\Alg(X)\subset A$ to denote the associative subalgebra (without unit) generated by~$X$.
In case of Poisson and Jordan superalgebras we use notations $\Poisson(X)$ and $\Jord(X)$.
A grading of an algebra is called {\em fine} if
it cannot be splitted by taking a bigger grading group (see definitions in~\cite{BaSeZa01,Eld10}).

\subsection{Lie superalgebra  $\WW(\Lambda_I)$ of special superderivations}
Assume that $I$ is a well-ordered set of arbitrary cardinality. Put $\Z_2=\{0,1\}$.
Let $\Z_2^I=\{\a: I\to\Z_2\}$ be a set of functions with finitely many nonzero values.
Suppose that $\a \in \Z_2^I$ has nonzero values at $\{i_1,\dots,i_t\}\subset I$, where $i_1<\cdots< i_t$, put
$\mathbf x^\a=x_{i_1}x_{i_2}\cdots x_{i_t}$ and $|\a|=t$.
Now $\{ \mathbf x^\a\mid \a\in \Z_2^I\}$ is a basis of the Grassmann algebra $\Lambda_I=\Lambda(x_i\mid i\in I)$,
which is an associative superalgebra $\Lambda_I=\Lambda_{\bar 0}\oplus \Lambda_{\bar 1}$,
all $x_i$, $i\in I$, being odd.
Let $\partial_i$, $i\in I$, denote the superderivatives of
$\Lambda$, which are determined by the values
$\partial_i(x_j)=\delta_{ij}$, $i,j\in I$.
We identify $x_i$, $i\in I$, with the operator
of the left multiplication on $\Lambda_I$, thus we get odd elements
 $x_i\in \End_{\bar 1}(\Lambda_I)$, $i\in I$.
Consider a space of all formal sums
\begin{equation}\label{WW}
\WW(\Lambda_I)=\bigg\{\sum_{\a\in \Z_2^I} {\mathbf x}^{\a}\sum_{j=1}^{m(\a)}
                           \lambda_{\a,i_j}\,\partial_{i_j}
                    \ \bigg|\ \lambda_{\a,i_j}\in K,\ i_j\in I
                    \bigg\}.
\end{equation}
It is essential that the sum at each ${\mathbf x}^{\a}$, $\a\in \Z_2^I$,  is finite.
This construction is similar to the Lie algebra of {\em special derivations},
see~\cite{Rad86}, \cite{Razmyslov}, \cite{PeRaSh}.
It is similarly verified that
the product in $\WW(\Lambda_I)$ is well defined and
$\WW(\Lambda_I)$ acts on $\Lambda_I$ by superderivations.

\section{Main results: superalgebras $\QQ$, $\AA$, $\PP$, $\JJ$, $\KK$, and their properties}

In this paper, we study the following five objects.
A core of our constructions is a Lie superalgebra $\QQ$.
Next we construct the associative hull $\AA$,
a related Poisson superalgebra $\PP$, and two Jordan superalgebras $\JJ$ and $\KK$.
We call these superalgebras {\em fractal} because they contain infinitely many copies of themselves.

Let us briefly describe their constructions, the next picture shows relations between constructions.
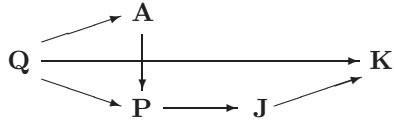
\begin{center}
\begin{picture}(140,50)
\put(5,22){$\mathbf{Q}$}
\put(52,40){$\mathbf{A}$}
\put(52,04){$\mathbf{P}$}
\put(98,04){$\mathbf{J}$}
\put(142,22){$\mathbf{K}$}
\put(18,32){\vector(3,1){30}}
\put(18,18){\vector(3,-1){30}}
\put(56,35){\vector(0,-1){21}}
\put(106,08){\vector(3,1){33}}
\put(64,07){\vector(3,0){28}}
\put(18,25){\vector(1,0){120}}
\end{picture}
\end{center}

Consider the Grassmann algebra  in infinitely many variables $\Lambda=\Lambda(x_i\mid i\ge 0)$.
Let $\dd_i$ 
be its superderivative defined by $\dd_i(x_j)=\delta_{i,j}$, $i,j\ge 0$.
Observe that $\{x_i,\dd_i\mid i\ge 0\}$ are odd elements of the associative superalgebra $\End\Lambda$,
where $x_i$ is identified with the left miltiplication on $\Lambda$.
These elements anticommute except for nontrivial relations:
\begin{equation*}
[\dd_i,x_i]=\dd_ix_i+x_i\dd_i=1;\qquad
x_i^2=0,\quad \dd_i^2=0,\qquad i\ge 0.
\end{equation*}

Now we define {\em pivot elements}:
\begin{equation}\label{pivot}
v_i = \dd_i + x_ix_{i+1}(\dd_{i+3} + x_{i+3}x_{i+4}(\dd_{i+6} + x_{i+6}x_{i+7}(\dd_{i+9}+\cdots) )),\qquad i\ge 0.
\end{equation}
The action of
the pivot elements on the Grassmann letters is well defined and produces letters with smaller indices:
\begin{equation}\label{action}
v_n(x_k)=
\begin{cases}
0, & k<n;\\
1, & k=n;\\
x_n x_{n+1}\hat x_{n+2} x_{n+3} x_{n+4}\hat x_{n+5}\cdots \hat x_{k-4} x_{k-3}x_{k-2},\quad & k=n+3l, \quad l\ge 1;\\
0, & k-n\ne 0\, (\mathrm{mod}\, 3);
\end{cases}
\end{equation}
where $\hat x_i$ denote omitted variables.
Thus, we obtain a sequence of superderivatives $\{v_i \mid i\ge 0\}\subset \Der \Lambda$,
moreover, they belong to $\WW(\Lambda)$.
First, we define a Lie superalgebra
$\QQ=\Lie(v_0,v_1,v_2)\subset\WW(\Lambda)\subset \Der \Lambda$ generated by $\{v_0,v_1,v_2\}$.
Second, we take its {\it associative hull}, namely, we consider the associative superalgebra
$\AA=\Alg(v_0,v_1,v_2)\subset \End \Lambda$ generated by $\{v_0,v_1,v_2\}$.
(We warn that another Lie superalgebra was denoted by $\QQ$ in~\cite{Pe16}, see also subsection~\ref{SSsuper}).
We start the present paper with a study of  properties of the algebras $\QQ$ and $\AA$.

Next, in Section~\ref{SPoisson} we consider the Grassmann algebra
$H_\infty=\Lambda(x_i,y_i| i\ge 0)$ which is turned into a Poisson superalgebra by a bracket determined by relations:
\begin{equation*} 
\{y_i,x_j\}=\delta_{i,j},\quad \{x_i,x_j\}=\{y_i,y_j\}=0,\quad  i,j\ge 0.
\end{equation*}
In its completion $\tilde H_\infty$,
the next elements will be referred to as the {\em pivot elements} as well:
$$ 
V_i = y_i + x_ix_{i+1}(y_{i+3} + x_{i+3}x_{i+4}(y_{i+6} + x_{i+6}x_{i+7}(y_{i+9}+\cdots) ))\in
\tilde H_\infty ,\qquad i\ge 0.
$$ 
We actually obtain the same Lie superalgebra: $\QQ=\Lie(v_0,v_1,v_2)\cong \Lie(V_0,V_1,V_2)$.

Third, we define a Poisson subalgebra $\PP=\Poisson(V_0,V_1,V_2)\subset \tilde H_\infty$ generated by $\{V_0,V_1,V_2\}$.
Using the Kantor double, we
construct the forth object, a Jordan superalgebra $\JJ=\Kan(\PP(V_0,V_1,V_2))=\PP\oplus \bar \PP$
and prove that $\JJ=\Jord(V_0,V_1,V_2,\bar 1)$ (Section~\ref{SJordan}).

Finally, a Jordan superalgebra $\KK$ is a factor algebra of $\JJ$, it also can be constructed directly
as a double $\KK=\Jor(\QQ)$, namely, as a vector space
supplied with an operation as follows (Section~\ref{SJordanK}):
$$
\KK=\langle 1\rangle \oplus \QQ\oplus \langle \bar 1\rangle \oplus \bar \QQ,\qquad
\bar x\bullet \bar y=[x,y], \quad x\bullet \bar 1=(-1)^{|x|}\bar 1\bullet x=\bar x,\quad x,y\in \QQ; \ 1 \text{ the unit}.
$$

Now we formulate main properties of these five superalgebras established in the paper.

\begin{enumerate}
\renewcommand{\theenumi}{\roman{enumi}}
\item Section~\ref{Srelations} yields multiplication rules of the Lie superalgebra $\QQ$.
\item $\QQ$ has a clear monomial basis consisting of standard monomials of two types ($\ch K\ne 2$, Theorem~\ref{Tbasis3}).
      In case $\ch K=2$, a basis of $\QQ$ consists of monomials of the first type
      and squares of the pivot elements (Corollary~\ref{Cbasis2}), and $\QQ$
      coincides with the restricted Lie (super)algebra $\Lie_p(v_0,v_1,v_2)$.
\item In Section~\ref{SPoisson} we define the Poisson superalgebra $\PP(V_0,V_1,V_2)=\Poisson(V_0,V_1,V_2)$,
      determined (actually, generated) by the Lie superalgebra $\QQ$.
\item We describe monomial bases of the Poisson superalgebra $\PP$ 
      and associative hull $\AA$. 
      In case $\ch K\ne 2$, we prove that for a filtration of $\AA$, the associated graded algebra has
      a structure of a Poisson superalgebra such that
      $\gr \AA\cong\PP$, in particular, both algebras have "the same" bases.
      Also, the Poisson superalgebra $\PP$ admits an algebraic quantization using a deformed superalgebra $\AA^{(t)}$ (Section~\ref{SbasesPA}).
\item We essentially use weight functions additive on products of monomials.
      We prove that $\QQ$, $\AA$, $\PP$, $\JJ$, and $\KK$ are $\NO^3$-graded by multidegree in three generators 
      (Theorem~\ref{TZ3graduacao}, Lemma~\ref{LZ3graduacaoP}, Lemma~\ref{LZ3graduacaoJ},
       but the Jordan superalgebras have one more generator).
      This allows us to introduce coordinate systems in space:
      multidegree coordinates $(X_1,X_2,X_3)$, and twisted weight coordinates $(Y_1,Y_2,Y_3)$
      (Section~\ref{Sweight}).
\item Components of the $\NO^3$-gradation of $\QQ$ by multidegree in the generators are at most one-dimensional (Theorem~\ref{Tfine}),
      so the $\NO^3$-grading of $\QQ$ is fine.
\item $\QQ$ is just infinite but nor hereditary just infinite (Section~\ref{Sfunctions}).
\item We compute initial coefficients of generating functions of $\QQ$  (Section~\ref{Sfunctions}).
      The results and proofs on basis monomials of $\QQ$ are illustrated by Figure~\ref{Fig1}.
\item We find bounds on weights of the basis monomials of $\QQ$, $\PP$, and $\AA$
      (Sections~\ref{SweightboundsQ}, \ref{SweightboundsAP})
      and prove that images of their monomials in space
      are inside "almost cubic paraboloids" (Theorem~\ref{Tparab}, see Figure~\ref{Fig1}, and Theorem~\ref{TparabP}).
      Asymptotically, a nonzero share of lattice points inside the first paraboloid corresponds to monomials of~$\QQ$ (Corollary~\ref{Cparab}).
\item We conjecture that the superalgebras $\QQ$, $\AA$, and $\PP$ are not self-similar.
      We discuss the notion of self-similarity for Jordan superalgebras in~\cite{PeSh18Jslow}.
\item The Jordan superalgebras $\JJ$, $\KK$ are $\NO^4$-graded by multidegree in the generators (Corollary~\ref{CgradedZ4J-A}),
      we determine a hypersurface in $\R^4$ that bounds monomials of $\JJ$ and $\KK$ (Theorems~\ref{TparabJ}, \ref{TallK}).
\item $\QQ$, $\AA$, $\PP$, $\JJ$, $\KK$ have slow polynomial growth:
      $\GKdim\QQ=\GKdim\KK=\log_\lambda 2\approx 1.6518$ and
      $\GKdim\AA=\GKdim\PP=\GKdim\JJ=2\log_\lambda 2\approx  3.3036$
      (Theorems~\ref{TgrowthQ}, \ref{TgrowthP}, \ref{TgrowthJ}, \ref{TallK}).
\item $\JJ$, $\KK$ are weakly special, but not special (Corollary~\ref{Cspecial}, Theorem~\ref{TallK}).
\item $\QQ$, $\AA$, and the algebras without unit $\PP^o$, $\JJ^o$, $\KK^o$ are direct sums
      of two locally nilpotent subalgebras and there are continuum such different decompositions (Theorem~\ref{Ttwosums}).
\item $\QQ=\QQ_{\bar 0}\oplus \QQ_{\bar 1}$ is a nil graded Lie superalgebra (Theorem~\ref{Tadnilpotente}).
      Thus, $\QQ$ again shows that an extension of~Theorem~\ref{TMarZel} (Martinez and Zelmanov~\cite{MaZe99})
      for Lie superalgebras of characteristic zero is not valid.
      Such a counterexample of a nil finely $\Z^3$-graded Lie superalgebra of slow polynomial growth $\QQ$ was suggested before~\cite{Pe16}.
      There is also a recent counterexample of a nil finely $\Z^2$-graded Lie superalgebra of linear growth and of finite width 4~\cite{PeOtto}.
\item In case $\ch K=2$, $\QQ$ has a structure of a restricted Lie algebra $\QQ=\Lie_p(v_0,v_1,v_2)$
       with a nil $p$-mapping (Theorem~\ref{TnilQ2}).
\item Components of the $\NO^4$-gradation of $\KK$ by multidegree in the generators are at most one-dimensional (Theorem~\ref{TallK}),
      so the $\NO^4$-grading of $\KK$ is fine.
\item $\KK$ is just infinite but nor hereditary just infinite (Theorem~\ref{TallK}).
\item An extension of Theorem~\ref{TZelmanov} to Jordan {\it super}algebras of characteristic zero is not valid.
      Indeed, $\KK$ is a $\NO^4$-graded Jordan superalgebra with at most one-dimensional components, where
      the subalgebra without unit $\KK^o$ is nil of bounded degree.
\item  The constructions of the paper can be applied to Lie (super)algebras studied before to
       obtain Poisson and Jordan superalgebras as well.
\end{enumerate}

\begin{Remark}
Indeed, one can apply constructions of the paper to two Lie superalgebras of~\cite{Pe16}
and one more Lie superalgebra of~\cite{PeOtto}
and obtain respective associative, Poisson, and Jordan superalgebras.
But these new superalgebras shall enjoy only {\em triangular decompositions}~\eqref{decompLL} as
sums of three subalgebras, e.g. $\tilde\JJ^o=\tilde\JJ_-\oplus \tilde\JJ_0\oplus \tilde\JJ_+$,
because the roots of that characteristic polynomials are integers.
In the present paper we get decompositions into sums of two locally nilpotent subalgebras because of
nonintegral roots of the characteristic polynomial.
\end{Remark}

\begin{Remark}
In particular, recall that the Lie superalgebra $\RR$ constructed in~\cite{PeOtto}
is just infinite, two-generated, nil $\Z_2$-graded, with at most one-dimensional $\Z^2$-components, of linear growth,
moreover, of finite width~4.
Namely, its $\N$-gradation by degree in the generators has non-periodic components of dimensions $\{2,3,4\}$.
The arguments of the present paper yield the following.
Consider  the related Jordan superalgebra $\tilde\KK=\Jor(\RR)$.
Then $\tilde \KK$ is just infinite, three-generated, $\Z^3$-graded with at most one-dimensional components, 
the ideal without unit $\tilde \KK^o$ is nil of bounded degree.
Also,  $\tilde \KK$ is of linear growth, moreover, of finite width~4, 
namely, its $\N$-gradation by degree in the generators has components of dimensions $\{0,2,3,4\}$,
their sequence is non-periodic~\cite{PeSh18Jslow}.
That example also shows that just infinite $\Z$-graded Jordan superalgebras of finite width can have
a fractal complicated structure
unlike the classification of such simple algebras over an algebraically closed field of characteristic zero~\cite{KacMarZel01}.
\end{Remark}
\begin{Remark}
We continue this research in~\cite{PeSh18Jslow}, were in particular we discuss {\it self-similarity} of different types of superalgebras.
Despite that all our superalgebras look very "self-similar",
we conjecture that $\QQ$ is not self-similar in terms of the definition of Bartholdi~\cite{Bartholdi15}. 
\end{Remark}

\section{Multiplication rules of Lie superalgebra $\QQ$}\label{Srelations}

Since $\{x_i,\dd_i\mid i\ge 0\}$ are odd,
the pivot elements~\eqref{pivot} are also odd. Write them recursively:
\begin{equation}\label{pivot2}
v_i = \dd_i + x_ix_{i+1} v_{i+3},\qquad i\ge 0.
\end{equation}

Recall that we consider the Lie superalgebra
$\QQ=\Lie(v_0,v_1,v_2)\subset\WW(\Lambda)\subset \Der \Lambda$ and the associative algebra
$\AA=\Alg(v_0,v_1,v_2)\subset \End \Lambda$, where
\begin{equation}\label{aibici}
\begin{split}
v_0 &= \dd_0 + x_0x_1 v_3,\\
v_1 &= \dd_1 + x_1x_2 v_4,\\
v_2 &= \dd_2 + x_2x_3 v_5,
\end{split}
\qquad\qquad i\ge 0.
\end{equation}

Define a {\em shift} mapping $\tau:\Lambda\to \Lambda$, $\tau:\WW(\Lambda)\to\WW(\Lambda)$ by
$\tau(x_i)=x_{i+1}$,  $\tau(\dd_i)=\dd_{i+1}$,  $i\ge 0$.
Clearly, we get endomorphisms such that $\tau(v_i)=v_{i+1}$ for all $i\ge 0$.

We shall use the following basic commutation relations without special mentioning.
\begin{Lemma}\label{L_BASIC_PROD}
For all $i\ge 0$ we have:
\begin{enumerate}
\item $v_i^2= x_{i+1}v_{i+3}$;
\item $[v_i,v_i]=2 v_i^2=  2 x_{i+1}v_{i+3}$;
\item $[v_i,v_{i+1}] =-x_i v_{i+3}$;
\item $[v_i^2,v_{i+1}] =-v_{i+3}$;
\item $[v_i,v_{i+2}] = - x_ix_{i+1}x_{i+2} v_{i+5}$.
\end{enumerate}
\end{Lemma}
\begin{proof} We check the first claim
$$
v_i^2=(\dd_i + x_ix_{i+1} v_{i+3})^2=[\dd_i,x_ix_{i+1} v_{i+3} ]=x_{i+1} v_{i+3}.
$$
Now, the second claim is evident. We check claims (iii) and (iv):
\begin{align*}
[v_i,v_{i+1}]&=[\dd_i + x_ix_{i+1} v_{i+3}, \dd_{i+1} + x_{i+1}x_{i+2} v_{i+4} ]=
[x_ix_{i+1} v_{i+3}, \dd_{i+1}]=-x_iv_{i+3};\\
[v_i^2,v_{i+1}]&=[v_i,[v_i,v_{i+1}]]=[v_i,-x_iv_{i+3}]=-v_{i+3}.
\end{align*}
Finally, let us check claim (v):
\begin{align*}
[v_i,v_{i+2}]&=[\dd_i + x_ix_{i+1} \dd_{i+3} + x_ix_{i+1}x_{i+3}x_{i+4} v_{i+6}  , \dd_{i+2} + x_{i+2}x_{i+3} v_{i+5} ]\\
  & =  [x_ix_{i+1}\dd_{i+3}, x_{i+2}x_{i+3} v_{i+5}]=- x_ix_{i+1}x_{i+2} v_{i+5}.
\qedhere
\end{align*}
\end{proof}

\begin{Lemma}\label{Lproducts}
General multiplication rules for the pivot elements are as follows. Let $i,k\ge 0$.
\begin{align*}
[v_i,v_{i+3k}] &=2\bigg(\prod_{n=0}^{k-1} x_{i+3n}x_{i+3n+1}\bigg)  x_{i+3k+1}v_{i+3k+3};\\
[v_i,v_{i+3k+1}] &=-\bigg(\prod_{n=0}^{k-1} x_{i+3n}x_{i+3n+1}\bigg)  x_{i+3k}v_{i+3k+3};\\
[v_i,v_{i+3k+2}] &=-\bigg(\prod_{n=0}^{k} x_{i+3n}x_{i+3n+1}\bigg)  x_{i+3k+2}v_{i+3k+5}.
\end{align*}
\end{Lemma}
\begin{proof}
Iterating~\eqref{pivot2}, we get another presentation:
\begin{align}\nonumber
v_i=\dd_i
+ x_ix_{i+1} \dd_{i+3}+\ldots
&+ x_ix_{i+1} \hat x_{i+2} x_{i+3}x_{i+4} \hat x_{i+5}\cdots  x_{i+3k-6}x_{i+3k-5}\dd_{i+3k-3}\\
&+ x_ix_{i+1}\hat x_{i+2} x_{i+3}x_{i+4}\hat x_{i+5} \cdots   x_{i+3k-3}x_{i+3k-2}v_{i+3k},\qquad i\ge 0,\  k\ge 1.
\label{recursive_present}
\end{align}
Using presentation~\eqref{recursive_present} and Lemma~\ref{L_BASIC_PROD}, we obtain
\begin{align*}
[v_i,v_{i+3k}]
&=x_ix_{i+1}\hat x_{i+2} x_{i+3}x_{i+4}\hat x_{i+5} \cdots   x_{i+3k-3}x_{i+3k-2} [v_{i+3k},v_{i+3k}]\\
&=2x_ix_{i+1}\hat x_{i+2} x_{i+3}x_{i+4}\hat x_{i+5} \cdots   x_{i+3k-3}x_{i+3k-2}\cdot x_{i+3k+1}v_{i+3k+3};\\
[v_i,v_{i+3k+1}]
&=x_ix_{i+1}\hat x_{i+2} x_{i+3}x_{i+4}\hat x_{i+5} \cdots   x_{i+3k-3}x_{i+3k-2} [v_{i+3k},v_{i+3k+1}]\\
&=-x_ix_{i+1}\hat x_{i+2} x_{i+3}x_{i+4}\hat x_{i+5} \cdots   x_{i+3k-3}x_{i+3k-2}\cdot x_{i+3k}v_{i+3k+3};\\
[v_i,v_{i+3k+2}]
&=x_ix_{i+1}\hat x_{i+2} x_{i+3}x_{i+4}\hat x_{i+5} \cdots   x_{i+3k-3}x_{i+3k-2} [v_{i+3k},v_{i+3k+2}]\\
&=-x_ix_{i+1}\hat x_{i+2} x_{i+3}x_{i+4}\hat x_{i+5} \cdots   x_{i+3k-3}x_{i+3k-2}\cdot x_{i+3k}x_{i+3k+1}x_{i+3k+2}v_{i+3k+5}.
\qedhere
\end{align*}
\end{proof}
Consider Lie superalgebras $L_i=\Lie(v_i,v_{i+1},v_{i+2})$ for all  $i\ge 0$, so $L_0=\QQ$.
\begin{Corollary} \label{Cbases}
Let $\QQ=\Lie(v_0,v_1,v_2)$. Then
\begin{enumerate}
\item $v_i\in \QQ$, $i\ge 0$  (we get these elements using Lie bracket only in case of an arbitrary $K$);
\item $\tau^i:\QQ\to L_i$ is an isomorphism for any $i\ge 1$;
\item we get a proper chain of isomorphic subalgebras:
$$\QQ=L_0\supsetneqq L_1 \supsetneqq \cdots \supsetneqq L_i\supsetneqq L_{i+1}\supsetneqq \cdots,\qquad
\mathop{\cap}_{n=0}^\infty L_i=\{0\}.
$$
\item $\QQ$ is infinite dimensional.
\end{enumerate}
\end{Corollary}
\begin{proof}
We have $v_0,v_1,v_2\in\QQ$. By Lemma, $[v_0^2,v_1]=-v_3\in\QQ$.
Similarly, by induction we conclude that $v_i\in\QQ$ for all $i\ge 0$.
Claim (ii) follows because we have an isomorphism $\tau:\WW(\Lambda)\to\WW(\Lambda)$ such that $\tau(v_i)=v_{i+1}$, $i\ge 0$.
The intersection of $L_i$ is trivial by a description of a basis of $\QQ$ (Theorem~\ref{Tbasis3}).
\end{proof}

\section{Monomial basis of Lie superalgebra $\QQ$}

By $r_n$ denote a {\em tail} monomial:
\begin{equation}\label{rmm}
r_n=x_0^{\xi_0}\cdots x_n^{\xi_n}=x_0^*\cdots x_n^*\in \Lambda,\quad \xi_i\in\{0,1\};\ n\ge 0,
\end{equation}
where $x_i^*$ denote any power $\{0,1\}$.
If $n<0$, we consider that $r_n=1$.
Another monomials of type~\eqref{rmm} will be denoted by $r_n'$, $\tilde r_n$, etc.
Below, $\hat x_i$ denote the missing variable in a product.


We call $r_{n-3} v_n$, where $n\ge 0$, a {\em quasi-standard monomial of the first type},
and $r_{n-5} x_{n-2}v_n$, where $n\ge 2$, a {\em quasi-standard monomial of the second type}.
Among them, we exclude 24 {\em false monomials}, see below,
the remaining monomials are {\em standard monomials},
we prove that they constitute a basis of $\QQ$ in case $\ch K\ne 2$.
Let us call $n$ the {\em length}, $v_n$ the {\em head}, $r_{n-3}$ (or $r_{n-5}$) the {\em tail},
and $x_{n-2}$ the {\em neck} of a (quasi)standard monomial.

\begin{Theorem}\label{Tbasis3}
Let $\ch K\ne 2$.
A basis of the Lie superalgebra $\QQ=\Lie(v_0,v_1,v_2)$ is given by the following
{\em standard monomials} of two types
(where $r_n$ are tail monomials~\eqref{rmm})
\begin{enumerate}
\item monomials of the first type:
$$\{r_{n-3} v_n \mid n\ge 0\}
\setminus \{x_0x_1^*v_4,\, x_0x_1^*x_2^*x_3x_4^* v_7\},
$$
(i.e. in case of length 4 we exclude monomials containing $x_0$,
and in case of length 7 we exclude monomials containing both $\{x_0,x_3\}$).
We shall refer to the excluded monomials as {\em false monomials of the first type});
\item monomials of the second type:
$$
\{ x_1v_3, x_2v_4, x_3v_5\}\cup \{r_{n-5} x_{n-2}v_n \mid n\ge 6\}
\setminus \{x_0x_1^*x_2^*\, x_5 v_7,\, x_0x_1^*x_2^*x_3x_4^*x_5\, x_8v_{10}\},
$$
(i.e. in case of length 7 we exclude monomials containing $x_0$,
and in case of length 10 we exclude monomials containing all  three letters $\{x_0,x_3,x_5\}$).
We refer to the excluded monomials and $\{x_0v_2, x_0x_3v_5\}$
as {\em false monomials of the second type}.
\end{enumerate}
\end{Theorem}
\begin{proof}
A) We prove that all standard monomials belong to $\QQ$.
A1) We start with monomials of the first type.
By Corollary~\ref{Cbases},  $\{v_i\mid i\ge 0\}\subset \QQ$.
Using Lemma~\ref{L_BASIC_PROD}, $[v_0,v_1]=-x_0v_3$ and $[v_1,v_2]=-x_1v_4$ belong to $\QQ$.
Thus, all non-false monomials of the first type of length at most 4 belong to $\QQ$.
This is the base of induction. Let $n\ge 5$ and assume that
the standard monomials of the first type of length less than $n$ belong to $\QQ$.
Using claim~(v) of Lemma~\ref{L_BASIC_PROD}, we get
\begin{equation}\label{n63}
[r_{n-6} v_{n-3}, v_{n-5}]= r_{n-6} [v_{n-5},v_{n-3}]=-r_{n-6}x_{n-5}x_{n-4}x_{n-3}v_{n}\in\QQ.
\end{equation}
Multiplying by $v_{n-5}$ and (or) $v_{n-4}$, $v_{n-3}$
we can delete any subset of letters $\{x_{n-5},x_{n-4},x_{n-3}\}$ in~\eqref{n63} and obtain all monomials of the first type of length $n$.
But this argument fails when $r_{n-6} v_{n-3}$ was a false monomial.
We have two cases.

a) Consider that $r_{n-6}v_{n-3}$ above is a false monomial of the first type of length 4, so $n=7$.
By setting $r_{n-6} v_{n-3}=x_1^*v_4$ in~\eqref{n63},
we get all standard monomials of the first type of degree 7 without $x_0$.
Using $[r_{2}v_5,v_4]=-r_2\hat x_3 x_4 v_7$ and deleting $x_4$ (if necessary), we obtain all
standard monomials of the first type of degree 7 without $x_3$.

b) Let $r_{n-6}v_{n-3}$ be a false monomial of the first type of length 7, so $n=10$.
Using
$$ [x_1^*x_2^*x_3^*x_4^*v_7, x_0^*v_5]=\pm x_0^*x_1^*x_2^*x_3^*x_4^* x_5x_6x_7 v_{10}\in\QQ,$$
and deleting (if necessary) letters $x_5,x_6,x_7$
we get all standard monomials of the first type of length 10.

A2) Next, we deal with monomials of the second type.
Using (formal) squares, we get $v_{n-3}^2=x_{n-2}v_{n}\in\QQ$ for all $n\ge 3$.
In particular, we obtain all non-false standard monomials of the second type of length at most~5.
Let $n\ge 6$.
We commute monomials of the first type with the pivot elements or their squares:
$$
[r_{n-6} v_{n-3}, x_{n-5}^* v_{n-3}]=\pm r_{n-6}x_{n-5}^* [v_{n-3},v_{n-3}]
=\pm 2 r_{n-6}x_{n-5}^* x_{n-2}v_{n}\in\QQ,\quad n\ge 6.
$$
As a rule, we get all required monomials of the second type.
The arguments fail in case $r_{n-6} v_{n-3}$ is a false monomial (of the first type).
a) The case of a false monomial of the first type of length 4.
Nevertheless, using $r_{n-6} v_{n-3}=x_1^*v_4$ above, we obtain
$[x_1^*v_4,x_2^*v_4]=\pm 2 x_1^*x_2^*x_5v_7\in\QQ$,
the required standard monomials of the second type of length 7, i.e. those without $x_0$.

b) Consider that $r_{n-6} v_{n-3}$ is a false monomial of the first type of length 7.
Nevertheless, we can get the following monomials:
\begin{align*}
   [x_1^*x_2^*x_3^*x_4^* v_7, x_5^*v_7]&=\pm 2\hat x_0 x_1^*x_2^*x_3^*x_4^* x_5^*\, x_8 v_{10}\in \QQ; \\
   [x_0^*x_1^*x_2^*\hat x_3x_4^* v_7, x_5^*v_7]&=\pm 2x_0^* x_1^*x_2^*\hat x_3x_4^* x_5^*\, x_8 v_{10}\in \QQ; \\
   [x_1^*x_2^*x_3^*x_4^* v_7, x_0^*v_7]&=\pm 2 x_0^*x_1^*x_2^*x_3^*x_4^* \hat x_5\, x_8 v_{10}\in \QQ.
\end{align*}
Thus, we can obtain all monomials of the second type of length 10, i.e. those that contain at most two of
the letters $\{x_0,x_3,x_5\}$, as required.

B) We prove that products of the standard monomials are expressed via the standard monomials.
We write two standard monomials as
$a= r_{n-2}v_n$, $b=\tilde r_{m-2}v_m$ and assume that their lengths satisfy $0\le n\le m$.
B1). Let $m \equiv n\, (\,\mathrm{mod}\, 3)$.
Using presentation~\eqref{recursive_present}, we have
\begin{align}
  a&=r_{n-2}\dd_n+r_{n+1}\dd_{n+3}+\cdots+r_{m-5}\dd_{m-3}+r_{m-2}v_m; \nonumber\\
  \label{sum_above}
  [a,b] &=\Big(r_{n-2}\dd_n(\tilde r_{m-2})+r_{n+1}\dd_{n+3}(\tilde r_{m-2})
  +\cdots+r_{m-5}\dd_{m-3}(\tilde r_{m-2})\Big)v_m\\
  &\quad +r''_{m-2}[v_m,v_m].
  \label{sum_above2}
\end{align}
The last term~\eqref{sum_above2} is of the second type because $r''_{m-2}[v_m,v_m]=2r''_{m-2}v_{m+1}v_{m+3}$.
If $b$ was of the first type, namely, $b=\tilde r_{m-3}v_m$, then
all terms~\eqref{sum_above} remain of the first type.
Assume that $b$ was of the second type $b=\tilde r_{m-5}x_{m-2}v_m$, then
all terms~\eqref{sum_above} remain to be of the second type.

We need to check that~\eqref{sum_above2} cannot yield a false monomial of the second type.
Suppose the contrary and it is false of length 10, then $m=7$.
The second factor $b$ is one of three types:
$b=\hat x_0x_1^*x_2^*x_3^*x_4^* v_7$, or $b=x_0^*x_1^*x_2^*\hat x_3x_4^* v_7$,
or $b=\hat x_0x_1^*x_2^* x_5 v_7$.
Consider different possibilities for the first factor $a$.
a) Let $n=7$, then the first factor $a$ is of the same three types.
Their mutual product does not contain one of the letters $\{x_0,x_3,x_5\}$.
b) Let $n=4$. Then the first factor in~\eqref{sum_above2} comes from the last term in
$a=r_2v_4=x_1^*(\dd_4+x_4x_5v_7)$ or $a=r_2v_4=x_2(\dd_4+x_4x_5v_7)$.
The product does not contain one of $\{x_0,x_3\}$.
c) Let $n=1$.
Then the first factor in~\eqref{sum_above2} comes from the last term in $a=v_1=\dd_1+x_1x_2\dd_3+ x_1x_2\hat x_3 x_4x_5 v_7$.
Again, the product does not contain one of $\{x_0,x_3\}$.
Now, let us check that~\eqref{sum_above2} cannot be a false monomial of the second type of length 7.
Otherwise, either $b=x_1^*v_4$ or $b=x_2v_4$.
The first factor $a$ is either of the same type or the last term in $a=\dd_1+x_1x_2v_4$.
Their products lack $x_0$, as required.
Also, the false monomial $x_0x_3v_5$ cannot appear in~\eqref{sum_above2}
because in this case $m=2$ but we have only the product $[v_2,v_2]=x_3v_5$.
Moreover, we cannot obtain the false monomial $x_0v_2$.

Similarly, one needs a special check that the action on tails~\eqref{sum_above} cannot produce false monomials.
Recall that we cannot change the type, i.e. a neck remains the same.
The case of a standard monomial of length 4 is trivial.
Next, consider a standard monomial of length 7.
Let it does not contain $x_0$.
(e.g. $b=\tilde r_{m-2}v_m=x_1^*x_2^*x_3^*x_4^* v_7$.)
We are acting by monomials of length at most 4.
Observe that all standard monomials of length 4 do not contain $x_0$, thus, the action by them cannot help.
The only possibility to obtain $x_0$ is to use either $x_0v_3=x_0(\dd_3+x_3x_4 v_6)$
or $v_0=\dd_0+x_0x_1\dd_3+x_0x_1x_3x_4v_6$.
Thus, we can obtain $x_0$ at price of loosing $x_3$ and the resulting monomial is not false.
If a standard monomial of length 7 lacked $x_3$,
then the cation cannot produce $x_3$,
because we act "at most" by $+\cdots \hat x_3 \dd_4+\cdots$.
Now, consider a standard monomial of the second type of length 10, namely
$b=\tilde r_{m-2}v_{m}=x_0^*x_1^*x_2^*x_3^*x_4^*x_5^*\, x_8v_{10}$.
If it is lacking $x_5$, then the result is lacking it as well,
because for this we need to kill a senior absent letter $x_7$, (recall that the neck $x_8$ is untouchable).
Next, assume that $b$ does not contain $x_3$, we can produce it only by using $\cdots x_2x_3\dd_5+\cdots$ or $x_3v_5=x_3(\dd_5+\cdots)$,
thus loosing $x_5$.
Finally, assume that $b$ lacks $x_0$.
The action by a monomial of length 5 (i.e. $a=r_{n-2}v_5=r_{n-2}(\dd_5+x_5x_6v_8)$) deletes $x_5$.
Recall that all standard monomials of length 4 do not contain $x_0$
and all their terms do not as well.
Consider a monomial of length 3: $a=x_0v_3=x_0(\dd_3+x_3x_4\dd_5+x_3x_4 x_6x_7 v_8)$,
it can yield $x_0$ but we loose either $x_3$ or $x_5$.
Again, the standard monomials of lengths 1,2 do not contain $x_0$
and all their terms do not as well.
It remains to consider the monomial of length 0:
$v_0=\dd_0+x_0x_1\dd_3+ x_0x_1x_3x_4\dd_5+x_0x_1x_3x_4x_6x_7v_9$.
Again, we can get $x_0$ but loose either $x_3$ or $x_5$.
All these considerations also apply to the actions in the brackets of cases B2), B3) below.

B2). Let $m-n\equiv 1\, (\,\mathrm{mod}\, 3)$.
Using presentation~\eqref{recursive_present},
\begin{align*}
  a&=r_{n-2}\dd_n+r_{n+1}\dd_{n+3}+\cdots+r_{m-6}\dd_{m-4}+r_{m-3}v_{m-1}; \\
  [a,b] &=\Big(r_{n-2}\dd_n(\tilde r_{m-2})+r_{n+1}\dd_{n+3}(\tilde r_{m-2})
  +\cdots+r_{m-6}\dd_{m-4}(\tilde r_{m-2})\Big)v_m+r''_{m-2}[v_{m-1},v_m].
\end{align*}
The last term is $r''_{m-2}[v_{m-1},v_m]=-r''_{m-2}x_{m-1}v_{m+2}$, which is of the first type,
one again needs to check that it cannot be false.
Consider length 4, then $m=2$ and we have only $[v_1,v_2]=-x_1v_4$.
Consider length 7, then $m=5$ and
either $x_1^*v_4$ or $x_2v_4$ is multiplied by either $x_0^*x_1^*x_2^*v_5$ or $x_3v_5$.
The product does not contain either $x_0$ or $x_3$.

B3). Let $m-n\equiv 2\, (\,\mathrm{mod}\, 3)$.
Using presentation~\eqref{recursive_present},
\begin{align*}
  a&=r_{n-2}\dd_n+r_{n+1}\dd_{n+3}+\cdots+r_{m-7}\dd_{m-5}+r_{m-4}v_{m-2}; \\
  [a,b] &=\Big(r_{n-2}\dd_n(\tilde r_{m-2})+r_{n+1}\dd_{n+3}(\tilde r_{m-2})
  +\cdots+r_{m-7}\dd_{m-5}(\tilde r_{m-2})\Big)v_m+r''_{m-2}[v_{m-2},v_m].
\end{align*}
The last term is
$r''_{m-2}[v_{m-2},v_m]=-r''_{m-2}x_{m-2}x_{m-1}x_{m}v_{m+3}$, which is of the first type.
We check that it cannot be false.
Consider length 4, then $m=1$ and there are no such products.
Consider length 7, then $m=4$ and we have either
$[v_2,x_1^*v_4]=\pm x_1^*x_2x_3x_4v_7$ or $[v_2,x_2v_4]=v_4$.
\end{proof}

\begin{Corollary}\label{Cbasis2}
Let $\ch K=p=2$. Then
\begin{enumerate}
\item
a basis of the Lie algebra $\QQ=\Lie(v_0,v_1,v_2)$ is given by the standard monomials of the first type;
\item
a basis of the Lie superalgebra $\QQ=\Lie(v_0,v_1,v_2)$,
as well as a basis of the restricted Lie (super)algebra $\Lie_p(v_0,v_1,v_2)$,
is given by
\begin{enumerate}
\item the standard monomials of the first type;
\item squares of the pivot elements: $\{x_{n-2}v_{n}\mid n\ge 3\}$.
\end{enumerate}
\end{enumerate}
\end{Corollary}

\section{Weight functions, $\NO^3$-gradation, and three coordinate systems}
\label{Sweight}

In this section we introduce different weight functions.
Using theses functions we prove that our algebras are $\NO^3$-graded my multidegree in the generators and derive further corollaries.
We introduce three coordinate systems that allow to put monomials in space and determine their positions.

We start with the Lie superalgebra $\WW(\Lambda_I)$
of special superderivations of the Grassmann algebra
$\Lambda_I=\Lambda(x_i\mid i\in I)$
and consider a subalgebra spanned by {\it pure} Lie monomials:
$$\WW_{\mathrm{fin}}(\Lambda_I)= \langle x_{i_1}\cdots x_{i_m}\dd_j
\mid i_k,j\in I\rangle_K\subset \WW(\Lambda_I).$$
Define a {\it weight function} on the Grassmann variables and respective superderivatives related as:
$$wt(\dd_i)=-wt(x_i)=\alpha_i\in\C,\qquad i\geq 0,$$
and extend it to pure Lie monomials as
$wt(x_{i_1}\cdots x_{i_m}\dd_j)=-\a_{i_1}-\cdots-\a_{i_m}+\a_j$, $i_k,j\in I$.
One checks that the weight function is {\it additive}, namely,
$wt([w_1,w_2])=wt(w_1)+wt(w_2)$, where $w_1,w_2$ are pure Lie monomials.
The weight function is also extended to an associative hull $\Alg(\WW_{\mathrm{fin}}(\Lambda_I))$
and it is additive on associative products of its monomials.

Now we return to our algebras $\QQ=\Lie(v_0,v_1,v_2)$ and $\AA=\Alg(v_0,v_1,v_2)$.
We want to extend  a weight function on the pivot elements so that
all terms in~\eqref{pivot2} have the same weight.
Namely, we additionally assume that the weight function satisfies the equalities:
\begin{eqnarray*}
wt(v_i)=wt(\dd_i)=\alpha_{i}=-\alpha_{i}-\alpha_{i+1}+\alpha_{i+3},\qquad i\ge 0.
\end{eqnarray*}
We get a recurrence relation
\begin{equation}\label{recorrencia}
\alpha_{i+3}=\alpha_{i+1}+2\alpha_{i},\qquad i\geq 0.
\end{equation}
It has the characteristic polynomial $t^3-t-2=0$.
Using Cardano's formula, denote
$$\epsilon=e^{2/3\pi i}=\frac{-1+\sqrt 3i}2,\qquad
\theta_1=\sqrt[3]{1+\sqrt{26/27}}\approx 1.255,\qquad \theta_2=\sqrt[3]{1-\sqrt{26/27}}\approx 0.265.$$
Observe that $\theta_1\theta_2=1/3$.
One has three different roots:
\begin{align*}
t_{k}=\epsilon^k \theta_1+ \epsilon^{-k} \theta_2,\qquad k=0,1,2.
\end{align*}
Denote these roots as (we keep these notations for the whole of the paper):
\begin{align*}
\lambda&=t_0=\theta_1+\theta_2 \approx 1.5214  
,\\
\mu&=t_1= \epsilon \theta_1+\epsilon^2 \theta_2\approx -0.761+0.858i,\\   
\bar \mu&=t_2=\epsilon^2 \theta_1+\epsilon \theta_2\approx -0.761-0.858i.
\end{align*}
By Viet's formulas, one has
\begin{align*}
&\lambda+\mu+\bar\mu=0;\\
&\lambda\mu+\lambda\bar\mu+\mu\bar\mu=-1;\\
&\lambda\mu\bar\mu=2.
\end{align*}
Thus, $|\mu|=\sqrt{2/\lambda}\approx1.147$. 
The characteristic equation also yields
\begin{equation}\label{frac2}
\frac 2 \lambda=\lambda^2-1,\quad \frac 2\mu=\mu^2-1,\quad \frac 2{\bar\mu}=\bar\mu^2-1.
\end{equation}

Thus, a weight function  $wt(*)$ satisfies
$wt(\dd_n)=wt(v_n)=-wt(x_n)$, $n\ge 0$.
Moreover, by construction, all pure Lie monomials of the expansion of a pivot element~\eqref{pivot}
have the same weight as the pivot element.

Below, a {\it monomial} is
any (Lie or associative) product of the letters
$\lbrace x_i,\dd_i,v_i\mid i\geq 0\rbrace\subset \End\Lambda$.

\begin{Lemma} \label{Lpesos}
We identify weight functions with the space of solutions of recurrence equation~\eqref{recorrencia}, then
\begin{enumerate}
\item
A basis of the space of weight functions given by:
\begin{align*}
\wt(v_n)&=\lambda^n,\quad n\geq 0,\quad \text {{\rm (weight)}};\\
\swt(v_n)&=\mu^n,\quad n\geq 0,\quad \text{{\rm (superweight)}};\\
\sswt(v_n)&=\bar\mu^n,\quad n\geq 0,\quad \text{{\rm (conjugate superweight)}}.
\end{align*}
\item We replace the superweight functions by two real functions:
\begin{align*}
\wt_1(v_n)&=\Re(\mu^n)=\frac{\mu^n+\bar\mu^n}2,\quad n\geq 0;\\
\wt_2(v_n)&=\Im(\mu^n)=\frac{\mu^n-\bar\mu^n}{2i},\quad  n\geq 0.
\end{align*}
\item We combine these functions together into two {\rm vector weight functions}:
\begin{align*}
\Wt(v_n)&=\left(\wt(v_n),\swt(v_n),\sswt(v_n)\right)=(\lambda^n,\mu^n,\bar\mu^n),\quad n\geq 0,
\quad \text {{\rm (vector weight)}};\\
\WtR(v_n)&=\left(\wt(v_n),\wt_1(v_n),\wt_2(v_n)\right)=(\lambda^n,\Re(\mu^n),\Im(\mu^n)),\quad n\geq 0,
\quad \text {{\rm (twisted vector weight)}}.
\end{align*}
\item
The weight functions are well defined on monomials.
They are additive on (Lie or associative) products of monomials, e.g.,
$\Wt(a\cdot b)=\Wt(a)+\Wt(b)$, where $a,b$ are monomials of $\AA$.
\item
Let $w$ be a monomial, then
$\WtR (w)=(\wt w, \Re (\swt w),\Im(\swt w) ).$
\end{enumerate}
\end{Lemma}
\begin{proof}
Let us check the last claim.
Let $w$ be a pivot element, the equality follows by definition.
Now, the relation extends to all monomials by additivity.
\end{proof}
As a first application, we establish $\NO^3$-gradations.
\begin{Theorem}\label{TZ3graduacao}
The Lie superalgebra $\QQ=\Lie(v_0,v_1,v_2)$ and its associative hull
$\AA=\Alg(v_0,v_1,v_2)$ are $\NO^3$-graded
by multidegree in the generators $\lbrace v_0,v_1,v_2\rbrace$:
$$\QQ=\mathop{\oplus}\limits_{n_1,n_2,n_3\geq 0}\QQ_{n_1, n_2,n_3},\qquad
  \AA=\mathop{\oplus}\limits_{n_1,n_2,n_3\geq 0}\AA_{n_1, n_2,n_3}.$$
\end{Theorem}
\begin{proof}
By Lemma~\ref{Lpesos}, the generators have
the following vector weights:
$$\Wt(v_0)=(1,1,1),\quad
\Wt(v_1)=(\lambda,\mu,\bar \mu),\quad
\Wt(v_2)=(\lambda^2,\mu^2,\bar \mu^2).$$
For any $n_1,n_2,n_3\geq 0$, let $\QQ_{n_1 n_2 n_3}\subset\QQ$ be the subspace spanned
by all Lie products of multidegree $(n_1,n_2,n_3)$ in $\{v_0,v_1,v_2\}$.
By Lemma~\ref{Lpesos}, all $v\in\QQ_{n_1 n_2 n_3}$ have the same vector weight:
\begin{equation*}
\Wt(v)=n_1\Wt(v_0)+n_2\Wt(v_1)+n_3\Wt(v_2).
\end{equation*}
Elements of $\QQ_{n_1, n_2, n_3}\subset\WW(\Lambda_I)$
are infinite linear combinations of pure Lie monomials having the same vector weight.
Since $\Wt(v_0)$, $\Wt(v_1)$, $\Wt(v_2)$ are linearly independent,
different components $\QQ_{n_1, n_2,n_3}$ and $\QQ_{n_1', n_2',n_3'}$  have different vector weights,
hence their elements are expressed via different sets of pure Lie monomials.
Hence, the sum of the components  is direct.
The $\NO^3$-gradation follows by definition of these components.
\end{proof}

Given a nonzero homogeneous element
$v\in \AA_{n_1 n_2 n_3}$, $n_1,n_2,n_3\geq 0$, we define its {\it multidegree (vector)}
and a  {\it (total) degree}:
$$\Gr(v)=(n_1,n_2,n_3)\in\NO^3\subset\mathbb{R}^3,\qquad \deg(v)=n_1+n_2+n_3.$$
We put it in space using {\it standard coordinates} $(X_1,X_2,X_3)\in\mathbb{R}^3$,
which we also call {\it multidegree coordinates}.
Thus, we write $\Gr(v)=(n_1,n_2,n_3)=(X_1,X_2,X_3)$.
We also introduce complex {\it weight coordinates} $(Z_1,Z_2,Z_3)=\Wt(v)\in\C^3$
and real {\it twisted (weight) coordinates} $(Y_1,Y_2,Y_3)=\WtR(v)\in\R^3$.

Using Lemma~\ref{Lpesos}, we introduce {\it transition matrices}:
\begin{align}
\label{matrixB}
B=\Big(\Wt^T(v_0),\Wt^T(v_1),\Wt^T(v_2)\Big)&=
\begin{pmatrix}
    1 & \lambda & \lambda^2 \\
    1 & \mu   & \mu^2 \\
    1 & \bar\mu & \bar \mu^2
  \end{pmatrix};\\
\nonumber
C=\Big(\WtR{}^T(v_0),\WtR{}^T(v_1),\WtR{}^T(v_2)\Big)&=
\begin{pmatrix}
    1 & \lambda & \lambda^2 \\
    1 & \frac{\mu+\bar \mu}2   & \frac{\mu^2+\bar \mu^2}{2} \\
    0 &  \frac{\mu-\bar \mu}{2i}& \frac{\mu^2-\bar \mu^2}{2i}
  \end{pmatrix}
\approx
\begin{pmatrix}
  1 & 1.521   & 2.313 \\
  1 & -0.761  &  -0.157 \\
  0 &  0.858  &  -1.306
\end{pmatrix}.
\end{align}
\begin{Lemma}\label{Linv}
$$B^{-1}
=\begin{pmatrix}
  \frac{2/\lambda}{3\lambda^2-1}
          &\frac{2/\mu}{3\mu^2-1}
                  & \frac{2/\bar\mu}{3\bar\mu^2-1}\\
  \frac{\lambda}{3\lambda^2-1}
         &\frac{\mu}{3\mu^2-1}
                  & \frac{\bar\mu}{3\bar\mu^2-1}\\
  \frac{1}{3\lambda^2-1}
         & \frac{1}{3\mu^2-1}
                  & \frac{1}{3\bar\mu^2-1}
  \end{pmatrix}.
$$
\end{Lemma}
\begin{proof}
Using the formula of the inverse matrix, one computes the inverse of Vandermonde's matrix
(alternatively, a direct check shows that $B\cdot \text{(the matrix below)}=I$):
$$ 
B^{-1}
=\begin{pmatrix}
  \frac{\mu\bar\mu}{(\lambda-\mu)(\lambda-\bar\mu)}
          &\frac{\lambda\bar\mu}{(\mu-\lambda)(\mu-\bar\mu)}
                  & \frac{\lambda\mu}{(\bar\mu-\lambda)(\bar\mu-\mu)}\\
  \frac{-\mu-\bar\mu}{(\lambda-\mu)(\lambda-\bar\mu)}
         &\frac{-\lambda-\bar\mu}{(\mu-\lambda)(\mu-\bar\mu)}
                  & \frac{-\lambda-\mu}{(\bar\mu-\lambda)(\bar\mu-\mu)}\\
  \frac{1}{(\lambda-\mu)(\lambda-\bar\mu)}
         & \frac{1}{(\mu-\lambda)(\mu-\bar\mu)}
                  & \frac{1}{(\bar\mu-\lambda)(\bar\mu-\mu)}
  \end{pmatrix}.
$$
Using Viet's formulas and~\eqref{frac2}, we
treat the denominators in the columns above as follows:
$(\lambda-\mu)(\lambda-\bar\mu)=\lambda^2-(\mu+\bar\mu)\lambda+\mu\bar\mu
=\lambda^2-(-\lambda)\lambda+2/\lambda=3\lambda^2-1$.
\end{proof}
One also has
$$C^{-1}\approx
\begin{pmatrix}
  0.221 & 0.779   & 0.298 \\
  0.256 & -0.256  & 0.484 \\
  0.168 & -0.168  & -0.447
\end{pmatrix}.
$$

\begin{Lemma}\label{Ltrans}
Let $v\in\AA$ be a monomial
with the multidegree coordinates $\Gr(v)=(X_1,X_2,X_3)\in\NO^3$.
Let $\Wt^T(v)=(Z_1,Z_2,Z_3)$ and $\WtR{}^T(v)=(Y_1,Y_2,Y_3)$ be the respective weight and twisted coordinates. Then
\begin{enumerate}
\item  $(Y_1,Y_2,Y_3)=(Z_1,\Re Z_2,\Im Z_2)$;
\item $Z_1\in \R$ and $Z_3=\bar Z_2$;
\item  $\Wt^T(v)=B\cdot\Gr^T(v)$;
\item  $\WtR{}^T(v)=C\cdot \Gr^T(v)$.
\end{enumerate}
\end{Lemma}
\begin{proof}
The first two claims follow from Lemma~\ref{Lpesos}.
By assumption, $v$ is a product that involves $X_1$ factors $v_0$,  $X_2$ factors $v_1$, and  $X_3$ factors $v_2$.
We check the last two claims using additivity and~\eqref{matrixB}
\begin{equation*}
\Wt^T(v)=X_1\Wt^T(v_0)+X_2\Wt^T(v_1)+X_3\Wt^T(v_2)
=B\cdot
\begin{pmatrix} X_1\\X_2\\X_3 \end{pmatrix}.
\qedhere
\end{equation*}
\end{proof}
\begin{Corollary}
Let $(X_1,X_2,X_3)\in\mathbb{R}^3$ be a point of space in standard coordinates.
We introduce its {\it weight coordinates} $(Z_1,Z_2,Z_3)$ and {\it twisted coordinates} $(Y_1,Y_2,Y_3)$
using formulas of Lemma.
\end{Corollary}

Consider the axis $OY_1\subset\R^3$ which is determined by $Y_2=Y_3=0$ in terms of the twisted coordinates.
\begin{Lemma}\label{L_OY1}
The axis $OY_1$ is determined by the vector $(2/\lambda,\lambda,1)$ in terms of the standard coordinates.
\end{Lemma}
\begin{proof}
Since, $Z_2=Y_2+iY_3$, the condition $Y_2=Y_3=0$ is equivalent to $Z_2=\bar Z_3=0$.
We take $\Wt(v)=(1,0,0)$ and
use claim (iii) of Lemma~\ref{Ltrans} and Lemma~\ref{Linv}. The axis $OY_1$ is determined by the vector:
\begin{equation*}
\Gr^T(v)=
B^{-1} \begin{pmatrix} 1 \\ 0\\ 0\end{pmatrix}
=\frac 1 {3\lambda^2-1}\begin{pmatrix} 2/\lambda \\ \lambda \\ 1 \end{pmatrix}.
\qedhere
\end{equation*}
\end{proof}

\begin{Lemma}\label{Laxis}
The axis $OY_1$ does not contain the lattice points $\Z^3\subset\R^3$
in terms of the standard coordinates $(X_1,X_2,X_3)$, except the origin $O=(0,0,0)$.
\end{Lemma}
\begin{proof}
Consider a lattice point $O\ne (n_1,n_2,n_3)=A\in \Z^3\subset \R^3$.
Assume that $A$ belongs to $OY_1$. Then $(n_1,n_2,n_3)=r (2/\lambda,\lambda,1)$ for some $r\in\R$.
Hence $\lambda=n_2/n_3\in\Q$, a contradiction with irrationality of $\lambda$.
\end{proof}

\begin{Lemma}
Let $\sigma=\log_{|\mu|}\lambda \approx 3.068$.
The pivot elements $\{v_n\mid n\ge 0\}$
belong to a paraboloid-like surface with equation in twisted coordinates:
$$Y_1=(Y_2^2+Y_3^2)^{\sigma/2}. $$
\end{Lemma}
\begin{proof}
By Lemma~\ref{Lpesos},
$\Wt(v_n)=(Z_1,Z_2,Z_3)=(\lambda^n,\mu^n,\bar \mu^n)$ and
$\WtR(v_n)=(Y_1,Y_2,Y_3)=(Z_1,\Re Z_2, \Im Z_2)$, $n\ge 0$.
Then
\begin{align*}
&Y_2^2+Y_3^2=|Z_2|^2=|\mu|^{2n};\\
&Y_1=Z_1=\lambda^n=\lambda^{1/2\log_{|\mu|}(Y_2^2+Y_3^2) }=(Y_2^2+Y_3^2)^{1/2\log_{|\mu|}\lambda}.
\qedhere
\end{align*}
\end{proof}

\begin{Lemma}
The multidegree coordinates of the pivot elements $\Gr(v_n)=(X_1,X_2,X_3)$  are as follows:
\begin{align*}
\begin{pmatrix} X_1 \\ X_2 \\ X_3 \end{pmatrix}
=\frac{\lambda^n}{3\lambda^2-1}
    \begin{pmatrix} 2/\lambda \\ \lambda \\ 1 \end{pmatrix}
 +\frac{\mu^n}{3\mu^2-1}
    \begin{pmatrix} 2/\mu \\ \mu \\ 1\end{pmatrix}
 + \frac{\bar\mu^n}{3\bar\mu^2-1}
    \begin{pmatrix} 2/\bar\mu \\ \bar\mu \\ 1\end{pmatrix},\quad n\ge 0.
\end{align*}
\end{Lemma}
\begin{proof}
We use Lemma~\ref{Ltrans}, Lemma~\ref{Lpesos}, and Lemma~\ref{Linv}:
\begin{align*}
\Gr^T(v_n)&=B^{-1}\Wt^T(v_n)
          =B^{-1}\begin{pmatrix}\lambda^n \\ \mu^n \\ \bar\mu^n\end{pmatrix}\\
&=\frac{\lambda^n}{3\lambda^2-1}
    \begin{pmatrix} 2/\lambda \\ \lambda \\ 1 \end{pmatrix}
 +\frac{\mu^n}{3\mu^2-1}
    \begin{pmatrix} 2/\mu \\ \mu \\ 1\end{pmatrix}
 + \frac{\bar\mu^n}{3\bar\mu^2-1}
    \begin{pmatrix} 2/\bar\mu \\ \bar\mu \\ 1\end{pmatrix}.\qedhere
\end{align*}
\end{proof}
\begin{Corollary}
The total degrees of the pivot elements in the generators $\{v_0,v_1,v_2\}$ are as follows
$$
\deg(v_n)=\frac{4\lambda^2+\lambda+6}{26}\lambda^n+
\frac{4\mu^2+\mu+6}{26}\mu^n+ \frac{4\bar\mu^2+\bar\mu+6}{26}\bar\mu^n,\quad n\ge 0.
$$
\end{Corollary}
\begin{proof}
By definition of the degree, $\deg(v_n)=X_1+X_2+X_3$,
the sum of the multidegree coordinates, the latter are computed in Theorem.
By~\eqref{frac2},  $2/\lambda+\lambda+1=\lambda^2+\lambda$.
A direct check in the field $\Q[\lambda]$ shows that
$$
\frac{\lambda^2+\lambda}{3\lambda^2-1}=\frac{4\lambda^2+\lambda+6}{26}.
$$
The same computations are valid for the remaining roots $\mu,\bar\mu$.
\end{proof}

\section{$\QQ$ is finely $\NO^3$-graded and just infinite,
its generating functions}\label{Sfunctions}\label{Sfine}

The second example in~\cite{Pe16} yields a $\Z^3$-graded Lie superalgebra with
at most one-dimensional components.
Similarly, the Lie superalgebra constructed in~\cite{PeOtto} is $\Z^2$-graded with at most one-dimensional components.
Now, we establish a similar fact, that components of the multidegree $\Z^3$-grading of the Lie superalgebra $\QQ$
are at most one-dimensional  (Theorem~\ref{Tfine}).
This implies that the $\Z^3$-grading of $\QQ$ is {\it fine} (see definitions in~\cite{Eld10}).
We also prove that $\QQ$ is just infinite but nor hereditary just infinite,
the same properties were established for the example~\cite{PeOtto}.
At the end of this section we supply computations of generating functions for $\QQ$.
Figure~\ref{Fig1} below gives a geometric illustration of the results and proofs of the paper.

\begin{Lemma}\label{Lend}
Let $\tau:\AA\to \AA$ be the shift endomorphism.
Consider a multihomogeneous element $0\neq v\in\AA$ with $\Gr(v)=(n_1,n_2,n_3)$, $n_1,n_2,n_3\ge 0$.
Then
$$ \Gr(\tau(v))=(2n_3,n_1+n_3,n_2). $$
\end{Lemma}
\begin{proof}
The relation $[v_0^2,v_1]=-v_3$ implies that $\Gr(v_3)=(2,1,0)$.
By assumption,
$v$ is a linear combination of products involving
$n_1,n_2,n_3$ factors $v_0,v_1,v_2$, respectively.
Since $\tau$ is an endomorphism,
$\tau(v)$ is a linear combination of products involving $n_1$ factors
$\tau(v_0)=v_1$, $n_2$ factors $\tau(v_1)=v_2$, and  $n_3$ factors $\tau(v_2)=v_3$.
Using additivity of the multidegree function, we get
\begin{equation*}
\Gr(\tau(v))=n_1\Gr(v_1)+n_2\Gr(v_2)+n_3\Gr(v_3)=n_1(0,1,0)+n_2(0,0,1)+n_3(2,1,0)
=(2n_3,n_1+n_3,n_2).
\qedhere
\end{equation*}
\end{proof}

\begin{Theorem}\label{Tfine}
Components of the $\NO^3$-gradation
$\QQ=\mathop{\oplus}\limits_{n_1,n_2,n_3\geq 0}\QQ_{n_1, n_2,n_3}$
by multidegree in the generators $\lbrace v_0,v_1,v_2\rbrace$
(Theorem~\ref{TZ3graduacao}) are at most one-dimensional.
\end{Theorem}
\begin{proof}
Recall that the standard monomials and the false monomials are the { quasi-standard monomials}.
Let us list all quasi-standard monomials of length at most 4,
of the first type:
$\{v_0,v_1,v_2,x_0^*v_3,x_0^*x_1^*v_4 \}$,
and of the second type: $\{x_0v_2,x_1v_3,x_2v_4\}$.
We shall prove a more general fact, namely,
that different quasi-standard monomials have different multidegrees.

We make an observation.
Let $v$ be a quasi-standard monomial.
We can present it as $v=x_0^\alpha \tau(v')$, where $\alpha\in\{0,1\}$,
$\tau$ the shift endomorphism, and
$v'$ is a quasi-standard monomial of length less by one (and of the same type as a rule).
There is one exception: $v=v_0$, let us treat it now.
One has the multidegree $\Gr(v_0)=(1,0,0)$.
So, the standard monomials with the same multidegree
must contain the only factor $v_0$.
Thus, the only standard monomial of the same multidegree is $v_0$.
It remains to compare with multidegrees of the false monomials.
(First, consider the false monomials of small length. We have $\Gr(x_0v_2)=(-1,0,2)$;
using $[v_1^2,v_2]=-v_4$ we get $\Gr(x_0v_4)=(-1,2,1)$ and $\Gr(x_0x_1v_4)=(-1,3,1)$;
since $v_2^2=x_3v_5$ we get $\Gr(x_0x_3v_5)=(-1,0,2)$.
Let $v$ be a false monomial of length $n\ge 5$.
Being of the same multidegree implies that it has the same weight.
But by Corollary~\ref{Cpesos}, $\wt(v)> \lambda^{n-5}\ge 1=\wt v_0$.)

By way of contradiction, assume that
$u\ne v$ are quasi-standard monomials of the same multidegree, i.e. $\Gr(u)=\Gr(v)$.
Also, assume that in this counterexample the minimum of lengths of $u,v$ is the minimum possible.
By the observation above, $u=x_0^{\alpha}\tau(u')$ and $v=x_0^{\beta}\tau(v')$,
where $u',v'$ are quasi-standard monomials of the same types and of lengths less by one and $\a,\b\in\{0,1\}$.
Let $\Gr(u')=(n_1,n_2,n_3)$ and $\Gr(v')=(m_1,m_2,m_3)$.
Since $\Gr(x_0)=(-1,0,0)$, using Lemma~\ref{Lend}, we have
$$\Gr(u)=(2n_3-\alpha,n_1+n_3,n_2)=(2m_3-\beta,m_1+m_3,m_2)=\Gr(v).$$
Since $\alpha,\beta\in\{0,1\}$ we conclude that $n_3=m_3$ and $\alpha=\beta$, then
also $n_1=m_1$ and $n_2=m_2$.
Hence, $\Gr(u')=\Gr(v')$. By minimality of the example, $u'=v'$.
Therefore, $u=v$, a contradiction.
\end{proof}
\begin{Corollary}\label{Cweights_different}
Let $u,w$ be standard monomials of $\QQ$ such that $\wt u=\wt w$. Then $u=w$.
\end{Corollary}
\begin{proof}
Consider respective multidegrees and
assume that $\Gr u=(n_1,n_2,n_3)\ne \Gr w=(m_1,m_2,m_3)$.
Then $\wt u= n_1+ n_2\lambda+ n_3 \lambda^2 =\wt w= m_1+m_2 \lambda + m_3 \lambda^2 $ and
$(m_3-n_3)\lambda^2+(m_2-n_2)\lambda +(m_1-n_1)=0$,
a contradiction with the fact that $\lambda$ satisfies an irreducible polynomial of degree 3.
Hence, $\Gr u=\Gr w$. By Theorem, $u=w$.
\end{proof}

\begin{Theorem}
The Lie superalgebra $\QQ$ is just infinite.
\end{Theorem}
\begin{proof}
Let $I$ be a nonzero ideal of $\QQ$ and $0\ne a\in I$.
By Corollary~\ref{Cweights_different},
\begin{equation}\label{anu}
a=\nu_1 w_1+\cdots +\nu_m w_m\in I,\quad 0\ne \nu_j\in K,\ \text{ $w_j$ are standard monomials}, \ \wt w_1<\cdots<\wt w_m.
\end{equation}
Let us prove by induction on $m$ that some pivot element belongs to $I$.
We shall multiply~\eqref{anu} by monomials, the senior term $w_m$ will be transformed into a senior term,
we shall keep its coefficient nonzero.
By Theorem~\ref{Tfine},
the terms move to different at most one-dimensional multihomogeneous components.
Hence, we get a similar decomposition~\eqref{anu} with the same (or smaller) number of terms  $m$.
Consider the senior term $w_m=x_{i_1}\cdots x_{i_k} v_n$, $i_1<\cdots <i_k$.
Then $[v_{i_k},\ldots, v_{i_1},w_m]=v_n$ is a senior term of $[v_{i_k},\ldots, v_{i_1},a]$.
Thus, we get a pivot element $w'_{m'}=v_n$ in~\eqref{anu}.
If $m'=1$, the base of induction is proved.

By our arguments, we can assume that the senior term in~\eqref{anu} is $w_m=v_n$.
Using $[v_{n-1},v_{n-1},v_n]=-v_{n+2}$, we can make $n$ arbitrary big.
Since we always multiply by homogeneous monomials,
we either keep the following difference
or it is even getting smaller in case the smallest term disappear:
$\wt w'_{m'}-\wt w_1'\le \wt w_m-\wt w_1=C$.
Thus, $\wt w_1'\ge \wt v'_{m'}-C=\lambda^n-C$,
and the last number exceeds $\lambda^{n-1}=\wt v_{n-1}$ for sufficiently large $n$.
Hence, we can consider that all standard monomials in~\eqref{anu} are of length at least $n$.
On the other hand, using $\wt w_j \le \wt w_m=\wt v_n=\lambda^n$
and the lower estimates of Lemma~\ref{Lestimativas}, we can have
only standard monomials of the first type of length at most $n+4$ and
of the second type of length at most $n+3$.
Take a standard monomial
$w=r_{k-2}v_k$, where $n\le k\le n+5$, of our decomposition~\eqref{anu} and assume that it has a factor $x_i$ where $i<n$.
Then $x_iv_{n+2}\in\QQ$ is a standard monomial of the first type and
we get a new senior term $[x_iv_{n+2},v_n]=-x_i x_nx_{n+1}x_{n+2}v_{n+5}\ne 0$ (Lemma~\ref{L_BASIC_PROD})
while $[x_iv_{n+2},w]=\pm [x_iv_{n+2},x_ir'_{k-2}v_k]=0$,
thus reducing the number of monomials, and we apply the inductive assumption.

It remains to consider a few standard monomials with restrictions on lengths above having no factors $x_i$, $i<n$.
We compute their weights,
monomials of the second type:
$\wt(x_{n+1}v_{n+3})=\wt(v_n^2)=2\lambda^n$,
$\wt(x_{n}v_{n+2})=2\lambda^{n-1}$
and of the first type:
$\wt(x_n^*v_{n+3})\ge \lambda^n(\lambda^3-1)=\lambda^{n}(\lambda+1)$,
$\wt (x_{n}^*x_{n+1}^*v_{n+4})\ge \lambda^n(\lambda^4-\lambda-1)=\lambda^n(\lambda^2+\lambda-1)>\lambda^n$,
and $\{v_{n+2},v_{n+1},v_n\}$.
These monomials except $v_n$ cannot appear because their weights exceed
the weight of the senior term $\wt w_m=\wt v_n=\lambda^n$.
Hence, our decomposition  consists of a unique pivot element $v_n$.

Thus, we have $v_N\in \QQ$ for a large integer $N$. By Lemma~\ref{L_BASIC_PROD},
$[v_{N-1}^2,v_{N}] =-v_{N+2}\in I$ and
$b=[v_{N-2},v_{N}] = - x_{N-2}x_{N-1}x_{N} v_{N+3}\in I$,
$[v_N, v_{N-1},v_{N-2},b]=-v_{N+3}\in I$.
By induction, we derive that $v_k\in I$ for $k\ge N+2$.
Fix $k\ge N+2$, using claim~(v) of Lemma~\ref{L_BASIC_PROD}, we get
\begin{equation*}
[r_{k-1} v_{k+2}, v_{k}]= r_{k-1} [v_{k+2},v_{k}]=-r_{k-1}x_{k}x_{k+1}x_{k+2}v_{k+5}\in I.
\end{equation*}
Multiplying by $v_{k}$ and (or) $v_{k+1}$, $v_{k+2}$
we get all standard monomials of the first type of length $k+5\ge N+7$.
Using (formal) squares, we get $v_{n-3}^2=x_{n-2}v_{n}\in I$ for all $n\ge N+5$.
In case $\ch K\ne 2$ we  also get
$$
[r_{n-6} v_{n-3}, x_{n-5}^* v_{n-3}]
=\pm r_{n-6}x_{n-5}^* [v_{n-3},v_{n-3}]
=\pm 2 r_{n-6}x_{n-5}^* x_{n-2}v_{n}\in I,\quad n\ge N+10.
$$
We proved that $I$ contains all basis monomials of lengths $n\ge N+10$.
Therefore, $\dim \QQ/I$ is bounded by a finite number of basis monomials of length at most $N+9$.
\end{proof}

\begin{Lemma}\label{Lnonjustinf}
The Lie superalgebra $\QQ$ is not hereditary just infinite.
\end{Lemma}
\begin{proof}
Fix $m\ge 1$.
Let $\QQ(m)\subset\QQ$ be the linear span of its basis monomials of length at least $m$,
so $v_0\notin \QQ(m)$.
By multiplication rules (see Section~\ref{Srelations}), $\QQ(m)$ is an ideal of $\QQ$.
Observe that the ideal $\QQ(m)\subset \QQ$ has a finite codimension.
In particular, $\QQ=\langle v_0\rangle \oplus \QQ(1)$ and $\dim\QQ/\QQ(1)=1$.
Let $J=x_0\QQ(m)$ be the subspace of $\QQ(m)$ spanned by its basis monomials involving $x_0$.
Since $x_0$ can be deleted only by $v_0$ that does not belong to $\QQ(m)$,
we see that $J$ is an abelian ideal of $\QQ(m)$.
Since $v_i\in \QQ(m){\setminus} J$ for all $i\ge m$, we conclude that
$\dim \QQ(m)/J=\infty$ and the ideal $\QQ(m)$ is not just infinite.
\end{proof}

Let $A=\oplus_{n,m,k} A_{nmk}$ be a $\Z^3$-graded algebra,
one has an induced $\Z$-gradation: $A=\mathop{\oplus}\limits_n  A_{n}$, where $A_l=\mathop{\oplus}\limits_{n+m+k=l} A_{nmk}$.
Define respective {\em generating functions}:
\begin{align*}
\H(A,t_1,t_2,t_3)&=\sum_{n,m,k} \dim A_{nmk} t_1^{n}t_2^m t_3^k;\\
\H(A,t)&=\sum_{n} \dim A_{n} t^n=\H(A,t,t,t).
\end{align*}
Using somewhat recursive structure of the basis of $\QQ$ (Theorem~\ref{Tbasis3}),
computer calculations yield the following series:
\begin{align*}
\H(\QQ,t_1,t_2,t_3)&=
t_1+t_2+t_3+
t_1^2+t_1t_2+t_1t_3+t_2^2+t_2t_3+t_3^2\\
&+t_1^2t_2+ t_1^2t_3+ t_1t_2^2+ t_1t_2t_3+ t_1t_3^2+ t_2^2t_3+ t_2t_3^2\\
&+t_1^3t_2+ t_1^2t_2^2+ t_1^2t_2t_3+ t_1^2t_3^2+ t_1t_2^3+ t_1t_2^2t_3+ t_1t_2t_3^2+ t_1t_3^3+ t_2^3t_3+ t_2^2t_3^2+ t_2t_3^3\\
&+t_1^4t_2+ t_1^3t_2^2+ t_1^3t_2t_3+ t_1^2t_2^3+ t_1^2t_2^2t_3+ t_1^2t_2t_3^2\\
&\ \ +t_1^2t_3^3+ t_1t_2^3t_3+ t_1t_2^2t_3^2+ t_1t_2t_3^3+ t_1t_3^4+ t_2^4t_3+ t_2^3t_3^2+ t_2^2t_3^3\\ 
&+t_1^4t_2^2+ t_1^4t_2t_3+ t_1^3t_2^3+ t_1^3t_2^2t_3+ t_1^3t_2t_3^2+ t_1^2t_2^3t_3+ t_1^2t_2^2t_3^2\\
&\ \ + t_1^2t_2t_3^3+ t_1^2t_3^4+ t_1t_2^4t_3+t_1t_2^3t_3^2+ t_1t_2^2t_3^3+ t_1t_2t_3^4+ t_2^4t_3^2+ t_2^3t_3^3+\cdots\\
\H(\QQ,t)&=
3t+6t^2+7t^3+11t^4+14t^5+15t^6+17t^7+18t^8+21t^9+25t^{10}\\
&+25t^{11}+26t^{12}+30t^{13}+32t^{14}+33t^{15}+35t^{16}+35t^{17}+35t^{18}+38t^{19}+39t^{20}\\
&+38t^{21}+38t^{22}+39t^{23}+43t^{24}+44t^{25}+42t^{26}+47t^{27}
+51t^{28}+50t^{29}+53t^{30}+\cdots
\end{align*}

\begin{figure}[h]
\caption{Three small read vectors at origin are generators $v_0,v_1,v_3$.
Dots show standard monomials of $\QQ$
(first type -- green, second -- blue).
Pivot elements are red, marked by red dashed arrows, and belong to small "paraboloid".
Two "paraboloids" are cut by plane of fixed weight:}
\label{Fig1} 
\begin{center}
\epsfig{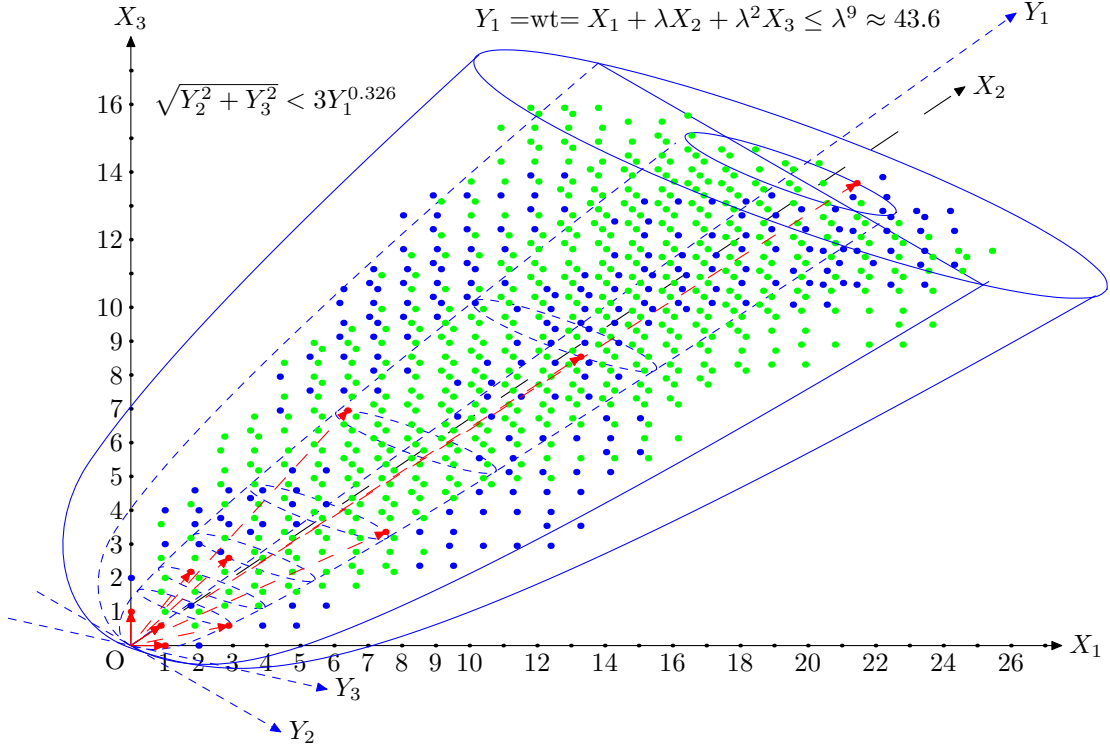}
\end{center}
\end{figure}

\section{Bounds on weights, growth, and paraboloid for Lie superalgebra $\QQ$}\label{SweightboundsQ}

In this section, we establish estimates on weights and superweights of
standard monomials of the Lie superalgebra $\QQ$.
Using these estimates we specify the growth of $\QQ$ (Theorem~\ref{TgrowthQ}) and
prove that the standard monomials are situated in a region of space restricted by
a surface of rotation close to a cubic paraboloid (Theorem~\ref{Tparab}, see Fig.~\ref{Fig1}).
Below, $\lambda$, $\mu$ are the roots of the characteristic polynomial (Section~\ref{Sweight}).
\begin{Lemma}\label{Lestimativas}
We have estimates for weights of the quasi-standard monomials of the first and second type:
\begin{align*}
1.3 \lambda^{n-5} &<\wt(r_{n-3}v_n)\leq \lambda^n,\qquad\qquad\ \, n\geq 0;\\
1.1 \lambda^{n-4} &<\wt(r_{n-5}x_{n-2}v_n)\leq 2\lambda^{n-3},\quad n\geq 2.
\end{align*}
\end{Lemma}
\begin{proof}
One checks that $(\lambda-1)^{-1}=(\lambda^2+\lambda)/2$.
The upper bound $\wt(r_{n-3}v_n)\leq \lambda^n$, $n\ge 0$, is trivial.
First, consider a tail
$r_m=x_0^{\xi_0}\cdots x_m^{\xi_m}$, $\xi_i\in\lbrace 0,1\rbrace,$ and find a bound on its weight:
\begin{equation}\label{desigualdades}
\wt(r_m)\geq -(\lambda^0+\lambda^1+\cdots +\lambda^m)
>-\frac{\lambda^{m+1}}{\lambda-1}
=-\frac{\lambda^{m+1}(\lambda^2+\lambda)}2=-\frac{\lambda^{m+3}+\lambda^{m+2}}2,\quad m\geq 0.
\end{equation}
This bound is formally valid for $m=-1,-2,-3$. Using~\eqref{desigualdades}, we get:
$$\wt(r_{n-3}v_n)> -\frac{\lambda^{n}+\lambda^{n-1}}2+\lambda^n=\frac{\lambda^{n-1}(\lambda-1)}2>1.3\lambda^{n-5}, \qquad n\geq 0, $$
because $\lambda^4(\lambda-1)/2\approx 1.39$.
For monomials of the second type, one has an upper bound
$$\wt(r_{n-5}x_{n-2}v_n)\le \lambda^n-\lambda^{n-2}=\lambda^{n-3}(\lambda^3-\lambda)
=\lambda^{n-3}(\lambda+2-\lambda)=2\lambda^{n-3},\quad n\ge 2.$$
Using~\eqref{desigualdades} and $\lambda(3-\lambda)/2\approx 1.12$, we check the lower bound for monomials of the second type:
\begin{align*}
\wt(r_{n-5}x_{n-2}v_n)
&>-\frac{\lambda^{n-2}+\lambda^{n-3}}2-\lambda^{n-2}+\lambda^n
=\lambda^{n-3}(-3/2 \lambda-1/2+\lambda^3)\\
&=\lambda^{n-3}(-3/2 \lambda-1/2+\lambda+2)
=\lambda^{n-3}\frac{3-\lambda}2>1.1\lambda^{n-4},
\qquad\quad n\ge 2.\qedhere
\end{align*}
\end{proof}
\begin{Corollary}\label{Cpesos}
Let $w$ be a quasi-standard monomial of length $n\geq 0$. Then
$$\lambda^{n-5} <\wt w\leq \lambda^n. $$
\end{Corollary}

\begin{Lemma}\label{Lsuperpeso}
Let $w$ be a quasi-standard monomial of length $n\geq 0$. Then
$|\swt w|< 7 |\mu|^n$.
\end{Lemma}
\begin{proof}
Write monomials of both types as $w=r_{n-2}v_n$, $n\ge 0$. Below, we use that $|\mu|\approx 1.14656$:
\begin{equation*}
|\swt w|=|\swt(r_{n-2}v_n)|
\le |\mu|^n+\sum_{i=0}^{n-2}|\mu|^{i}
< |\mu|^n+\frac{|\mu|^{n-2}}{1-1/|\mu|}=|\mu|^n\bigg(1+\frac1{|\mu|^2-|\mu|}\bigg)< 7|\mu|^n.\qedhere
\end{equation*}
\end{proof}
\begin{Theorem}\label{TgrowthQ}
Consider the Lie superalgebra $\QQ=\Lie(v_0,v_1,v_2)$ over an arbitrary field  $K$.  Then
$$\GKdim\QQ=\LGKdim\QQ=\log_\lambda 2\approx 1.6518.$$ 
\end{Theorem}
\begin{proof}
Let us find an upper bound on the weight growth function
$\tilde\gamma_\QQ(m)$ which counts standard monomials $w$ such that  $\wt w\le m$, where $m\ge 1$.
Consider such a monomial $w$ of length $n$.
By Corollary~\ref{Cpesos}, $\lambda^{n-5}<\wt w\le m$, hence $n\le n_0=[\log_\lambda m]+5$.
Counting standard monomials of both types of length at most $n_0$, we get a desired upper bound
$$
\tilde\gamma_\QQ(m)
\le 3+\sum_{n=2}^{n_0}2^{n-2}+ \sum_{n=4}^{n_0}2^{n-4}
<3+2^{n_0-1}+2^{n_0-3}<2^{n_0}\le 2^{\log_\lambda m+5}=32m^{{\log_\lambda}2}.
$$
Fix $m\ge 1$. Set $n=[\log_\lambda m]$.
Consider all monomials $w=r_{n-3}v_{n}$ of the first type of length $n$.
By Corollary~\ref{Cpesos}, $\wt w\le \lambda^n\le m$. Counting all such monomials we get a lower bound
in case of any characteristic:
\begin{equation*}
\tilde\gamma_\QQ(m)\ge 2^{n-2}\ge 2^{\log_\lambda m-3}=2^{-3}m^{\log_\lambda 2}.
\qedhere
\end{equation*}
\end{proof}

\begin{Theorem}  \label{Tparab}
Put $\sigma=\log_{|\mu|}\lambda \approx 3.068$.
The points of space depicting the (quasi)standard monomials of the Lie superalgebra $\QQ$
are inside an "almost cubic paraboloid", which equation is written in terms of
the twisted coordinates $\WtR(w)=(Y_1,Y_2,Y_3)$:
$$ \sqrt{Y_2^2+Y_3^2}< 14 \sqrt[\strut\sigma]{Y_1}. $$
\end{Theorem}
\begin{proof}
Let $w$ be a standard monomial of $\QQ$ of length $n\geq 0$
and the weight coordinates $\Wt(w)=(Z_1,Z_2,Z_3)=(\wt w,\swt w, \sswt w)$.
By Corollary~\ref{Cpesos}, $\lambda^{n-5}<\wt w=Z_1$, thus $n<\log_\lambda Z_1+5$.
By Lemma~\ref{Lsuperpeso}, we get
\begin{equation*}
|Z_2|=|\swt w |< 7|\mu|^n< 7|\mu|^{\log_{\lambda} Z_1+5}
=7|\mu|^5 Z_1^{\log_\lambda |\mu|}<14 Z_1^{1/\sigma},
\end{equation*}
using $7|\mu|^5\approx 13.89<14$.
Applying Lemma~\ref{Ltrans}, we get a transition to the twisted coordinates
\begin{equation*}  
Y_2^2+Y_3^2 =(\Re Z_2)^2+(\Im Z_2)^2=|Z_2|^2<196 Z_1^{2/\sigma}=196 Y_1^{2/\sigma}.
\qedhere
\end{equation*}
\end{proof}

Figure~\ref{Fig1} shows a paraboloid but with a smaller constant 3.
We have a weaker bound.
\begin{Corollary}
The monomials of $\QQ$ are inside of a cubic paraboloid:
$$ \sqrt{Y_2^2+Y_3^2}< 14 \sqrt[3]{Y_1}. $$
\end{Corollary}
\begin{proof}
We use that  $\sigma>3$ and $Y_1=Z_1=\wt w\ge 1$ for any standard monomial $w$.
\end{proof}
\begin{Corollary}\label{Cparab}
Consider the "almost cubic paraboloid" of Theorem.
\begin{enumerate}
\item
The volume of a part of the paraboloid cut by plane $Y_1\le m$ is equal to
$$ \mathrm{Volume}(m)=\mathrm{Const}\cdot m^{\log_\lambda 2},\quad m\ge 1. $$
\item
Asymptotically, a nonzero share of lattice points inside the paraboloid
corresponds to monomials of~$\QQ$.
\end{enumerate}
\end{Corollary}
\begin{proof}
We have a figure of rotation: $0\le Y_1\le m$, $\sqrt{Y_2^2+Y_3^2}\le R(Y_1)=14 Y_1^{1/\sigma}$ with
a volume:
$$
\mathrm{Volume}_{(Y_1,Y_2,Y_3)}(m)=\int_0^m \pi R^2(y)dy
=\int_0^m \pi 196 y^{2/\sigma}=\frac{196\pi}{1+2/\sigma}m^{1+2/\sigma},
$$
where $1+2/\sigma=1+2\log_\lambda |\mu|=\log_\lambda(\lambda|\mu|^2)=\log_\lambda 2$, by Viet's formulas.
Compute volume of the same figure in terms of the standard coordinates:
$\mathrm{Volume}_{(X_1,X_2,X_3)}(m)=\mathrm{Volume}_{(Y_1,Y_2,Y_3)}(m)/\det C$,
because $C$ makes a transition between these coordinates.
A number of lattice points  $\Z^3$ (in the standard coordinates) inside the figure is asymptotically equal to
$\mathrm{Volume}_{(X_1,X_2,X_3)}(m)$.

On the other hand,  by the proof of Theorem~\ref{TgrowthQ} on the growth of $\QQ$,
we have a lower polynomial bound with the same degree.
Recall that the multihomogeneous components of $\QQ$ are at most one-dimensional (Theorem~\ref{Tfine}).
Thus, a nonzero share of the lattice points inside the paraboloid correspond to monomials of~$\QQ$.
\end{proof}

\section{Poisson superalgebra $\PP$}\label{SPoisson}

In this section, we define a Poisson superalgebra $\PP(V_0,V_1,V_2)$,
determined (actually, generated) by the Lie superalgebra $\QQ=\Lie(v_0,v_1,v_2)$.

Recall our basic construction. We take the Grassmann algebra $\Lambda=\Lambda(x_i| i\ge 0)$ and
consider its generators and respective superderivatives
$\{x_i, \dd_i\mid i\ge 0 \}\subset \End_{\bar 1}\Lambda$.
They satisfy  the commutation relations:
\begin{equation}\label{relv}
[\dd_i,x_j]=\delta_{ij},\quad
[x_i,x_j]=[\dd_i,\dd_j]=0,\qquad i,j\ge 0.
\end{equation}
Next, we defined the pivot elements:
\begin{equation}\label{pivotv}
v_i = \dd_i + x_ix_{i+1}(\dd_{i+3} + x_{i+3}x_{i+4}(\dd_{i+6} + x_{i+6}x_{i+7}(\dd_{i+9}+\cdots) ))\in \Der\Lambda,\qquad i\ge 0.
\end{equation}

Now, consider the Grassmann superalgebra
$H_\infty=\Lambda(x_i,y_i| i\ge 0)$ which is turned into a Poisson superalgebra
by a bracket determined by relations:
\begin{equation}\label{relV}
\{y_i,x_j\}=\delta_{i,j},\quad \{x_i,x_j\}=\{y_i,y_j\}=0,\quad  i,j\ge 0.
\end{equation}
We obtain the bracket:
$$
\{f,g\}=(-1)^{|f|-1}\sum_{i=1}^\infty
\bigg(\frac{\partial f}{\partial x_i}\frac{\partial g}{\partial y_i}
+\frac{\partial f}{\partial y_i}\frac{\partial g}{\partial x_i}\bigg),\qquad f,g\in H_\infty.
$$

Next, we define a completion of $H_\infty$.
Denote by $\Xi$ the set of all tuples $\a=(\a_i| \a_i\in \{ 0,1\}, i\ge 0)$ with finitely many nonzero entrees.
Denote by $\epsilon_i\in\Xi$ the tuple with unique 1 on the $i$the place, $i\ge 0$.
Let $\a\in\Xi$, then put $|\a|=\sum_{i\ge 0}\a_i$, $\bar\a=\max\{i\ge 0\mid \a_i\ne 0\}$,
and $x^\a=\prod_{i\ge 0}x_i^{\a_i} \in H_{\infty}$, $y^\a=\prod_{i\ge 0}y_i^{\a_i} \in H_{\infty}$,
products being taken in increasing order.
Let $\a\in\Xi$ and for some $i\ge 0$ we have $\a_i=0$, then we consider that $x^{\a-\epsilon_i}=0$.
Below we assume that all degree tuples $\a,\b$ belong to $\Xi$.
Consider the following completion of $H_\infty$ that consists of all infinite formal sums:
$$
\tilde H_\infty=\bigg\{\sum_{
                              \bar\a<\bar\b}
\lambda_{\a,\b}x^\a y^\b\,\bigg|\, \lambda_{\a,\b}\in K \bigg\}.
$$
Since below $y_i$s will be substituted by derivatives,
we define {\em differential operators of finite order} $k$:
\begin{align*}
\tilde H_\infty^k&=\bigg\{\sum_{
                              \bar\a<\bar\b,\ |\b|= k}
\lambda_{\a,\b}x^\a y^\b\,\bigg|\, \lambda_{\a,\b}\in K \bigg\},\quad k\ge 0;\\
\HH 
    &=\mathop{\oplus}\limits_{k=0}^\infty \tilde H_\infty^k.
\end{align*}

\begin{Lemma} We formally extend the products of $H_\infty$ onto $\tilde H_\infty$. Then
\begin{enumerate}
\item $\tilde H_\infty$ is a Poisson superalgebra;
\item $\HH\subset \tilde H_\infty$ is its subalgebra.
\end{enumerate}
\end{Lemma}
\begin{proof}
Clearly, the associative product is well defined.
We check that the Poisson bracket is also well defined:
\begin{multline*}
\bigg\{
\sum_{\bar \a' <\bar\b'}\lambda_{\a'\b'}x^{\a'} y^{\b'},
\sum_{\bar \a'' <\bar\b''}\mu_{\a''\b''}x^{\a''} y^{\b''}
\bigg\}\\
=\sum_{\substack{\bar \a' <\bar\b'\\ \bar \a'' <\bar\b''}}\lambda_{\a'\b'}\mu_{\a''\b''}
\Bigg(\sum_{\substack{i\le\bar \a'<\bar\b' \\ i\le\bar \b'' \strut}}
          \pm x^{\a'-\epsilon_i} y^{\b'}x^{\a''}y^{\b''-\epsilon_i}+
      \sum_{\substack{i\le\bar \a''<\bar\b''\\ i\le \bar \b'\strut}}
          \pm x^{\a'} y^{\b'-\epsilon_i}x^{\a''-\epsilon_i} y^{\b''}\Bigg)\\
=\sum_{\a,\b}\sum_{\substack{\a'+\a''=\a\\ \b'+\b''=\b} }
\Bigg(\sum_{i<\bar \b'}\pm \lambda_{\a'+\epsilon_i,\b'}\mu_{\a'',\b''+\epsilon_i}
      +\sum_{i<\bar \b''}\pm \lambda_{\a',\b'+\epsilon_i}\mu_{\a''+\epsilon_i,\b''}
\Bigg)x^\a y^\b,
\end{multline*}
where the signs $\pm$ are uniquely determined.
While deleting $y_i$ above,
we have either $i<\bar \b'$ or $i<\bar \b''$,
the latter yield a factor $y_j$ with $i<j$, inherited by $y^{\b}$.
Hence, $\bar\a<\bar\b$ and the product belongs to $\tilde H_\infty$.

Let $f\in \tilde H_\infty^k$, $g\in \tilde H_\infty^m$, $k,m\ge 0$.
By computations above,
$f\cdot g\in \tilde H_\infty^{k+m}$ and $\{f,g\}\in \tilde H_\infty^{k+m-1}$.
Thus, $\HH$ is a subalgebra.
\end{proof}

The next elements will be referred to as the {\em pivot elements} as well:
\begin{equation}\label{pivotV}
V_i = y_i + x_ix_{i+1}(y_{i+3} + x_{i+3}x_{i+4}(y_{i+6} + x_{i+6}x_{i+7}(y_{i+9}+\cdots) ))\in
\tilde H_\infty^1\subset \HH,\qquad i\ge 0.
\end{equation}

Let $\pi:\Lambda\to H_\infty$ be the natural embedding.
Namely, consider a monomial $x^\a\in \Lambda$, $\a\in\Xi$.
Then $\pi$ maps $x^\a\in \Lambda$ on the same $x^\a\in H_\infty$.
So, we identify a tail $r_m\in\Lambda$ with the respective element $r_m\in H_{\infty}$.
Also, we define the mapping on pure derivatives
 $\pi(x^\a \dd_i)=x^\a y_i\in \tilde H_\infty^1\subset \HH$, for all $i\ge 0$, $\a\in\Xi$, $\bar\a<i$.
We extend the mapping onto infinite sums.
In particular,  we get
$\pi(v_i)=V_i $ for all $i\ge 0$.
We have images of the standard monomials:
\begin{align*}
\pi(r_{n-3}v_n)&=r_{n-3}V_n,\quad n\ge 0;\\
\pi(r_{n-5}x_{n-2}v_n)&=r_{n-5}x_{n-2}V_n,\quad n\ge 3.
\end{align*}

\begin{Lemma}\label{LisomQ}
The mapping $\pi:\QQ=\Lie(v_0,v_1,v_2)\to \Lie(V_0,V_1,V_2)\subset  \tilde H_\infty^1\subset \HH$ is an
isomorphic embedding onto a Lie subsuperalgebra of the Poisson superalgebra $\HH$.
\end{Lemma}
\begin{proof}
Observe that the Lie brackets~\eqref{relv} and~\eqref{relV} are "the same".
We conclude that the Lie brackets on
the pivot elements~\eqref{pivotv} and their images~\eqref{pivotV} are "the same", thus
$\pi([v_i,v_j])=\{V_i,V_j\}$ for all $i,j\ge 0$.
The same observation applies to the standard monomials and their images.
\end{proof}

Now we define a Poisson subalgebra $\PP=\Poisson(V_0,V_1,V_2)\subset \HH$
generated by $\{V_0,V_1,V_2\}$.
Recursive relation~\eqref{pivot2} is rewritten as:
\begin{equation}\label{pivot2V}
V_i = y_i + x_ix_{i+1} V_{i+3},\qquad i\ge 0.
\end{equation}
\begin{Lemma}
Using the associative product only,
the elements $\{x_i,V_i\mid  i\ge 0\}\subset \HH$ freely generate a Grassmann algebra in the same variables.
\end{Lemma}
\begin{proof}
Observe that both terms in~\eqref{pivot2V} are odd, they anticommute,
and their squares are equal to zero.
Thus, we get  $V_i^2=0$, $i\ge 0$.
\end{proof}

\begin{Lemma}
Let $\ch K=2$. Then $\PP$ is Poisson superalgebra,
namely, it has a formal square on the odd part and satisfy the additional axioms for $\ch K=2$.
\end{Lemma}
\begin{proof}
Let us discuss a formal square that should be defined on the odd part of $\PP$.
First, define a formal square on the odd part of $H_\infty$.
Since $(\ad x_i)^2=(\ad y_i)^2=0$, we put $x_i^{[2]}=y_i^{[2]}=0$, $i\ge 0$.
By the additional axiom (Subsection~\ref{SSPoisson}), $w^{[2]}=0$,
where $w$ is any monomial in $\{x_i,y_i|i\ge 0\}$ of odd length,
on the other hand, one checks that $(\ad w)^2=0$.
Similar to the restricted Lie algebras~\cite{JacLie},
this leads to a formal square on the whole of the odd components of $H_\infty$ and $\tilde H_\infty$.
One checks that it satisfies the additional axiom,
as was remarked above, it is sufficient to verify it on a basis consisting of words in $\{x_i,y_i|i\ge 0\}$.
Next, we restrict the formal square to $\PP$ and see that it coincides with the regular square on $\QQ$.
Finally, by the additional axiom,
a formal square on the whole of the odd part of $\PP$ does not lead to new monomials,
i.e. $\PP$ is spanned by products of the basis of $\QQ$.
\end{proof}

Define Poisson superalgebras $P_i=\Poisson(V_i,V_{i+1},V_{i+2})\subset \HH$, $i\ge 0$, so $P_0=\PP$.
We extend the shift endomorphism
$\tau:\QQ\to\QQ$ onto $\PP$ by $\tau(1)=1$ and $\tau(w_1\cdots w_m)=\tau(w_1)\cdots \tau(w_m)$, where $w_j\in \QQ$.
\begin{Corollary}
Let $\PP=\Poisson(V_0,V_1,V_2)$. Then
\begin{enumerate}
\item $V_i\in \PP$, $i\ge 0$;
\item $\tau^i:\PP\to P_i$ is an isomorphism for any $i\ge 0$;
\item we get a proper chain of isomorphic Poisson superalgebras:
$$\PP=P_0\supsetneqq P_1 \supsetneqq \cdots \supsetneqq P_i\supsetneqq P_{i+1}\supsetneqq \cdots,\qquad
\mathop{\cap}_{n=0}^\infty P_i=\langle 1 \rangle_K.
$$
\item $\PP$ is infinite dimensional.
\end{enumerate}
\end{Corollary}

\section{Bases of Poisson superalgebra $\PP$ and associative hull $\AA$}
\label{SbasesPA}

In this section, we find bases for  $\PP$ and $\AA$.
In case $\ch K\ne 2$, we prove that for a filtration of $\AA$ one
has $\gr \AA\cong\PP$, in particular, both algebras have "the same" bases.

For a series of previous examples of (self-similar) (restricted) Lie (super)algebras, bases for
respective associative hulls were not found~\cite{PeSh09,PeShZe10,PeSh13fib,Pe16,Pe17}.
Instead, we considered bigger (restricted) Lie (super)algebras $\tilde{\RR}\supset\RR$
whose bases were given by quasi-standard monomials
and we determined and used bases of their associative hulls $\tilde{\AA}=\Alg(\tilde{\RR})\supset\AA$.
The virtue of the example of a Lie superalgebra of linear growth~\cite{PeOtto} is that
for the first time, we were able to describe explicitly a basis of the associative hull.
Now, we are also able to describe bases of $\AA$ and $\PP$.

Consider a filtration $\{\AA^m\mid m\ge 0\}$ of $\AA$, where $\AA^m$ is spanned by all at most $m$-fold
products of standard monomials of the Lie superalgebra $\QQ$, $m\ge 0$.
Define the associated graded algebra
$$
\gr \AA=\mathop{\oplus}\limits_{m= 0}^\infty\AA_m,\quad \text{where}\quad \AA_m=\AA^m/\AA^{m-1},\ m\ge 0,\ \AA^{-1}=\{0\}.
$$
Similarly, let $\PP_m\subset \PP$ denote the linear span of all $m$-fold products
of the standard monomials of $\QQ$, where $m\ge 0$.
We get a direct sum $\PP=\mathop{\oplus}\limits_{m=0}^\infty \PP_m$, which is not a grading of a Poisson superalgebra
because one has $\{\PP_n,\PP_m\}\subset \PP_{n+m-1}$, $n,m\ge 1$.

\begin{Theorem}\label{TbasisP}
Let $\ch K\ne 2$.
A basis of the  Poisson superalgebra $\PP=\Poisson(V_0,V_1,V_2)$ is given by the unit
and the following monomials:
\begin{equation}\label{monomP}
x_0^{\a_0}x_1^{\a_1}\cdots x_{n-2}^{\a_{n-2}}V_0^{\b_0}V_1^{\b_1}\cdots V_{n-1}^{\b_{n-1}}V_n,
\qquad \a_i,\b_i\in\{0,1\},\quad n\ge 0,
\end{equation}
($n$ will be referred to as the {\em length})
where $\a_i$s satisfy the following restrictions:
\begin{enumerate}
\item let $\b_{n-1}=\b_{n-2}=1$, then $\a_0,\ldots,\a_{n-2}$ take all combinations;
\item let $\b_{n-1}=1$, $\b_{n-2}=0$, then at least one of $\{\a_{n-4},\a_{n-3},\a_{n-2}\}$ is zero;
\item let $\b_{n-1}=0$, $\b_{n-2}=1$, then at least one of $\{\a_{n-3},\a_{n-2}\}$ is zero;
\item let $\b_{n-1}=\b_{n-2}=0$, then either $\a_{n-2}=0$ or  $\a_{n-3}=\a_{n-4}=0$;
\item let $\b_{n-1}=\cdots=\b_{0}=0$ then we have the standard monomials of Theorem~\ref{Tbasis3};
\item we exclude finitely many monomials (of degree at most 10) that are products involving
series of standard monomials related with false monomials,  see an algorithm below.
\end{enumerate}
\end{Theorem}
\begin{proof}
Using the basis of the free Poisson superalgebra~\cite{Shestakov93},
we conclude that $\PP$ is spanned by all products of the standard monomials of $\QQ$ (Theorem~\ref{Tbasis3}).
Now, we consider all possible at most 3-fold products of the standard monomials,
the first monomial being of lengths $n$, and two optional monomials being of
lengths $n-1$ and $n-2$.
There are technical considerations because the monomials are of two types,
we omit this arguments.
One obtains restrictions (i--iv).
If a product involves only one standard monomial, we get~(v).

We need to exclude products that involve false monomials.
A {\it series} of standard monomials
is the set of the standard monomials with a head $V_n$ (i.e. the length $n$)
and a neck $x_{n-2}^{\a_{n-2}}$ fixed (so, the type is also fixed)
while the tail takes all allowed values so that we do not get a false monomial.
We have the series of the standard monomials related to false monomials:
\begin{align}\label{series1}
\begin{split}
&\hat x_0V_2,\qquad
\{\hat x_0x_1^*V_4\},\qquad
\hat x_0\, x_3 V_5,\qquad
\{\hat x_0x_1^*x_2^*\, x_5 V_7\},\\
&\{\tilde x_0x_1^*x_2^*\tilde x_3x_4^* V_7\},\qquad
\{\tilde x_0x_1^*x_2^*\tilde x_3x_4^*\tilde x_5\, x_8V_{10}\},
\end{split}
\end{align}
where $\tilde{\quad}$ denotes that the series cannot contain all the letters with this sign,
$*$ denotes that all powers are possible.
Above, the first line contains all the series,
that are simply described as not containing $x_0$.
There are some more series, actually consisting of one element, of the standard monomials not containing $x_0$:
\begin{align}\label{series2}
V_0,\quad V_1,\quad x_1V_3,\quad x_2V_4.
\end{align}
We consider a basis of $\PP$ as obtained by products of different series of
the standard monomials.
The series of the standard monomials except~\eqref{series1} and~\eqref{series2}
have arbitrary powers of $x_0$.
Observe that, multiplying by them remove all restrictions of~\eqref{series1}.

Thus, restrictions arise for products of the series,
that include at least one~\eqref{series1} and optionally some~\eqref{series2}.
Of course, we take only products without squares of any letters.
One obtains finitely many families of monomials~\eqref{monomP} with  restrictions on powers of the $x_i$s.
This leads to a finite list of monomials excluded from~\eqref{monomP}.
\end{proof}

\begin{Remark}
Consider  $\ch K=2$.
A basis of $\QQ$ consists of the standard monomials of the first type and squares of the pivot elements
(Corollary~\ref{Cbasis2}), the latter give a specific influence on a basis of $\PP$.
Recall that by the additional axiom (Subsection~\ref{SSPoisson}),
a formal square does not lead to new monomials, i.e. $\PP$ is spanned by products of the basis of $\QQ$.
For our purposes, we give only the following rough description of a basis of~$\PP$.
\end{Remark}

\begin{Corollary}\label{CbasisP2}
Let $\ch K= 2$, and $\PP=\Poisson(V_0,V_1,V_2)\subset \HH$. Then
\begin{enumerate}
\item $\PP$ is contained in a span of monomials~\eqref{monomP};
\item monomials~\eqref{monomP} with $n\ge 8$, $\a_{n-1}=\a_{n-2}=0$, and
arbitrary $\a_0,\ldots,\a_{n-3},\b_0,\ldots,\b_{n-1}\in\{0,1\}$
are linearly independent and belong to $\PP$.
\end{enumerate}
\end{Corollary}
\begin{proof}
We take the standard monomials of the first type $x_0^{\a_0}\cdots x_{n-3}^{\a_{n-3}}V_n\in\QQ$, where $n\ge 8$,
and multiply by arbitrary powers of $V_0,\ldots,V_{n-1}$.
\end{proof}

\begin{Theorem}\label{TbasisA}
Let $\ch K\ne 2$,  consider the associative hull $\AA=\Alg(v_0,v_1,v_2)\subset\End(\Lambda)$.  Then
\begin{enumerate}
\item a basis of $\AA$ consists of the unit and the replica of monomials~\eqref{monomP}:
\begin{equation}\label{monomA}
x_0^{\a_0}x_1^{\a_1}\cdots x_{n-2}^{\a_{n-2}}v_0^{\b_0}v_1^{\b_1}\cdots v_{n-1}^{\b_{n-1}}v_n,
\qquad \a_i,\b_i\in\{0,1\},\quad n\ge 0,
\end{equation}
that obey to all restrictions of Theorem~\ref{TbasisP} ($n$ will be referred to as the {\em length});
\item $\AA^m$ modulo $\AA^{m-1}$ is spanned by products $w_1\cdots w_m$ of standard monomials $w_i\in\QQ$
of strictly decreasing lengths, where $m\ge1$;
\item one has a natural isomorphisms of vector spaces $\AA_m\cong \PP_m$, $m\ge 0$;
\item $\gr \AA$ has a natural structure of a Poisson superalgebra and $\gr \AA\cong \PP$.
\end{enumerate}
\end{Theorem}
\begin{proof}
Let us prove (ii) by induction on $m$. The cases $m=0,1$ are clear. Let $m\ge 2$.
Fix a total order $\prec$ on the standard monomials that obeys to their lengths.
Consider a product $w_1 w_2 \cdots w_m\in \AA^m$, where $w_i$ are standard monomials. 
Since the commutator of two different monomials $[w_i,w_{i+1}]\in\QQ$ is expressed via standard monomials,
we can superpermute these monomials modulo $\AA^{m-1}$.
Thus, we assume that $w_1\succeq w_2\succeq \cdots \succeq w_m$.
Suppose that we obtain two elements of the same length $n$, we treat such a product:
\begin{align*}
w_iw_{i+1}&= r_{n-1}v_n\cdot r'_{n-1}v_n
  =\pm r_{n-1}r'_{n-1}v_n^2
  =\pm \frac 12 r_{n-1}r'_{n-1}[v_n,v_n]\\
&= \frac 12 [r_{n-1}v_n,r'_{n-1}v_n]
  = \frac 12 [w_i,w_{i+1}] \in\QQ.
\end{align*}
Thus, products containing such pairs belong to $\AA^{m-1}$ and we apply the inductive assumption.
As a result, we get products of standard monomials with strictly decreasing lengths, (ii) is proved.

By (ii), $\AA^m$ modulo $\AA^{m-1}$ is spanned by $m$-fold products of the standard monomials as follows:
\begin{equation}\label{decrescente}
r_{n_1-1}v_{n_1}\cdot r_{n_2-1}v_{n_2} \cdots r_{n_m-1}v_{n_m},\quad n=n_1>n_2>\cdots>n_m\ge 0, \quad m\ge 1.
\end{equation}
Now, we move all Grassmann letters in~\eqref{decrescente} to the left.
We proceed as follows.
Let $x_i$ be a Grassmann variable in a standard monomial $r_{n_j-1}v_{n_j}$, $j\ge 2$, then  $i<n_j$.
The standard monomials before it in~\eqref{decrescente} have lengths greater than $n_j$, thus, greater than $i$.
By~\eqref{action}, $x_i$ supercommutes with the preceding heads $\{v_{n_k}\mid  1\le k<j\}$,
and while moving all Grassmann letters to the left we obtain no additional terms.
Since the associative algebra $\PP$ is supercommutative, $\PP_m$ is spanned by
ordered  $m$-fold products of standard monomials the same as~\eqref{decrescente}
(one only needs to replace $v_i$s by $V_i$s).
Both products are reordered (both yield zeros provided that a Grassmann letter appears twice)
to obtain respective bases in the same way,
one of them being given by the list~\eqref{monomP} under the specified restrictions.
We get isomorphisms of vector spaces $\rho_m: \AA_m=\AA^m/\AA^{m-1}\cong \PP_m$ for all $m\ge 0$.
We get an isomorphism $\rho: \gr \AA\cong \PP$,
a check shows that this is an isomorphism of associative superalgebras.
Applying $\rho^{-1}$ to monomials~\eqref{monomP}, we get Claim~(i).
Since $\gr \AA$ is supercommutative, we supply it with a bracket as follows.
Let $a=w_1\cdots w_n\in \AA^n{\setminus} \AA^{n-1}$ and $b=w_1'\cdots w_m'\in \AA^m{\setminus} \AA^{m-1}$,
where $w_i$s, $w_j'$s are standard monomials of $\QQ$, $n,m\ge 1$.
Observe that the order in such products influences the sign only.
Denote by $\bar a$, $\bar b$ the respective images in $\AA_n=\AA^{n}/\AA^{n-1}$ and $\AA_m=\AA^{m}/\AA^{m-1}$.
Put
$$
\{\bar a,\bar b\}=[a,b]\ (\mathrm{mod}\ \AA^{n+m-2})
=\sum_{p,q}\pm\bigg(\prod_{i\ne p}w_i \prod_{j\ne q} w'_j \bigg)   [w_i,w_j']\in \AA^{n+m-1} (\mathrm{mod}\ \AA^{n+m-2}).
$$
This bracket satisfies the Leibnitz rule because it came from a supercommutator of an associative algebra
that satisfies the Leibnitz rule.
We get an isomorphism of Poisson superalgebras $\gr A\cong \PP$ because the brackets coincide
on $\QQ$ that generate both algebras as associative algebras, thus yielding~(iv).
\end{proof}
\begin{Corollary}
Let $\ch K=2$.
\begin{enumerate}
\item
The associative algebra $\AA=\Alg(v_0,v_1,v_2)$ has a basis the same as in other chractristics~\eqref{monomA}.
\item
We have a proper inclusion of Poisson superalgebras $\PP\subsetneqq  \gr\AA$.
\end{enumerate}
\end{Corollary}
\begin{proof}
Let us show that all standard monomials of the second type belong to $\AA$
by repeating the arguments of the proof of Theorem~\ref{Tbasis3}.
Recall that $x_{n-2}v_{n}\in\QQ\subset\AA$ for all $n\ge 3$, thus yielding
all standard monomials of the second type of length at most~5.
Let $n\ge 6$, then
$$
r_{n-6} v_{n-3}\cdot x_{n-5}^* v_{n-3}= r_{n-6}x_{n-5}^* v_{n-3}^2
=  r_{n-6}x_{n-5}^* x_{n-2}v_{n}\in\AA,\quad n\ge 6.
$$
We obtain all monomials of the second type except the cases
when $r_{n-6} v_{n-3}$ is false (of the first type).
a) The case of a false monomial of the first type of length 4, we get
the required standard monomials of the second type of length 7 by
$x_1^*v_4\cdot x_2^*v_4=  x_1^*x_2^*x_5v_7\in\AA$.
b) Consider that $r_{n-6} v_{n-3}$ is a false monomial of the first type of length 7.
We get all monomials of the second type of length 10, i.e. those that contain at most two of
the letters $\{x_0,x_3,x_5\}$ by:
\begin{align*}
   x_1^*x_2^*x_3^*x_4^* v_7\cdot x_5^*v_7   &=  \hat x_0 x_1^*x_2^*x_3^*x_4^* x_5^*\, x_8 v_{10}\in \AA; \\
   x_0^*x_1^*x_2^*\hat x_3x_4^* v_7\cdot x_5^*v_7 &=  x_0^* x_1^*x_2^*\hat x_3x_4^* x_5^*\, x_8 v_{10}\in \AA; \\
   x_1^*x_2^*x_3^*x_4^* v_7\cdot x_0^*v_7 &=  x_0^*x_1^*x_2^*x_3^*x_4^* \hat x_5\, x_8 v_{10}\in \AA.
\end{align*}
Now, the arguments on products of standard monomials of both types~\eqref{decrescente} above yield the same basis of $\AA$ as
that in case $\ch K\ne 2$.

To prove the second claim recall that
respective products of $\PP$ similar to~\eqref{decrescente} contain only standard monomials of the first type and squares of the pivot elements.
As a result, in the case $m=1$ we get $\PP_1\subsetneqq \AA_1=\AA^1/\AA^0$.
\end{proof}

\begin{Lemma}
Define $A_i=\Alg(v_i,v_{i+1},v_{i+2})\subset\End\Lambda$ for $i\ge 0$. Then
\begin{enumerate}
\item $\tau^i:\AA\to A_i$ is an isomorphism for any $i\ge 0$;
\item we get a proper chain of isomorphic associative superalgebras:
$$\AA=A_0\supsetneqq A_1 \supsetneqq \cdots \supsetneqq A_i\supsetneqq A_{i+1}\supsetneqq \cdots,\qquad
\mathop{\cap}_{n=0}^\infty A_i=\{0\}.
$$
\end{enumerate}
\end{Lemma}

\begin{Theorem}\label{Tquant}
Let $\ch K\ne 2$. The Poisson superalgebra $\PP$ admits an algebraic quantization.
\end{Theorem}
\begin{proof}
Consider a polynomial extension
$\Lambda^{(t)}=K[t]\otimes _K\Lambda(x_i| i\ge 0)$, where $t$ commutes with the Grassmann variables.
As above $\dd_i$ denote  the superderivative $\dd_i(x_j)=\delta_{i,j}$, $i,j\ge 0$.
Let $x_i^{(t)}$ be the operator of the left multiplication by $t x_i$ on $\Lambda^{(t)}$, $i\ge 0$.
These operators anticommute except for nontrivial relations:
\begin{equation}\label{quant_xd}
[\dd_i,x_j^{(t)}]_{(t)}=\dd_ix_j^{(t)}+x_j^{(t)}\dd_i=t \delta_{ij},\qquad
(x_i^{(t)})^2=0,\quad \dd_i^2=0,\qquad i\ge 0.
\end{equation}
Below we omit the indices $x_i^{(t)}=x_i$, $i\ge 0$.
Let elements of the Lie superalgebra $\QQ$ act on $\Lambda^{(t)}$ using relations above.
Their respective Lie products are sums of commutators of pure Lie monomials, the latter involving one commutator
of type~\eqref{quant_xd}.
Thus, $\QQ^{(t)}=K[t]\otimes _K\QQ$ is supplied with a deformed Lie superbracket:
$$
[f(t)\otimes a, g(t)\otimes b]_{(t)}=t\cdot f(t)g(t)\otimes  [a,b],\qquad f(t), g(t)\in K[t],\quad a,b\in\QQ.
$$
The actions of $\QQ^{(t)}$ generate an associative superalgebra $\AA^{(t)}=\Alg(\QQ^{(t)})\subset \End \Lambda^{(t)}$.
One checks that $\AA^{(t)}=K[t]\otimes _K\AA$, where elements of $A$ commute using the deformed superbracket.

Similarly, we define the deformed Poisson superalgebra
$H_\infty^{(t)}=K[t]\otimes_K \Lambda(x_i,y_i| i\ge 0)$ with the deformed superbracket
is uniquely determined by relations:
\begin{equation} \label{quant_relV}
\{y_i,x_j\}_{(t)}=t\delta_{i,j},\quad \{x_i,x_j\}_{(t)}=\{y_i,y_j\}_{(t)}=0,\quad  i,j\ge 0.
\end{equation}
We continue our considerations above and construct the deformed Poisson superalgebra $\PP^{(t)}=K[t]\otimes_K \PP$,
the  bracket $\{\ ,\ \}_{(t)}$  obeying to~\eqref{quant_relV}.

Let $\{\AA^m|m\ge 0\}$ be the filtration discussed in Theorem~\ref{TbasisA}.
By its arguments $\{K[t]\otimes_K\AA^m|m\ge 0\}$ is a filtration of $\AA^{(t)}$.
By construction, $\AA^{(t)}$ and $\PP^{(t)}$ are free left $K[t]$-modules with "the same" bases~\eqref{monomA}.
Repeating arguments of Theorem~\ref{TbasisA} we
get an isomorphism of associative superalgebras $\gr \AA^{(t)}\cong \PP^{t}$.

We identify the vector spaces $\AA^{(t)}=\PP^{(t)}$, this will be our algebraic quantization.
Let $*$ be the associative product of $\AA^{(t)}$ and $\cdot$ the associative product of $\PP^{(t)}$.
Consider $a=w_1\cdots w_n\in \AA^n{\setminus} \AA^{n-1}$ and $b=w_1'\cdots w_m'\in \AA^m {\setminus} \AA^{m-1}$,
where $w_i$s, $w_j'$s are standard monomials of $\QQ$, $n,m\ge 1$.
Denote respective images $\bar a\in \AA^n/ \AA^{n-1}\cong \PP_n$, and $\bar b\in \AA^m/ \AA^{m-1}\cong \PP_m$.
Permuting two basis elements yields a factor $[w_i,w_j]_{(t)}=t[w_i,w_j]\in t\AA^{(t)}t$, we simply write $O(t)$.
We have
$$
a*b=\bar a\cdot \bar b \ (\mathrm{mod}\ t).
$$
Similarly, products of $\AA^{(t)}$ that involve either two commutators e.g. $[w_i,w_j]_{(t)}$, or
a triple commutator in $w_i,w_j'$ yield a factor $t^2$.
Thus, such products belong to $t^2\AA^{(t)}$, we simply write $O(t^2)$.
We have
\begin{align*}
a*b-(-1)^{|a||b|}b*a&=[a,b]_{(t)}
=\sum_{p,q}\pm\bigg(\prod_{i\ne p}w_i \prod_{j\ne q} w'_j \bigg)   [w_i,w_j']_{(t)}\ (\mathrm{mod}\ O(t^2))\\
&=t\sum_{p,q}\pm\bigg(\prod_{i\ne p}w_i \prod_{j\ne q} w'_j \bigg)   [w_i,w_j'] \ (\mathrm{mod}\ O(t^2))\\
&=t [a,b] \ (\mathrm{mod}\ O(t^2)) =t \{\bar a,\bar b\} \ (\mathrm{mod}\ O(t^2)).\qedhere
\end{align*}
\end{proof}
\section{Weights, growth, and paraboloid for superalgebras $\PP$ and $\AA$}
\label{SweightboundsAP}

In this section, we
establish bounds on weights of algebras $\PP$ and $\AA$, prove that both algebras have a polynomial growth,
and determine positions of their multihomogeneous $\NO^3$-components in space.

In Section~\ref{Sweight} we defined different weight functions on the Lie superalgebra $\QQ$.
Since theses functions are determined by the weights of the letters $\{x_i,y_i\mid i\ge 0\}$,
these functions are extended onto $\PP$ by additivity.
A  {\it Poisson monomial} is a product in the letters $\{ x_i, V_i\mid i\ge 0\}$,
they are either even or odd with respect to $\Z_2$ superalgebra grading.
The next result is proved as Theorem~\ref{TZ3graduacao}.

\begin{Lemma}\label{LZ3graduacaoP}
The Poisson superalgebra $\PP=\Poisson(V_0,V_1,V_2)$ is $\NO^3$-graded
by multidegree in the generators $\lbrace V_0,V_1,V_2\rbrace$:
$$\PP=\mathop{\oplus}\limits_{n_1,n_2,n_3\geq 0}\PP_{n_1, n_2,n_3}.$$
\end{Lemma}

Below, $\lambda$, $\mu$ are the roots of the characteristic polynomial (Section~\ref{Sweight}).
Since $\PP$ and $\AA$ have the same bases (they differ only in case $\ch K=2$),
the proofs below are given only in case of $\PP$.
\begin{Lemma}\label{LestP}
Let $w$ be a monomial~\eqref{monomP} of $\PP$ (or a monomial~\eqref{monomA} of $\AA$) of length $n$, $n\ge 0$.
Then
$$
\lambda^{n-5}<\wt w< 2\lambda^{n+1},\quad
|\swt w|< 8|\mu|^n,\qquad n\geq 0.
$$
\end{Lemma}
\begin{proof}
Recall that $w$ arises from a product of standard monomials, one of them of length $n$,
each monomial being of positive weight.
Thus, the lower bound on the weight function follows from the lower bound of Corollary~\ref{Cpesos}.
We compute the upper bound, using that $(\lambda-1)^{-1}\approx 1.92<2$.
$$
\wt(x_0^{\a_0}\cdots x_{n-2}^{\a_{n-2}}V_0^{\b_0}\cdots V_{n-1}^{\b_{n-1}}V_n)\le \sum_{i=0}^n\lambda^i
<\frac{\lambda^{n+1}}{\lambda-1}<2\lambda^{n+1}, \quad n\ge 0.
$$
Observe that
$\swt(x_i^{\a_i}V_i^{\b_i})=\mu^i(\b_i-\a_i)\in \{0,\pm \mu^i\}$ for all $i\ge 0$. Then
$$ |\swt(x_0^{\a_0}\cdots x_{n-2}^{\a_{n-2}}V_0^{\b_0}\cdots V_{n-1}^{\b_{n-1}}V_n)|\le \sum_{i=0}^n|\mu|^i
<\frac{|\mu|^n}{1-1/|\mu|}<8|\mu|^n, $$
where we used that $(1-1/|\mu|)^{-1}\approx 7.8<8$.
\end{proof}

\begin{Theorem}\label{TgrowthP}
Consider the Poisson superalgebra $\PP$ and associative hull $\AA$ over an arbitrary field.
Then
$$\GKdim\PP=\LGKdim\PP=\GKdim\AA=\LGKdim\AA=2\log_\lambda 2\approx 3.3036.$$
\end{Theorem}
\begin{proof}
Let us find an upper bound on the weight growth function
$\tilde\gamma_\PP(m)$ which counts basis monomials $w$ such that  $\wt w\le m$, where $m\ge 1$.
Consider such a monomial $w$  of length $n$.
By Lemma~\ref{LestP}, $\lambda^{n-5}<\wt w\le m$, hence $n\le n_0=[\log_\lambda m]+5$.
We get an upper bound by counting the number of  all monomials~\eqref{monomP} of length at most $n_0$
$$
\tilde\gamma_\PP(m)
\le 1+\sum_{n=1}^{n_0}2^{2n-1}
<1+\frac 23 4^{n_0}<4^{n_0}\le 4^{\log_\lambda m+5}=2^{10}m^{2{\log_\lambda}2}.
$$
Fix $m$ and set $n=[\log_\lambda (m/2)]-1$, we may assume that $n\ge 8$.
By Corollary~\ref{CbasisP2}, monomials~\eqref{monomP}
of length $n$ with $\a_{n-1}=\a_{n-2}=0$ belong to a basis of $\PP$ in case of any characteristic.
By Lemma~\ref{LestP}, $\wt w< 2\lambda^{n+1}\le m$.
Our monomials $w$ contain $2n-2$ arbitrary powers and their number yields a lower bound:
\begin{equation*}
\tilde\gamma_\PP(m)\ge 2^{2n-2}\ge 2^{2\log_\lambda (m/2)-6}=2^{-6-2\log_\lambda 2}m^{2\log_\lambda 2}.
\qedhere
\end{equation*}
\end{proof}

\begin{Theorem}
\label{TparabP}
Put $\sigma=\log_{|\mu|}\lambda \approx 3.068$.
The lattice points of space corresponding to basis monomials of the Poisson superalgebra $\PP$ (or the associative hull $\AA$)
in terms of the standard (i.e. multidegree) coordinates
are inside an "almost cubic paraboloid", given by an equation in terms of
the twisted coordinates $\WtR(w)=(Y_1,Y_2,Y_3)$:
$$\sqrt{Y_2^2+Y_3^2}<16\sqrt[\sigma]  {Y_1}. $$
\end{Theorem}
\begin{proof}
Let $w$ be a monomial~\eqref{monomP} of $\PP$ of length $n\geq 0$
with the weight coordinates $\Wt(w)=(Z_1,Z_2,Z_3)=(\wt w,\swt w, \sswt w)$.
By Lemma~\ref{LestP}, $\lambda^{n-5}<\wt w=Z_1$, thus $n<\log_\lambda Z_1+5$.
The second inequality of Lemma~\ref{LestP} yields
\begin{equation*}
|Z_2|=|\swt w |< 8|\mu|^n< 8|\mu|^{\log_\lambda Z_1+5}
=8|\mu|^5 Z_1^{\log_\lambda |\mu|}<16 Z_1^{1/\sigma},
\end{equation*}
using $8|\mu|^5\approx 15.86<16$.
By Lemma~\ref{Ltrans}, we have relations $Z_1=Y_1$ and $Z_2=Y_2+iY_3$. We obtain:
\begin{equation*}
\sqrt{Y_2^2+Y_3^2}=|Z_2|<16 Z_1^{1/\sigma}=16 Y_1^{1/\sigma}. \qedhere
\end{equation*}
\end{proof}

\section{Jordan superalgebra $\JJ$, its $\Z^4$-grading and properties}\label{SJordan}

Assume that $\ch K\ne 2$.
Now we consider the Poisson superalgebra $\PP=\Poisson(V_0,V_1,V_2)$, its Kantor double yields
a Jordan superalgebra $\JJ=\Kan(\PP)=\PP\oplus \bar \PP$.
In this section, we determine its properties.
Namely, we establish a $\Z^4$-grading of the Jordan superalgebra $\JJ$, determine its growth,
and determine positions of its basis monomials in $\R^4$.

First, let us determine its generators.
\begin{Lemma}\label{LgeneratorsJ}
The Jordan superalgebra $\JJ=\Kan(\PP(V_0,V_1,V_2))$ is generated by $\{V_0,V_1,V_2,\bar 1\}$.
\end{Lemma}
\begin{proof}
Let $J=\Jord(V_0,V_1,V_2,\bar 1)\subset\JJ$ be a Jordan superalgebra generated by $\{V_0,V_1,V_2,\bar 1\}$.
We identify $\QQ$ with $\Lie(V_0,V_1,V_2)$ (Lemma~\ref{LisomQ}).
Let us prove by induction on $n$ that $V_n,\bar V_n\in J$ for all $n\ge 0$.
Using Lemma~\ref{L_BASIC_PROD}, we get
\begin{align*}
  V_n\bullet \bar 1&=\bar V_n; \\
 \bar V_{n-1}\bullet \bar V_n&=\{V_{n-1},V_n\}=-x_{n-1}V_{n+2};\\
 x_{n-1}V_{n+2}\bullet \bar 1&=\overline {x_{n-1}V_{n+2}};\\
 \bar V_{n-1}\bullet \overline {x_{n-1}V_{n+2}}&=\{V_{n-1},x_{n-1}V_{n+2}\}= V_{n+2}.
\end{align*}
Similarly, for any standard monomial $w\in \QQ$ we show that $w,\bar w\in J$.
Indeed, consider standard monomials $w_1,w_2\in \QQ$ and suppose that $w_1,w_2\in J\cap\PP$.
Then
$w=(w_1\bullet 1)\bullet (w_2\bullet 1)=\bar w_1\bullet \bar w_2=\{w_1,w_2\}\in J\cap \PP$.
Recall that $\PP$ is spanned by products of standard monomials (proof of Theorem~\ref{TbasisP}).
Let $w_1,\ldots,w_m\in\QQ$ be standard monomials.
Then $w=w_1\cdots w_m=w_1\bullet\cdots\bullet w_m\in J$ and $\overline{w}=w\bullet\bar 1\in \bar J$.
Therefore, $J=\PP\oplus\bar\PP=\JJ$.
\end{proof}
\begin{Corollary}
$(V_{n-1}\bullet\bar 1)\bullet (((V_{n-1}\bullet\bar 1)\bullet (V_n\bullet \bar 1))\bullet \bar 1)=-V_{n+2}$,\
$n\ge 1$.
\end{Corollary}
We extend the shift endomorphism $\tau:\PP\to\PP$ onto $\JJ$ by
$\tau(\bar v)=\overline{\tau(v)}$,  $v\in \PP$.
We get $\tau(\bar 1)=\bar 1$, and $\tau(V_i)=V_{i+1}$, for all $i\ge 0$.
Define Jordan superalgebras $J_i=\Jord(V_i,V_{i+1},V_{i+2},\bar 1)$ for all $i\ge 0$, so $J_0=\JJ$.
\begin{Corollary}\label{Cspecial}
Let $\JJ=\Jord(V_0,V_1,V_2,\bar 1)$. Then
\begin{enumerate}
\item $\{V_i\mid i\ge 0\}\subset \JJ$;
\item $\tau^i:\JJ\to J_i$ is an isomorphism for any $i\ge 0$;
\item we get a proper chain of isomorphic subalgebras:
$$\JJ=J_0\supsetneqq J_1 \supsetneqq \cdots \supsetneqq J_i\supsetneqq J_{i+1}\supsetneqq \cdots,\qquad
\mathop{\cap}_{n=0}^\infty J_i=\langle 1,\bar 1\rangle.
$$
\item $\JJ$ is infinite dimensional;
\item $\JJ$ is weakly special but not special.
\end{enumerate}
\end{Corollary}
\begin{proof}
The last claim follows from the known fact that
the Kantor double of a Poisson superalgebra is weakly special~\cite{Shestakov93,Sko94}
(a more general similar fact for arbitrary Poisson brackets is established in~\cite{MaShZe01}).
On the other hand, the Kantor double is special if and only if the Poisson superalgebra is Lie nilpotent of class 2,
namely, it satisfies the identity $\{X,\{Y,Z \}\}= 0$~\cite{Shestakov93}, which is not true in our case.
\end{proof}

We extend the weight functions of $\QQ$ and $\PP$ onto $\JJ$ by setting
$\wt(\bar 1)=\swt(\bar 1)=0$.
Using Lemma~\ref{LZ3graduacaoP}, we get.
\begin{Lemma}\label{LZ3graduacaoJ}
The Jordan superalgebra $\JJ=\Jord(V_0,V_1,V_2,\bar 1)$ is $\NO^3$-graded
by a partial multidegree in $\lbrace V_0,V_1,V_2\rbrace$:
\begin{equation*}
\JJ=\mathop{\oplus}\limits_{n_1,n_2,n_3\geq 0}\JJ_{n_1, n_2,n_3}.
\end{equation*}
\end{Lemma}

\begin{Remark}
Consider the case $\ch K=2$. The Kantor double of the Poisson superalgebra $\PP$
yields an algebra with a binary operation $\JJ=\Kan(\PP)$.
Similarly, below we can define an algebra with a binary operation $\KK=\Jor(\QQ)$ as well.
Probably, these superalgebras can be supplied with appropriate ternary operations and
be considered as Jordan superalgebras in characteristic 2.
\end{Remark}

Now let us study $\JJ$ in more details.
A {\it monomial} of the Jordan superalgebra $\JJ=\PP\oplus\bar\PP$ is
either $w\in\PP$ or $\bar w\in \PP$, where $w$ is
a product in the letters $\{ x_i, V_i\mid i\ge 0\}$,
such monomials are either even or odd with respect to the $\Z_2$-grading of the superalgebra.
Below, formulas involving $\JJ$ are written for $\Z_2$-homogeneous elements either of $\PP$ or $\bar \PP$.

Let $w\in \JJ_{n_1, n_2,n_3}$ we keep the notation $\deg(w)=n_1+n_2+n_3$,
the total degree in the set $\lbrace V_0,V_1,V_2\rbrace$.
Now we are going to introduce functions specific to the Jordan algebra $\JJ$.
Consider a monomial
\begin{equation}\label{monomP2}
u=x_{i_1}\cdots x_{i_k}V_{j_1}\cdots V_{j_m}\in \PP,\quad
\quad i_1<\cdots< i_k,\quad 0 \leq j_1<\cdots< j_m,\qquad m\ge 0.
\end{equation}
We count a {\em multiplicity of the pivot elements} in this record of $u\in\PP$
(or in its copy $\bar u=u\bullet \bar 1\in \overline\PP$) by setting:
$$ \mult_V(u)=\mult_V(\bar u)= m.$$
Let $w\in\JJ=\PP\oplus \bar\PP$, put
$$
\epsilon(w)=\begin{cases} 0,\quad & w\in \PP;\\ 1, & w\in\overline \PP,\end{cases}
$$
where using $\epsilon(w)$ we assume that either $w\in\PP$ or $w\in\overline\PP$.
Define a specific {\em Jordan weight} function $\jwt(*)$:
\begin{equation}\label{mult1}
\jwt(w)=2\mult_V(w)-\epsilon(w)
,\qquad w\in\JJ.
\end{equation}
\begin{Lemma} \label{LjordProd}
The Jordan weight $\jwt(*)$ has the following properties. Let $a,b$ be monomials of $\JJ$. Then
\begin{enumerate}
\item $\jwt(1)=0$;
\item $\jwt(\bar 1)=-1$;
\item $\jwt(V_j)=2$, $j\ge 0$;
\item $\jwt(\overline V_j)=1$, $j\ge 0$;
\item $1\le \jwt(a)$ for $a\ne 1,\bar 1$;
\item $\jwt(a\bullet b)=\jwt(a)+\jwt(b)$ (i.e. the function is additive);
\item $-1\le \jwt(a)< 12+2\log_\lambda \wt(a)$.
\end{enumerate}
\end{Lemma}
\begin{proof}
Items (i--iv) follow by definition.
Consider (v), we observe that $\mult_V(a)\ge 1$, hence $\jwt(a)\ge 1$.

Let us prove the additivity.
The cases $a,b\in \PP$ and $a\in\PP$, $b\in\overline\PP$ are trivial.
Consider the case $a,b\in\overline\PP$.
Then we can consider that $a=\overline{a'(x)V_{i_1}\cdots V_{i_k}}$, $i_1<\cdots <i_k$ and
$b=\overline{b'(x)V_{j_1}\cdots V_{j_m}}$, $j_1<\cdots< j_m $,
where $a'(x),b'(x)$ are monomials in $\{x_i\mid i\ge 0\}$.
The product $a\bullet b$ is a linear combination of products,
where the original factors are being kept except those of
either $\{V_{i_p},V_{j_q}\}=\sum_{l} c_l(x) V_{l}$, $c_l(x)\in \Lambda(x_i\mid i\ge 0)$
or $\{V_{i_p},x_{j_q}\}=c(x)\in \Lambda(x_i\mid i\ge 0)$.
In both cases, we lose one pivot letter $V_i$, thus
\begin{equation*}
\jwt(a\bullet b)=2(k+m-1)=2k-1+2m-1=\jwt(a)+\jwt(b). 
\end{equation*}

Let us check bounds (vii).
The lower bound follows from the definition~\eqref{mult1}.
Let $u\in\PP$ be a monomial~\eqref{monomP2} of length $n\ge 0$, i.e. $j_m=n.$
Then $\mult_V(u)=m\le n+1$.
By Lemma~\ref{LestP}, $\lambda^{n-5}< \wt(u)$.
Hence, $\mult_V(u)< 6+\log_\lambda \wt(u)$.
Finally, either $a=u$ or $a=\bar u$ and we apply~\eqref{mult1}.
\end{proof}

\begin{Lemma}\label{LJGr}
Consider the Jordan superalgebra $\JJ$ as generated by the set $\{V_0,V_1,V_2,\bar 1\}$.
Let $w\in \JJ$ be a monomial with the multidegree coordinates $\Gr(w)=(X_1,X_2,X_3)$ and a (partial) degree $\deg w=X_1+X_2+X_3$
(see Lemma~\ref{LZ3graduacaoJ}).
\begin{enumerate}
  \item there exists a well-defined degree $\deg_{\bar 1}(w)$  with respect to $\bar 1$ for a monomial $w\in\JJ$.
  \item $\deg_{\bar 1}(w)=2\deg w-\jwt w$, $w\in\JJ$;
  \item $\deg_{\bar 1}(w)=2(\deg w-\mult_V(w))+\epsilon(w)$, $w\in\JJ$;
  \item $\deg_{\bar 1}(*)$ is additive on $\JJ$.
\end{enumerate}
\end{Lemma}
\begin{proof} Let $w\in\JJ$ be a Jordan monomial, which involves $X_1,X_2,X_3$ factors $V_0,V_1,V_2$, respectively,
and  $\deg_{\bar 1}(w)$ factors $\bar 1$.
Using additivity of $\jwt(*)$ and its basic values (Lemma~\ref{LjordProd}), we get
\begin{align*}
\jwt(w)&=X_1\jwt(V_0)+X_2\jwt(V_1)+X_3\jwt(V_2)+\deg_{\bar 1}(w)\jwt(\bar 1)=2\deg w-\deg_{\bar 1}(w);\\
\deg_{\bar 1}(w)&=2\deg w-\jwt w,
\end{align*}
thus proving (i), (ii).
Using~\eqref{mult1} we get (iii).
Additivity of $\deg(*)$ and $\jwt(*)$  yields additivity of $\deg_{\bar 1}(*)$.
\end{proof}
Consider a monomial $w\in\JJ$ with the multidegree coordinates $\Gr(w)=(X_1,X_2,X_3)$.
We introduce one more coordinate $X_4=\deg_{\bar 1}(w)$.
Define an {\em extended multidegree} with respect to the generators
$\{V_0,V_1,V_2,\bar 1\}$ and an {\em extended degree}:
\begin{align*}
\wGr(w)&=(X_1,X_2,X_3,X_4)\in\NO^4;\\
\wdeg(w)&=X_1+X_2+X_3+X_4=\deg w+\deg_{\bar 1}(w),\quad w\in\JJ;
\end{align*}
in particular,  $\wdeg(1)=0$, $\wdeg(\bar 1)=1$.
We draw monomials $w\in\JJ$ using the extended multidegree coordinates,
thus putting monomials at lattice points $\wGr(w)\in \Z^4\subset \R^4$.

\begin{Corollary}\label{CgradedZ4J-A}
Consider the Jordan superalgebra $\JJ$.
\begin{enumerate}
\item The functions $\wGr(*)$, $\wdeg(*)$ are additive on $\JJ$;
\item $\JJ$ is $\NO^4$-graded using the extended
multidegree $\wGr(w)=(X_1,X_2,X_3,X_4)\in\NO^4$ in the generators $\{V_0,V_1,V_2,\bar 1\}$.
\end{enumerate}
\end{Corollary}
\begin{proof}
Follow from Lemma and the definitions.
\end{proof}

\begin{Corollary}\label{Cextended_growth}
Consider a monomial $w\in\JJ$. Then
\begin{enumerate}
\item $\wdeg w=3\deg w-\jwt w$;
\item
$ \deg w\le \wdeg w< 3\deg w$, where $w\ne \bar 1, 1$;
\item
$\wdeg V_n=3\deg V_n-2$, $n\ge 0$.
\end{enumerate}
\end{Corollary}
\begin{proof}
We use item (ii) of Lemma, definition of the extended degree,
and items (iii), (v) of Lemma~\ref{LjordProd}.
\end{proof}

\begin{Theorem}\label{TgrowthJ}
Consider the Jordan superalgebra $\JJ=\Jord(V_0,V_1,V_2,\bar 1)$. Then
$$\GKdim\JJ=\LGKdim\JJ=2\log_\lambda 2\approx 3.3036.$$
\end{Theorem}
\begin{proof}
Fix $m\ge 0$.
The ordinary growth function $\gamma_{\JJ}(m, \{V_0,V_1,V_2,\bar 1\})$
counts basis monomials $w\in\JJ$ such that $\wdeg(w)\le m$,
by the lower inequality of Corollary~\ref{Cextended_growth}, we have $\deg w\le \wdeg(w)\le  m$.
Thus, the above set of monomials is contained in $\{u,\bar u\mid u \text{ basis monomial of } \PP,\ \deg u\le m\}$.
Since  $\gamma_{\PP}(m, \{V_0,V_1,V_2\})$ counts basis monomials $u\in\PP$ such that $\deg u\le m$,
we obtain the upper bound below
\begin{equation}\label{estimates_gr_J}
2(\gamma_{\PP}(m/3, \{V_0,V_1,V_2\})-1)
\le \gamma_{\JJ}(m, \{V_0,V_1,V_2,\bar 1\})
\le 2\gamma_{\PP}(m, \{V_0,V_1,V_2\}),\quad m\ge 1.
\end{equation}
Similarly,
let $u\in\PP$ be a basis monomial with $\deg u\le m/3$ and $u\ne 1$. Then $w=u$ and $w=\bar u$ are basis elements of $\JJ$
with  $\wdeg w\le 3\deg w\le m$ by the upper bound of Corollary~\ref{Cextended_growth},
thus we prove the claimed lower bound.
Now, it remains to use bounds of Theorem~\ref{TgrowthP}.
\end{proof}

Consider a monomial $w\in\JJ$, then either $w=u\in\PP$ or $w=\bar u\in\overline\PP$.
By our constructions above, this monomial has the twisted coordinates
$\WtR (w)=(Y_1,Y_2,Y_3)\in\R^3$.
We add one more coordinate $Y_4=\jwt w$.
Now we define {\em extended twisted coordinates}:
$\wWtR(w)=(Y_1,Y_2,Y_3,Y_4)\in\R^4$.

\begin{Lemma} \label{LwGr_J}
Let $w\in\JJ$ be a monomial.
\begin{enumerate}
\item the function $\wWtR(*)$ is additive on $\JJ$;
\item the first three components of
$\wGr(w)=(X_1,X_2,X_3,X_4)$ and $\wWtR(w)=(Y_1,Y_2,Y_3,Y_4)$ are
related by~$\mathrm{(iv)}$ of Lemma~\ref{Ltrans}. The forth coordinates are related by
$$ Y_4=2(X_1+X_2+X_3)-X_4; $$
\item $-1\le Y_4< 12+2\log_\lambda Y_1 $;
\item
The axis $OY_1$ in terms of the standard coordinates 
is given by $(2/\lambda,\lambda,1, 2\lambda^2+2\lambda)$.
\end{enumerate}
\end{Lemma}
\begin{proof}
The additivity of $\wWtR(*)$ follows from that for $\jwt(*)$.
By~(ii) of~Lemma~\ref{LJGr},
$X_4=\deg_{\bar 1}(w)=2\deg w-\jwt w=2(X_1+X_2+X_3)-Y_4$, thus yielding the second claim.

Recall that $Y_1=\wt w$ and $Y_4=\jwt w$.
Using estimates (vii) of Lemma~\ref{LjordProd}, we have
$-1\le Y_4< 12+2\log_\lambda \wt(w)= 12+2\log_\lambda Y_1$.

Let us prove (iv).
By Lemma~\ref{L_OY1}, let $(X_1,X_2,X_3)=(2/\lambda,\lambda,1)$.
The condition $Y_4=0$, (ii), and~\eqref{frac2}  yield
$X_4=2(X_1+X_2+X_3)=2(2/\lambda+\lambda+1)=2(\lambda^2-1+\lambda+1)=2\lambda^2+2\lambda$.
\end{proof}

\begin{Theorem}\label{TparabJ}
Let monomials $w$ of the Jordan superalgebra $\JJ$ be drawn in $\R^4$ using the extended multidegrees
$\wGr(w)=(X_1,X_2,X_3,X_4)\in\NO^4\subset \R^4.$
In terms of the extended twisted coordinates $(Y_1,Y_2,Y_3,Y_4)$,
the respective points are inside a figure determined by inequalities:
\begin{align*}
\sqrt{Y_2^2+Y_3^2}<16\sqrt[\sigma]  {Y_1},& \qquad (\text{where}\ \sigma=\log_{|\mu|}\lambda \approx 3.068);\\
-1\le Y_4 < 12+2\log_\lambda Y_1,&\qquad
(1 \le Y_4 < 12+2\log_\lambda Y_1, \quad \text{if}\quad  w\ne 1,\bar 1).
\end{align*}
\end{Theorem}
\begin{proof}
The inequalities are established in Theorem~\ref{TparabP} and Lemma~\ref{LwGr_J}.
\end{proof}

\begin{Corollary}\label{CgradingZ4J}
Consider the $\NO^4$-grading of the Jordan superalgebra $\JJ$ by multidegree in the generators $\{V_0,V_1,V_2,\bar 1\}$:
$$
\JJ=\mathop{\oplus}\limits_{n_1,n_2,n_3,n_4\ge 0} \JJ_{n_1n_2n_3n_4}.
$$
The numbers $\{\dim \JJ_{n_1n_2n_3n_4} \mid (n_1,n_2,n_3,n_4)\in \NO^4\}$ are not bounded.
\end{Corollary}
\begin{proof}
By way of contradiction, suppose that the dimensions are bounded by a constant $C$.
The ordinary growth function $\gamma_\JJ(m, \{V_0,V_1,V_2,\bar 1\})$
counts basis monomials of $\JJ$ with $\wdeg (w)\le m$, where $m\ge 0$.
By Corollary~\ref{Cextended_growth},
$Y_1=\deg w\le \wdeg (w)\le m$.
So, we introduce a bigger function $g(m)=\dim\langle w\in\JJ\mid \deg w\le m\rangle$.
We cut the figure of Theorem by the hyperplane $Y_1\le m$,
consider a larger cylinder, and evaluate volume of the latter (in the extended twisted coordinates):
\begin{align}
&\{(Y_1,Y_2, Y_3,Y_4)\mid 0\le Y_1\le m,\ \sqrt{Y_2^2+Y_3^2}<16\sqrt[\sigma]  {m},\ -1 \le Y_4 < 12+2\log_\lambda m\};\\
&\mathrm{Volume}(m)= m\cdot \pi 256 m^{2/\sigma}
\cdot(13+2\log_\lambda m)
\le C_1 m^{5/3},
\quad m\gg1,
\end{align}
because $\sigma>3$.
The volume of the cylinder in the extended standard coordinates and
the number of lattice points in it (in terms of the extended standard coordinates) have the same asymptotic, with a constant $C_2$.
Thus, $\gamma_\JJ(m)\le g(m)\le CC_2m^{5/3}$, $m\gg 1$,
a contradiction with Theorem~\ref{TgrowthJ}.
\end{proof}

\section{Jordan superalgebra $\KK$ and its properties}\label{SJordanK}

Now we introduce our last object, the Jordan superalgebra $\KK$ and study its properties.
We show that $\KK$ is a factor algebra of the Jordan superalgebra $\JJ$ constructed above, thus
we can apply all the machinery developed for $\JJ$.

Let $L$ be an arbitrary Lie superalgebra.
Its symmetric algebra $S(L)$ has the structure of a Poisson superalgebra.
Observe, that the subspace $H\subset S(L)$ spanned by all tensors of length at least two is its ideal.
Thus, one obtains a (rather trivial) Poisson superalgebra $P(L)=S(L)/H$,
which equivalently can be obtained as a vector space endowed with Poisson products which are nontrivial in the following cases only:
$$
P(L)=\langle 1\rangle \oplus L,\qquad 1\cdot x=x,\ \{x,y\}=[x,y], \quad x,y\in L.
$$
Using Kantor double, define a Jordan superalgebra $\Jor(L)=\Kan(P(L))$.
Equivalently, one can just take a vector space supplied with a product $\bullet$
which is nontrivial in the following cases
(see an example at the end~\cite{She99}):
$$
\Jor(L)=\langle 1\rangle \oplus L\oplus \langle \bar 1\rangle \oplus \bar L,\qquad
\bar x\bullet \bar y=[x,y], \quad x\bullet \bar 1=(-1)^{|x|}\bar 1\bullet x=\bar x,\quad x,y\in L;\ 1\text{ the unit}.
$$

Now we define the Jordan superalgebra $\KK=\Jor(\QQ)$.
\begin{Lemma} Let $\KK=\Jor(\QQ)$. Then
\begin{enumerate}
\item  We have generators: $\KK=\Jord(v_0,v_1,v_2,\bar 1)$;
\item define Jordan superalgebras $K_i=\Jord(v_i,v_{i+1},v_{i+2},\bar 1)\subset\KK$ for all $i\ge 0$, so $K_0=\KK$.
We get a proper chain of isomorphic subalgebras:
$$\KK=K_0\supsetneqq K_1 \supsetneqq \cdots \supsetneqq K_i\supsetneqq K_{i+1}\supsetneqq \cdots,\qquad
\mathop{\cap}_{n=0}^\infty K_i=\langle 1,\bar 1\rangle;$$
\end{enumerate}
\end{Lemma}
\begin{proof}
Define the subalgebra $K'=\Jord(v_0,v_1,v_2,\bar 1)\subset \KK$.
Computations of Lemma~\ref{LgeneratorsJ} yield that $v_i\in K'$ for all $i\ge 0$, moreover all basis elements of $\QQ$ belong to $K'$.
Thus, $K'=\KK$. The second claim follows by applying the endomorphism $\tau$.
\end{proof}

If an associative superalgebra $A$ is just infinite then
the related Jordan superalgebra $A^{(+)}$ is just infinite as well~\cite{ZhePan17}.
We establish a similar fact.
\begin{Lemma}\label{Ljust-inf-Jor}
Let $L$ be a Lie superalgebra, consider the Jordan superalgebra $\Jor(L)$.
\begin{enumerate}
\item $\Jor(L)$  is just infinite if and only if $L$ is just infinite.
\item The ideal without unit $\Jor^o(L)=L\oplus \langle \bar 1\rangle \oplus \bar L$ is solvable of length 3.
\item This is a nil-ideal of bounded degree: $a^6=0$ for $a\in\Jor^o(L)$.
\end{enumerate}
\end{Lemma}
\begin{proof}
Let $L$ be not just infinite. Then there exists an ideal of infinite codimension $I\triangleleft L$
and $I\oplus \bar I$ is an ideal of infinite codimension in $\Jor(L)$. Therefore, $\Jor(L)$ is not just infinite.

Conversely, suppose that $L$ is just infinite.
By way of contradiction,
assume that $H\subset \Jor(L)$ is an ideal of infinite codimension.
Then $\tilde H=H\cap (L\oplus \bar L)\subset \Jor(L)$ is also an ideal of infinite codimension.
Denote by $H_0$ and $\bar H_1$
the projections of $\tilde H$ onto  $L$, $\bar L$, respectively ($\bar H_1$ being the copy of a subspace $H_1\subset L$).
Since $\tilde H$ is an ideal,
$\bar 1\bullet \tilde H=\bar H_0\subset \bar H_1$ and
$\bar L\bullet \tilde H=[L,H_1]\subset H_0$ and
we get $[L,H_1]\subset H_0\subset H_1\subset L$.
Hence $H_0\subset L$ is an ideal, which must be either zero or of finite codimension by our assumption.
Let $H_0\subset L$ be of finite codimension then $\tilde H\subset \Jor(L)$ is of finite codimension, a contradiction.
Now assume that $H_0=0$. Then $[L,H_1]=0$ and $H_1$ is central. By taking $0\ne z\in H_1$,
we get an ideal $\langle z\rangle \subset L$ of infinite codimension, a contradiction.
Thus, $\Jor(L)$ is just infinite.

We repeat the arguments of~\cite{She99}.
Denote $J=\Jor^o(L)$.
Then $J^2\subset L\oplus \bar L$, $(J^2)^2\subset L$, and $((J^2)^2)^2=0$.
Thus, $J$ is solvable of length 3.
\end{proof}

Let $\KK^o$ be the ideal of the Jordan superalgebra $\KK=\Jor(\QQ)$ without unit. We have a basis
$$\KK=\langle 1, \bar 1, w, \bar w\mid w\text{ are standard monomials of } \QQ \rangle. $$
In particular, all the pivot elements $\{v_i|i\ge 0\}$, as well as their copies $\{\bar v_i|i\ge 0\}$ belong to $\KK$.

\begin{Lemma}
One has a canonical isomorphism of Jordan superalgebras
$\KK\cong \JJ/I$,
where $I$ is the ideal of $\JJ$ spanned by all its monomials containing two pivot letters $V_i$ or
two their copies $\bar V_i$.
\end{Lemma}
\begin{proof}
Consider the Jordan superalgebra $\JJ=\Kan(\PP)=\PP\oplus \bar\PP$ with the product $\bullet$.
Fix $m\ge 0$, as above, denote by  $\PP_m\subset \PP$ a linear span of all $m$-fold products
of standard monomials of $\QQ$,
equivalently, $\PP_m$ is spanned by the basis monomials containing exactly $m$ letters $V_i$.
We get vector space decompositions $\PP=\mathop{\oplus}\limits_{m=0}^\infty \PP_m$ and
$\bar \PP=\mathop{\oplus}\limits_{m=0}^\infty \bar \PP_m$.
Observe that
$$
\PP_n\bullet \PP_m\subset \PP_{n+m},\quad
\PP_n\bullet \bar \PP_m=\bar \PP_m\bullet\PP_n \subset \bar\PP_{n+m}, \quad
\bar \PP_n\bullet \bar \PP_m\subset \PP_{n+m-1},\qquad n,m\ge 0.
$$
Let $I=\mathop{\oplus}\limits_{n\ge 2}(\PP_n\oplus \bar\PP_n)$.
The multiplication rules above imply that $I$ is an ideal in $\JJ$.
Indeed, one needs to check the last product, where we use
that $\bar\PP_0=\langle\bar 1\rangle$ and $\bar 1\bullet \bar L=0$.
We get
\begin{equation*}
\JJ/I\cong \PP_0\oplus \PP_1\oplus \bar\PP_0\oplus\bar \PP_1
=\langle 1\rangle \oplus\QQ\oplus\langle \bar 1\rangle \oplus \bar \QQ\cong \Jor(\QQ)=\KK.\qedhere
\end{equation*}
\end{proof}
By~\eqref{mult1}, $\jwt(\PP_n)=2n$ and $\jwt(\bar \PP_n)=2n-1$ for all $n\ge 0$.
We have another description of the ideal above, namely, $I$ is spanned by monomials $u=w$ or $u=\bar w$,
where $w$ are basis monomial of $\PP$ such that $\jwt(u)\ge 3$.

\begin{Theorem}\label{TallK}
Consider the Jordan superalgebra $\KK=\Jor(\QQ)$. Then
\begin{enumerate}
\item $\KK$ is generated by $\{v_0,v_1,v_2,\bar 1\}$;
\item $\KK$ is $\NO^3$-graded by multidegree in $\{v_0,v_1,v_2\}$,
the respective components are either trivial or two-dimensional and are inside
the "almost cubic paraboloid" of Theorem~\ref{Tparab} (see also Figure~\ref{Fig1});
\item $\KK$ is $\NO^4$-graded by multidegree in the generators $\{v_0,v_1,v_2,\bar 1\}$:
$$
\KK=\mathop{\oplus}\limits_{n_1,n_2,n_3,n_4\ge 0} \KK_{n_1n_2n_3n_4},
$$
the components are at most one-dimensional, so, the $\NO^4$-grading is fine.
The points for nontrivial $\NO^4$-components in $\R^4$ satisfy the following inequalities using
the extended twisted coordinates $\WtR^\sharp(w)=(Y_1,Y_2,Y_3,Y_4)$:
\begin{align*}
\sqrt{Y_2^2+Y_3^2}&< 14 \sqrt[\strut\sigma]{Y_1},\qquad  \sigma=\log_{|\mu|}\lambda \approx 3.068;\\
-1&\le Y_4 \le 2; \qquad (1 \le Y_4 \le 2\quad \text{if}\quad  w\ne 1,\bar 1).
\end{align*}
\item consider a (natural) gradation $\KK=\mathop{\oplus}\limits_{n\ge 0}\KK_n$ by degree in the generators;
except the initial components $\KK_0=\langle 1\rangle$, $\KK_1=\{v_0,v_1,v_2,\bar 1\}$, we have:
$$\KK_{3n-2}=\QQ_n,\quad \KK_{3n-1}=\bar \QQ_n, \quad \KK_{3n}=0, \qquad n\ge 1;$$
\item $\GKdim\KK=\LGKdim\KK=\log_\lambda 2\approx 1.6518$;
\item $\KK$ is just infinite but not hereditary just infinite;
\item the ideal without unit $\KK^o$ is solvable of length 3;
\item elements of $\KK^o$ are nil of degree at most 6;
\item $\KK$ is weakly special but not special.
\end{enumerate}
\end{Theorem}
\begin{proof}
Since $\KK$ is a factor algebra of $\JJ$ by a homogeneous ideal,
all the weight functions $\Gr$, $\Gr^\sharp$, $\Wt$, $\WtR$, $\WtR^\sharp$ as well as
the $\NO^3$ and $\NO^4$-gradings are inherited.
We get the almost cubic paraboloid by Theorem~\ref{Tparab}.
Let $w=r_{n-2}v_n\in\QQ_{n_1n_2n_3}$ be a standard monomial of $\QQ$.
We get a two-dimensional component $\KK_{n_1n_2n_3}=\langle w,\bar w\rangle$.
Also, $\KK_{000}=\langle 1,\bar 1\rangle $.
By~\eqref{mult1}, $\jwt w=2$ and $\jwt \bar w=1$, also $\jwt 1=0$ and $\jwt \bar 1=-1$.
Hence, due to the forth different coordinates,
the components of the $\NO^4$-grading of $\KK$ are at most one dimensional.
Also, we get $Y_4=\jwt u\in \{1,2\}$, where $u\in\KK$ is a basis element distinct from $1,\bar 1$.

Consider the gradation of $\KK$ be degree in all generators, which was called the extended degree.
Let $u\in\KK$ be a basis monomial, distinct from $1,\bar 1$. We have two cases.
First, assume that $u$ is a standard monomial of $\QQ$, which has a unique letter $V_i$.
By~\eqref{mult1} and (i) of Corollary~\ref{Cextended_growth}, we have $\jwt u=2$ and $\deg^\sharp u=3\deg u-\jwt u=3\deg u-2$.
Second, $u=\bar w$, $w$ being a standard monomial. Then $\jwt \bar w=1$ and $\deg^\sharp u=3\deg u-1$.
Recall that the condition $\deg w=n$ is equivalent  to $w\in\QQ_n$. We obtain the desired correspondence between components.

To evaluate the growth we use estimates~\eqref{estimates_gr_J} and Theorem~\ref{TgrowthQ}.

By Lemma~\ref{Ljust-inf-Jor}, $\KK$ is just infinite.
Let us prove that $\KK$ is not hereditary just infinite.
We use notations of Lemma~\ref{Lnonjustinf}.
Fix $m\ge 1$.
Let $\QQ(m)\subset\QQ$ be the linear span of the standard monomials of length at least $m$.
By multiplication rules, $H=\QQ(m)\oplus \overline{\QQ(m)}\subset \KK $ is an ideal of finite codimension.
Let $J=x_0\QQ(m)\oplus \overline{x_0\QQ(m)}\subset H$ be the subspace spanned by the monomials involving $x_0$.
We see that $J$ is an abelian ideal of $H$.
Since $v_i\in H{\setminus} J$, where $i\ge m$, we conclude that
$\dim H/J=\infty$ and the ideal $H$ is not just infinite.

Consider the last claim.
As an image of $\JJ$, $\KK$ is weakly special as well.
Also, $\KK$  is a Kantor double of the Poisson superalgebra $P(\QQ)=\langle 1\rangle \oplus \QQ $ which is not Lie nilpotent of class 2,
hence, $\KK$ is not special by~\cite{Shestakov93}.
\end{proof}

In particular, we get a Jordan superalgebra $\KK$ which Gelfand-Kirillov dimension belongs to $(1,2)$,
that is not possible for associative and Jordan algebras~\cite{KraLen,MaZe96}.
A more general fact that the gap $(1,2)$ does not exist for {\it Jordan super}algebras is proved in~\cite{PeSh18Jslow}.

\section{Nillity of superalgebras $\QQ$, $\AA$, $\PP$, $\JJ$, and $\KK$}

In this section, we establish different statements on nillity of our five superalgebras.

First, we prove that $\QQ$, $\AA$, and superalgebras without unit $\PP^o$, $\JJ^o$, $\KK^o$
are direct sums of two locally nilpotent subalgebras
and there are continuum such different decompositions (Theorem~\ref{Ttwosums}).
Second, $\QQ$ is ad-nil for $\Z_2$-homogeneous elements (Theorem~\ref{Tadnilpotente}).
Third, in case $\ch K=2$, the restricted Lie algebra $\QQ=\Lie_p(v_0,v_1,v_2)$ has a nil $p$-mapping.
Proofs of the last two facts are omitted because they are the same as that in supplied references.
We start with a technical fact.

\begin{Lemma}\label{Ldense}
Let $\lambda,\mu,\bar\mu$ be the real and complex roots of the polynomial $t^3-t-2$. Then
\begin{enumerate}
\item $\mu^n\notin\R$ for any $n\ge 1$.
\item The set $\{\arg(\mu^n)\mid n\ge 1\}$ is dense on $[0,2\pi]$.
\end{enumerate}
\end{Lemma}
\begin{proof}
Consider the field extension $\Q\subset\Q[\lambda,\mu]$.
Since the Galois group has the conjugation, this is an extension of degree 6 and the Galois group is $S_3$.
Assume that $\mu^n\in\R$ for some $n\ge 1$.
By Viet's formulas,
$\lambda\mu\bar\mu=2$.
Denote $\xi=\mu^2\lambda/2$, then $|\xi|=1$.
We obtain $\xi^{n}=\mu^{2n}\lambda^n/2^n\in\R^+$ and $|\xi^{n}|=1$.
Hence, $\xi^{n}=1$, we have a root of unity such that  $\xi\in\Q[\lambda,\mu]$.
Moreover, we can assume that $\xi$ is primitive of degree $n$.
Let $n=\prod_p p^{n_p}$, then by Euler's formula,
$|\Q[\xi]:\Q|=\phi(n)=\prod_p (p-1)p^{n_p-1}$.
Since the Galois group of a cyclotomic extension is abelian, $\phi(n)$ properly divides 6.
Clearly, $p\in\{2,3\}$ and $n\in\{2,3,4\}$.
We have $\mu^3=\mu+2\notin \Q$.
For two remaining cases observe that $\R\cap \Q[\lambda,\mu]=\Q[\lambda]$.
We have either $\mu^2\in \Q[\lambda]$, or $\mu^4=\mu\mu^3=\mu(\mu+2)= \mu^2+2\mu\in\Q[\lambda]$,
in both cases, $\mu$ satisfies a polynomial of degree 2 over $\Q[\lambda]$.
On the other hand, $\mu$ satisfies the following irreducible polynomial of degree 2:
$ h(t)=(t-\mu)(t-\bar\mu)=t^2-(\mu+\bar \mu)t+\mu\bar\mu=t^2+\lambda t+2/\lambda \in \Q[\lambda][t].$
A contradiction proves the first claim.

By the first claim, $2\pi/\arg(\mu)\notin\Q$, we obtain an {\em irrational rotation} of the unit circle,
the classical example of ergodic theory.
Ergodic theory says that an orbit of an irrational rotation of a circle is dense.
\end{proof}

Let $\PP^o\subset \PP$, $\JJ^o\subset \JJ$, and $\KK^o\subset\KK$
be the respective Poisson and Jordan superalgebras without unit.
\begin{Theorem} \label{Ttwosums}
Consider the Lie superalgebra $\QQ=\Lie(v_0,v_1,v_2)$,
its associative hull $\AA=\Alg(v_0,v_1,v_2)$,
the Poisson superalgebra without unit $\PP^o$, and
the Jordan superalgebras without unit $\JJ^o$ and $\KK^o$.
\begin{enumerate}
\item
there exist decompositions into direct sums of two locally nilpotent subalgebras:
$$\QQ=\QQ_+\oplus\QQ_-,\qquad \AA=\AA_+\oplus\AA_-,\qquad
\PP^o=\PP_+\oplus\PP_-,\qquad \JJ^o=\JJ_+\oplus\JJ_-,\qquad \KK^o=\KK_+\oplus\KK_-. $$
\item there are continuum such different decompositions.
\end{enumerate}
\end{Theorem}
\begin{proof}
First, we consider the Lie superalgebra $\QQ$.
Consider a plane $\Pi$ passing through the axis $OY_1$,
it is determined by an equation $\a Y_2+\b  Y_3=0$ in the twisted coordinates, where $\a,\b\in \R$ are some constants.
By Lemma~\ref{Laxis}, the axis $OY_1$ does not contain lattice points $\Z^3$, except the origin.
By rotation of the plane $\Pi$ around $OY_1$, we obtain a continuum of planes
that intersect $\Z^3$ only at origin, because the number of points of the lattice is countable.
Fix such a plane  $\Pi$.
Let $\QQ_+$, $\QQ_-$ be sums of homogeneous components of $\QQ$ that lie on different sides of $\Pi$.
By construction, we get a vector space decomposition $\QQ=\QQ_+\oplus\QQ_-$.
Additivity of the multidegree implies that $\QQ_+$ and $\QQ_-$ are subalgebras.
The plane $\Pi$ splits the "paraboloid" (Theorem~\ref{Tparab}) into two halves, see Figure~\ref{Fig1}.
Now the same geometric arguments as in~\cite{PeShZe10} prove that the subalgebras $\QQ_+$, $\QQ_-$ are locally nilpotent.
Two such different planes yield different decompositions.
Indeed, consider all pivot elements $\{v_k\mid k\ge 0\}$,
and their weight and twisted coordinates
$$\swt(v_k)=\mu^k=Z_2(k)=Y_2(k)+iY_3(k),\quad  k\ge 0.$$
Since the set of their arguments is dense on $[0,2\pi]$ (Lemma~\ref{Ldense}),
the decompositions determined by two different planes differ by (infinitely many) pivot elements.
Similarly, we get decompositions for $\AA$ and $\PP^o$
because their monomials are inside another paraboloid (Theorem~\ref{TparabP}).

Finally, consider the Jordan superalgebra without unit $\JJ^o$.
We use $\Z^3$-grading of $\JJ$ by multidegree in $\{V_0,V_1,V_2\}$ only.
Let $J$ be a span of all monomials $u, \bar u$, where $u$ is a basis monomial of $\PP$ such that $u\ne 1$.
All such monomials belong to lattice points $\Z^3$ distinct from the origin.
As above, using continuum different appropriate planes passing through $OY_1$,
we split monomials into two parts
and get decompositions into direct sums of two locally nilpotent subalgebras $J=J_+\oplus J_-$.
Since $\bar 1$ is at the origin,
a multiplication by $\bar 1$ keeps the lattice points, thus,
$\bar 1\bullet J_+\subset J_+$ and $\bar 1\bullet  J_-\subset J_-$.
By our construction, $\JJ^o=\langle \bar 1\rangle_K\oplus J$.
Put $\JJ_-=J_-$ and $\JJ_+=J_+\oplus \langle \bar 1\rangle_K$.
Then $\JJ^o=\JJ_+\oplus\JJ_-$. We have $(\JJ_+)^2\subset J_+$,
and $\JJ_+$ is a locally nilpotent subalgebra as well.
\end{proof}

\begin{Theorem}\label{Tadnilpotente}
Consider the Lie superalgebra $\QQ=\Lie(v_0,v_1,v_2)=\QQ_{\bar{0}}\oplus\QQ_{\bar{1}}$. For any $a\in\QQ_{\bar{n}}$,
$\bar{n}\in\lbrace\bar{0},\bar{1}\rbrace$, the operator $\ad(a)$ is nilpotent.
\end{Theorem}
\begin{proof} The same as~\cite[Theorem 10.1]{Pe16} or~\cite[Theorem 12.1]{PeOtto}.
\end{proof}
\begin{Corollary}
For any $a\in \QQ_{n_1, n_2,n_3}$, where $n_1,n_2,n_3\ge 0$, the the operator $\ad(a)$ is nilpotent.
\end{Corollary}

Recall that in case $\ch K=2$ the Lie superalgebra $\QQ=\Lie(v_0,v_1,v_2)$
coincides with the restricted Lie algebra generated by the same elements, i.e. $\QQ=\Lie_p(v_0,v_1,v_2)$ (Corollary~\ref{Cbasis2}).
\begin{Theorem}\label{TnilQ2}
Let $\ch K=2$.
The restricted Lie algebra $\QQ=\Lie_p(v_0,v_1,v_2)$ has a nil $p$-mapping.
\end{Theorem}
\begin{proof}
The same as in~\cite[Proposition 1]{ShZe08}.
The ideas of that proof were further developed in~\cite[Corollary 2.9]{Bartholdi15}
and \cite[Theorem 8.6]{Pe17}.
\end{proof}



\begin{thebibliography}{99}
\bibitem{BMPZ}
    Bahturin~Yu.A., Mikhalev~A.A., Petrogradsky~V.M., and Zaicev M.~V.,
    {Infinite dimensional Lie superalgebras},
    {de Gruyter Exp. Math.} vol.~{7},
    de Gruyter, Berlin, 1992.
\bibitem{BaOl07}
   Bahturin, Yu.A.; Olshanskii, A.,
   Large restricted Lie algebras,
   {\it J. Algebra} (2007) {\bf 310}, No.~1, 413--427.
\bibitem{BaSeZa01}
   Bahturin, Yu.A.; Sehgal, S.K.; Zaicev, M.V.,
   Group gradings on associative algebras,
   {\it J. Algebra} (2001) {\bf 241}, No.~2, 677--698.
\bibitem{BaoYeZhang17}
   Bao, Y., Ye Y.,; Zhang, J.,
   Restricted Poisson algebras.
   {\it Pac. J. Math.} {\bf 289}, No. 1, 1-34 (2017).
\bibitem{Bartholdi06}
    Bartholdi L.,
    Branch rings, thinned rings, tree enveloping rings.
    {\it Israel J. Math.} {\bf 154} (2006), 93--139.
\bibitem{Bartholdi15}
    Bartholdi L.,
    Self-similar Lie algebras.
    {\it J. Eur. Math. Soc. (JEMS)} {\bf 17} (2015), no. 12, 3113--3151.
\bibitem{BaGr00}
    Bartholdi L., Grigorchuk R.I.,
    Lie methods in growth of groups and groups of finite width.
    Computational and geometric aspects of modern algebra,
    1--27, London Math. Soc. Lecture Note Ser., 275,
    Cambridge Univ. Press, Cambridge, 2000.
\bibitem{Ber67}
  Berezin, F. A.,
  {Several remarks on the associative hull of a Lie algebra}. (Russian) {\it Funkcional. Anal. i Prilo\v zen} {\bf 1} (1967), no. 2, 1--14.
\bibitem{BezKal07}
  Bezrukavnikov, R.; Kaledin, D.
  Fedosov quantization in positive characteristic.
  {\it J. Am. Math. Soc.} {\bf 21}  (2008). No. 2, 409--438.
\bibitem{Bou-Leites-09}
    Bouarroudj, S., Grozman, P., Leites, D.
    Classification of finite dimensional modular Lie superalgebras with indecomposable Cartan matrix,
    {\it SIGMA} {\bf 5} (2009), 060, 1--63.
\bibitem{CarMatNew97}
    Caranti, A.; Mattarei, S.; Newman, M.F.
     Graded Lie algebras of maximal class.
    {\it Trans. Am. Math. Soc.} {\bf 349} (1997) No.~10, 4021--4051.
\bibitem{CarNew00}
    Caranti, A.; Newman, M.F.
    Graded Lie algebras of maximal class. II.
    {\it J. Algebra} {\bf 229} (2000), No.~2, 750--784.
\bibitem{PeOtto}
   de Morais Costa O.A., Petrogradsky\,V.,
   Fractal just infinite nil Lie superalgebra of finite width, {\it J. Algebra}, {\bf 504}, (2018), 291--335.
\bibitem{Dixmier}
  Dixmier J., Enveloping algebras. AMS, Rhode Island, 1996.
\bibitem{DreHam04}
    Drensky~V. and Hammoudi~L.,
    Combinatorics of words and semigroup algebras which are sums of
    locally nilpotent subalgebras.
    {\it Canad. Math. Bull.} {\bf 47} (2004), no.~3, 343--353.
\bibitem{Eld10}
    Elduque, A.
    Fine gradings on simple classical Lie algebras.
    {\it J. Algebra} {\bf 324}  (2010), No. 12, 3532--3571.
\bibitem{Ershov12}
    Ershov M.,
    Golod-Shafarevich groups: a survey,
    {\it Int. J. Algebra Comput}. {\bf 22}, (2012) No.~5, Article ID 1230001.
\bibitem{FabGup85}
    Fabrykowski, J., Gupta, N.,
    On groups with sub-exponential growth functions.
    {\it J. Indian Math. Soc. (N.S.)} {\bf 49} (1985), no. 3-4, 249--256.
\bibitem{Farkas98}
   Farkas D. R., Poisson polynomial identities.
   {\it Comm. Algebra} {\bf 26} (1998), no. 2, 401--416.
\bibitem{Farkas99}
   Farkas D. R., Poisson polynomial identities. II.
   {\it Arch. Math. (Basel)} {\bf 72} (1999), no. 4, 252--260.
\bibitem{FutKochSis}
    Futorny F.,  Kochloukova D.H.,  Sidki S.N.,
    On Self-Similar Lie Algebras and Virtual Endomorphisms,
    arXiv:1801.03005.
\bibitem{Golod64}
    Golod, E.S.
    On nil-algebras and finitely approximable p-groups.
    {\it Am. Math. Soc., Translat., II. Ser.} {\bf 48}, 103--106 (1965);
    translation from {\it Izv. Akad. Nauk SSSR, Ser. Mat}. {\bf 28}, 273--276 (1964).
\bibitem{Golod69}
    Golod, E.S.
    On some problems of Burnside type. 
    {\it Am. Math. Soc., Translat., II. Ser.} {\bf 84}, (1969) 83--88;
    translation from Tr. Mezdunarod. Kongr. Mat., Moskva 1966, 284--289 (1968).
\bibitem{Grigorchuk80}
    Grigorchuk,~R.I.,
    On the Burnside problem for periodic groups.,
    {\it Funktsional. Anal. i Prilozhen}. {\bf 14} (1980), no. 1, 53--54.
\bibitem{Grigorchuk84}
    Grigorchuk, R.I.
    Degrees of growth of finitely generated groups, and the theory of invariant means.
    {\it Math. USSR, Izv}. {\bf 25} (1985), 259--300;
    translation from {\it Izv. Akad. Nauk SSSR, Ser. Mat}. {\bf 48} (1984), No.5, 939--985.
\bibitem{Grigorchuk00horizons}
    Grigorchuk,~R.I.,
    Just infinite branch groups. New horizons in pro-$p$ groups, 121--179,
    Progr. Math., 184, Birkhauser Boston, Boston, MA, 2000.
\bibitem{GuptaSidki83}
    Gupta~N., and Sidki~S.,
    On the Burnside problem for periodic groups.,
    {\it Math. Z}. {\bf 182} (1983), no. 3, 385--388.
\bibitem{GuptaSidki83A}
    Gupta, N.; Sidki, S.
    Some infinite p-groups.
    {\it Algebra Logic} {\bf 22},  (1983) 421--424; translation from: {\it Algebra Logika} {\bf 22} (1983) No. 5, 584--589.
\bibitem{JacLie}
    Jacobson~N.,
    {Lie algebras},
    Interscience, New York. 1962.
\bibitem{Kac77}
   Kac, V.G.
   Lie superalgebras.
   {\it Adv. Math.} {\bf 26}, (1977), 8--96.
\bibitem{Kac77CA}
   Kac, V.G.
   Classification of simple $\mathbb Z$-graded Lie superalgebras and simple Jordan superalgebras. 
   {\it Comm. Algebra} {\bf 5}, (1977) 1375--1400.     
\bibitem{KacMarZel01}
   Kac, V.G., Martinez, C.; Zelmanov, E.
   Graded simple Jordan superalgebras of growth one.
   {\it Mem. Am. Math. Soc}. {\bf 711}, 140p. (2001).
\bibitem{Kantor92}
    Kantor I.L.,
    Jordan and Lie superalgebras determined by a Poisson algebra.
    Aleksandrov, I. A. (ed.) et al., Second Siberian winter school “Algebra and Analysis”.
    Proceedings of the second Siberian school, Tomsk State University, Tomsk, Russia, 1989.
    Transl. AMS Transl., Ser. 2, Am. Math. Soc. 151, 55--80 (1992).
\bibitem{Kel93}
    Kelarev~A.V.,
    A sum of two locally nilpotent rings may be not nil.
    {\it Arch. Math}. {\bf 60} (1993), no.~5, 431--435.
\bibitem{KingMcCrimon92}
   King D.; McCrimmon K.
   The Kantor construction of Jordan superalgebras.
   {\it Commun. Algebra} {\bf 20} (1992)., No.1, 109--126.
\bibitem{KraLen}
   Krause~G.R. and Lenagan~T.H.,
   {Growth of algebras and Gelfand-Kirillov dimension},
   AMS, Providence, R.I., 2000.
\bibitem{Kry11}
    Krylyouk Ia.,
    The enveloping algebra of the Petrogradsky-Shestakov-Zelmanov
    algebra is not graded-nil in the critical characteristics,
    {\it J. Lie Theory}, {\bf 21}, (2011), No. 3, 703--709.
\bibitem{LenSmo07}
   Lenagan, T.H.,  Smoktunowicz  Agata.,
   An infinite dimensional affine nil algebra with finite Gelfand-Kirillov dimension,
   {\it J. Am. Math. Soc.} {\bf 20}, (2007) No. 4, 989--1001.
\bibitem{MakShe12}
  Makar-Limanov, L., Shestakov, I.,
  Polynomial and Poisson dependence in free Poisson algebras and free Poisson fields.
  {\it J. Algebra} {\bf 349}, (2012) No. 1, 372--379.
\bibitem{MakUmi11}
  Makar-Limanov, L., Umirbaev, U.,
  The Freiheitssatz for Poisson algebras.
  {\it J. Algebra} {\bf 328}, (2011) No. 1, 495--503.
\bibitem{MaShZe01}
    Martinez, C., Shestakov, I., Zelmanov, E.,
    Jordan superalgebras defined by brackets.
    {\it J. Lond. Math. Soc., II. Ser.} {\it 64}, (2001)  No. 2, 357--368.
\bibitem{MaZe96}
   Martinez C., Zelmanov E.,
   Jordan algebras of Gelfand-Kirillov dimension one.
   {\it J. Algebra} {\bf 180},(1996). No.1, 211--238.
\bibitem{MaZe99}
   Martinez C., Zelmanov E.,
   Nil algebras and unipotent groups of finite width.
   {\it Adv. Math.} {\bf 147}, (1999) No.2, 328--344.
\bibitem{Mikh88}
   Mikhalev A.A.,
   Subalgebras of free Lie $p$-superalgebras.
   {\it Math. Notes} {\bf 43}, (1988) No.2, 99--106;
   translation from {\it Mat. Zametki} {\bf 43} (1988) No.2, 178--191.
\bibitem{Mil}
   Millionschikov D.V.,
   Naturally graded Lie algebras (Carnot algebras) of slow growth.
   arXiv:1705.07494.
\bibitem{MiPeRe}
   Mishchenko~S.P., Petrogradsky~V.M. and Regev~A.,
   Poisson PI algebras, {\it Trans. Amer. Math. Soc.}, {\bf 359}, (2007), no.\,10, 4669--4694.
\bibitem{Nekr05}
   Nekrashevych, V.,
   Self-similar groups.
   Mathematical Surveys and Monographs {\bf 117}.
   Providence, RI: American Mathematical Society (AMS) (2005).
\bibitem{Pape01}
   Passman~D.S. and Petrogradsky~V.M.,
   Polycyclic restricted Lie algebras,
   {\it Comm. Algebra}, {\bf 29} (2001), no.~9, 3829--3838.
\bibitem{Pe97}
   Petrogradsky~V.M.,
   On Lie algebras with nonintegral $q$-dimensions.
   {\it Proc. Amer. Math. Soc}. {\bf 125} (1997), no.\,3, 649--656.
\bibitem{Pe06}
   Petrogradsky~V.M.,
   Examples of self-iterating Lie algebras,
   {\it J. Algebra}, {\bf 302} (2006), no.\,2, 881--886.
\bibitem{Pe16}
   Petrogradsky~V.,
   Fractal nil graded Lie superalgebras,
   {\it J. Algebra}, {\bf 466} (2016), 229--283.
\bibitem{Pe17}
   Petrogradsky V.,
   Nil Lie $p$-algebras of slow growth,
   {\it Comm. Algebra.} {\bf 45}, (2017), no.~7, 2912--2941.
\bibitem{PeRaSh}
   Petrogradsky~V. M., Yu. P.~Razmyslov, and E. O.~Shishkin,
   Wreath products and Kaluzhnin-Krasner embedding for Lie algebras,
   {\it Proc. Amer. Math. Soc.}, {\bf 135}, (2007), 625--636.
\bibitem{PeSh09}
   Petrogradsky\,V.M. and Shestakov\,I.P.
   Examples of self-iterating Lie algebras, 2,
   {\it J. Lie Theory}, {\bf 19} (2009), no.\,4, 697--724.
\bibitem{PeSh13ass}
   Petrogradsky\,V.M. and Shestakov\,I.P.
   Self-similar associative algebras, {\it J. Algebra}, {\bf 390} (2013), 100--125.
\bibitem{PeSh13fib}
   Petrogradsky\,V.M. and Shestakov\,I.P.
   On properties of Fibonacci restricted Lie algebra, {\it J. Lie Theory}, {\bf 23} (2013), no. 2, 407--431.
\bibitem{PeShZe10}
   Petrogradsky\,V.M., Shestakov\,I.P., and Zelmanov E.,
   Nil graded self-similar algebras,
   {\it Groups Geom. Dyn.}, {\bf 4} (2010), no.\,4, 873--900.
\bibitem{PeSh18Jslow}
   Petrogradsky\,V., and Shestakov I.P.,
   On Jordan doubles of slow growth of Lie superalgebras, preprint.
\bibitem{Rad86}
   Radford~D. E.,
   Divided power structures on Hopf algebras and embedding
   Lie algebras into special-derivation algebras,
   {\it J. Algebra}, {\bf 98} (1986), 143--170.
\bibitem{Ratseev13}
   Ratseev, S.M.,
   Poisson algebras of polynomial growth,
  {\it Sib. Math. J.} {\bf 54}, (2013) No. 3, 555--565,
  translation from {\it Sib. Mat. Zh.} {\bf 54}, (2013) No. 3, 700--711. 
\bibitem{Razmyslov}
   Razmyslov Yu.P., {\it Identities of algebras and their representations}.
   AMS, Providence RI 1994.
\bibitem{Rozh96}
   Rozhkov, A.V.
   Lower central series of a group of tree automorphisms,
   {\it Math. Notes} {\bf 60}, No.2, 165--174 (1996);
   translation from
   {\it Mat. Zametki} {\bf 60}, No.2, 225--237 (1996).
\bibitem{Scheunert}
   Scheunert, M.
   The theory of Lie superalgebras.
   Lecture Notes in Mathematics. 716. Berlin. Springer-Verlag. 1979.
\bibitem{Shestakov93}
  Shestakov I.P.,
  Quantization of Poisson superalgebras and speciality of Jordan Poisson superalgebras.
  {\it Algebra i Logika} {\bf 32} (1993), no.~5, 571--584;
  English translation:
  {\it Algebra and Logic} {\bf 32} (1993), no.~5, 309--317.
\bibitem{She99}
  Shestakov, I.P.
  Alternative and Jordan superalgebras.
  {\it Sib. Adv. Math}. {\bf 9}, (1999). No.2, 83--99.
\bibitem{ShZe08}
   Shestakov I.P. and Zelmanov E.,
   Some examples of nil Lie algebras.
   {\it J. Eur. Math. Soc. (JEMS)}  {\bf 10} (2008),  no. 2, 391--398.
\bibitem{SheUmi04}
  Shestakov, I.; Umirbaev, U.U.,
  The tame and the wild automorphisms of polynomial rings in three variables.
  {\it J. Am. Math. Soc.} {\bf 17}, (2004) No. 1, 197--227.
\bibitem{Sidki97}
   Sidki S.N.,
   A primitive ring associated to a Burnside $3$-group.
   {\it J. London Math. Soc.} (2) {\bf 55} (1997), no.~1, 55--64.
\bibitem{Sidki09}
   Sidki S.N.,
   Functionally recursive rings of matrices --- Two examples.
   {\it J.Algebra} {\bf 322} (2009), no.~12, 4408--4429.
\bibitem{Sko94}
   Skosyrskij, V.G.
   Prime Jordan algebras and the Kantor construction.
   {\it Algebra Logic} {\bf 33},  (1994) No.3, 169--179; translation from {\it Algebra Logika} {\bf 33},  (1994) No.3, 301--316.
\bibitem{Strade1}
    Strade H.,
    Simple Lie algebras over fields of positive characteristic. I: Structure theory.
    Berlin: de Gruyter. 2004.
\bibitem{StrFar}
   Strade~H. and  Farnsteiner~R.,
   {Modular Lie algebras and their representations},
   New York etc.: Marcel Dekker, 1988.
\bibitem{Ver69}
   Vergne M., {La structure de Poisson sur l'alg\`ebre sym\'etrique d'une alg\`ebre de Lie nilpotente}.
   {\it C. R. Acad. Sc. Paris} {\bf 269} (1969), S\'er. A-B, 950--952.
\bibitem{Voden09}
    Voden T.,
    Subalgebras of Golod-Shafarevich algebras,
    {\it Int. J. Algebra Comput}. {\bf 19}, (2009) No.~3, 423--442.
\bibitem{Zelmanov07}
    Zelmanov E.,
    Some open problems in the theory of infinite dimensional algebras.,
    {\it J. Korean Math. Soc.} {\bf 44} (2007), no. 5, 1185--1195.
\bibitem{Zelmanov}
    Zelmanov E., A private communication.
\bibitem{ZhePan17}
    Zhelyabin, V.N., Panasenko, A.S., Herstein’s construction for just infinite superalgebras,
    {\it  Siberian El. Math. Reports}
    (2017) {\bf 14}, 1317--1323.
\end{thebibliography}
\end{document}